\documentclass[11pt, oneside]{article} 
\usepackage[margin=2cm]{geometry}                		
\usepackage[parfill]{parskip}    		
\usepackage{graphicx}			
\usepackage[utf8]{inputenc}
\usepackage[english]{babel} 
\usepackage{amsmath,amssymb,amsthm,bm} 
\usepackage[margin=2cm]{geometry}
\usepackage[color=yellow]{todonotes}
\usepackage{mathrsfs}
\usepackage{url}	
\usepackage{esint}
\usepackage[colorlinks]{hyperref}
\usepackage{enumerate}
\usepackage{outlines}[enumerate]
\usepackage{framed,comment,enumerate}
\usepackage{upgreek}
\usepackage{pgfplots}
\usepackage{float}
\usepackage{tikz,tikz-cd}
\usepackage{rotating}
\usepackage{graphicx}
\usepackage{caption}
\usepackage{multicol}
\usepackage{cancel}
\usepackage{subcaption}
\usepackage[title]{appendix}
\usepackage{tikz-cd}
\usepackage{pgfgantt}
\usepackage{rotating}
\extrafloats{100}
\numberwithin{equation}{section}

\newtheorem{theorem}{Theorem}[section]
\newtheorem{lemma}{Lemma}[section]
\newtheorem{corollary}[lemma]{Corollary}
\newtheorem{proposition}[lemma]{Proposition}
\newtheorem{remark}{Remark}[section]
\newtheorem{assumption}{Assumption}[section]

\theoremstyle{definition}

\DeclareFontFamily{U}{MnSymbolC}{}
\DeclareSymbolFont{MnSyC}{U}{MnSymbolC}{m}{n}
\DeclareFontShape{U}{MnSymbolC}{m}{n}{
    <-6>  MnSymbolC5
   <6-7>  MnSymbolC6
   <7-8>  MnSymbolC7
   <8-9>  MnSymbolC8
   <9-10> MnSymbolC9
  <10-12> MnSymbolC10
  <12->   MnSymbolC12}{}
\DeclareMathSymbol{\intprod}{\mathbin}{MnSyC}{'270}

\newcommand{\ob}[1]{\overline{#1}}
\newcommand{\cev}[1]{\overleftarrow{#1}}
\newcommand{\wt}[1]{\widetilde{#1}}
\newcommand{\wh}[1]{\widehat{#1}}

\newcommand{\mc}[1]{\mathcal{#1}}

\newcommand{\bs}[1]{\boldsymbol{#1}}

\newcommand{\mcal}[1]{\mc{#1}}
\newcommand{\scp}[2]{{\left\langle {#1}\, , \, {#2}\right\rangle}}

\def\p{{\partial}}

\def\rmd{{\rm d}}

\def\eps{{\epsilon}}

\def\p{\partial}

\pgfplotsset{compat=1.16}

\title{Rough variational principles and applications to adjoint systems}
\author{Ruiao Hu\footnote{Corresponding author. Email: ruh030@ucsd.edu} ~and Melvin Leok\footnote{Email: mleok@ucsd.edu} \vspace{0.4cm}\\
Department of Mathematics, University of California San Diego,\\  9500 Gilman Drive, La Jolla, CA 92093-0112, USA}
\date{}

\begin{document}
\maketitle

\begin{abstract}
We consider a Type-II variational principle driven by geometric rough path with split boundary conditions naturally suited to adjoint systems. From this rough variational principle we derive rough Hamilton's equations, establish their pathwise conservation laws and associated Hamilton–Jacobi equation. We then specialise the framework to rough adjoint systems, obtaining pathwise conservation and quasi-conservation laws that underpin adjoint sensitivity analysis with respect to initial conditions and parameters. On the discrete side, we construct a rough Galerkin discretisation of the rough Type-II variational principle and show that it generates a symplectic flow with discrete analogues of the continuous conservation laws. We establish its equivalence to a class of Rough Symplectic Partitioned Runge–Kutta (RSPRK) methods and analyse its convergence and naturality properties. Lastly, we perform numerical experiments to validate the predicted convergence rates and demonstrate that RSPRK methods preserve the adjoint conservation laws to machine precision, yielding more accurate and stable gradients in optimisation problems than non-symplectic alternatives.
\end{abstract}

\tableofcontents
\section{Introduction}
Dynamical systems driven by irregular and random signals naturally appear in a wide range of applications including multi-scale modelling and machine learning where the primary role of these signals is to model unknown or under-resolved quantities. In these applications, one is simultaneously interested in simulating the state dynamics as well as computing accurate sensitivities of objective functionals with respect to initial conditions and parameters for optimisation, uncertainty quantification, and data assimilation. In the deterministic setting, adjoint sensitivity analysis is an efficient method for computing these gradients with respect to a large number of parameters for a small number of objective functionals. Deterministic adjoint sensitivity has seen applications in sensitivity analysis \cite{Cacuci1981, Cao2003}, geophysics \cite{Plessix2006}, aerospace system design \cite{giles2000introduction}, Monte Carlo methods \cite{Caflisch2021} and more recently, machine learning as a part of neural Ordinary Differential Equations (ODEs) \cite{Chen2019, matsubara2021symplectic}. Geometrically, deterministic adjoint systems have been studied in a variational framework through formal Hamiltonians \cite{Ibragimov2007, TL2024}. The Hamiltonian structures of adjoint systems strongly suggest the use of symplectic methods for accurate computation of gradients \cite{Sanz-Serna2016}, and have motivated several discrete-then-optimise methodologies for adjoint systems appearing in the literature.

Extending the study of adjoint systems to the case where the state dynamics are governed by a Stochastic Differential Equation (SDE) is subtle due to the forward-backward dynamics of the state variable and its adjoint. When the objective function is defined as the expectation over the probability space with respect to which the forward state SDE is defined, one may use the well-established theory of stochastic optimal control \cite{Yong1999}. The general theory yields a pair of Forward Backward SDEs (FBSDEs) where the forward state SDE and the backward adjoint SDE are defined with respect to the same forward filtration. However, the computational costs of simulating FBSDEs are typically high due to the iterative process of simulating the adapted backward SDE in general. 

In settings where the adaptedness of the solution is not essential, one may instead use the rough path interpretation of integration against Brownian motion, which accommodates non-adapted integrands and coincides with stochastic integration for adapted integrands \cite{FV10}. Rough path theory, pioneered in a series of papers starting with \cite{lyons1994differential, Lyons1998}, provides a robust framework for integration against rough, non-differentiable signals and treats the resulting dynamical systems as Rough Differential Equations (RDEs) \cite{LCL2007, FH2020}. Treating SDEs pathwise as RDEs is particularly useful for sample-wise optimisation and has contributed to the widespread adoption of pathwise adjoints for neural SDEs \cite{Li2020, Kidger2021, Kidger2022} due to their similarity to the deterministic case. We also remark that non-adapted gradient approaches have been developed outside of the rough path theory, mainly through the use of Malliavin calculus \cite{Fourni1999,Gobet2005, Malliavin2006}.

The goal of this paper is to bring the geometric description and structure preserving discretisation of adjoint systems to the rough path driven case through the variational characterisation of adjoint systems. To this end, the present paper builds on work in stochastic \cite{Holm2015, SC2019} and rough \cite{CHLN2022} variational principles; geometric properties of deterministic adjoint systems \cite{TL2024}, the deterministic and stochastic discrete Hamiltonian variational discretisations \cite{LZ2011, HT2018, HHW2018}, and geometric integration of adjoint systems  \cite{Sanz-Serna2016, TL2024}.

\paragraph{Main contributions.}
The main contributions of this work are twofold. On the continuous side, we prove that rough adjoint systems arise as critical points of a Type-II variational principle driven by geometric rough paths in Theorem \ref{thm:RVP}. We also show that rough adjoint systems possess the conservation laws underpinning rough adjoint sensitivity analysis in Propositions \ref{prop:canonical one form conservation} -- \ref{prop:aug 1form conservation}. In the discrete setting, we construct rough variational integrators and show that they preserve discrete analogues of the conservation laws for rough adjoint systems in Propositions \ref{prop:discrete 1 form CL} -- \ref{Prop:RSPRK conservation laws}. Furthermore, we prove naturality results for numerical methods of RDEs driven by geometric rough paths in Propositions \ref{prop:tangent general} -- \ref{prop:cotangent general}.

\paragraph{Outline of the paper.}
The remainder of the paper is organised as follows.
\begin{itemize}
    \item In Section \ref{sec:rough VP}, we formulate a rough Type-II variational principle in phase space with natural boundary conditions for adjoint systems, and show that its critical points imply rough Hamilton equations with these boundary conditions. We give the associated conservation laws and rough Hamilton--Jacobi equation. We then specialise this framework to rough adjoint systems by considering particular Hamiltonians and establish pathwise conservation laws that yield adjoint sensitivities with respect to initial conditions and parameters.
    \item In Section \ref{sec:variational discretisation}, we construct a rough Galerkin discretisation of the Type-II variational principle and show that the resulting methods define symplectic flows and they possess discrete analogues of the continuous conservation laws. We further show the equivalence of these methods with Rough Symplectic Partitioned Runge--Kutta (RSPRK) methods and analyse their convergence properties. We also demonstrate the naturality properties of the RSPRK methods.
    \item In Section \ref{sec:examples}, we illustrate both the convergence behaviour of RSPRK methods and the importance of their structure preserving properties for accurate gradient computation in optimisation problems.
    \item In Section \ref{sec:conclusion}, we conclude with a discussion of future directions.
\end{itemize}

\section{Rough variational principles}\label{sec:rough VP}
In this section, we develop a theory of rough path driven variational principle that is capable of incorporating different types of boundary conditions appearing in Hamiltonian boundary value problems in phase space. 
Special attention is given to Type-II boundary conditions where position is prescribed at initial time and momentum is prescribed at terminal time due to its natural occurrence in adjoint systems. We demonstrate pathwise conservation laws for the rough canonical Hamilton's equations before focusing on rough adjoint systems, where their conservation laws are considered in Section \ref{sec:cts rough adjoints} and their applications to adjoint sensitivity analysis are considered in Section \ref{sec:cts rough adjoints sensitivities}. 

The class of driving rough path considered in this section is restricted to geometric rough paths, in the sense made precise in Section \ref{subsec:rough VP}, where the classic example is the Stratonovich lifted Brownian motion \cite[Sec. 3.3]{FH2020}. This restriction is made such that ordinary calculus rules are valid and the resulting conservation laws follow closely to the deterministic counterparts. Furthermore, \cite{HK2015} established the equivalence between geometric rough path and branched rough path, the latter which is constructed using rooted coloured trees such that analysis of numerical methods follows closely to the deterministic setting. To the best knowledge of the authors, rough variational principles involving non-geometric rough paths have not appeared in the literature before. We will not include a review of the rough path theory, instead recall concepts when required. For definitive texts on rough paths, see e.g., \cite{FH2020, FV10, LCL2007}. 

\subsection{Type-II rough variational principles}\label{subsec:rough VP}
Let $\alpha \in \left(\frac{1}{3}, \frac{1}{2}\right]$ and let $\mcal{C}^\alpha_g([t_0, t_1], \mathbb{R}^K)$ be the space of geometric $\alpha$-H\"older $\mathbb{R}^K$ valued rough paths. For an arbitrary $\mathbf{Z} = (Z,\mathbb{Z}) \in \mcal{C}^\alpha_g([t_0, t_1],\mathbb{R}^K)$, $Z: [t_0,t_1] \rightarrow \mathbb{R}^K$ is the path itself and $\mathbb{Z}: [t_0,t_1]^2 \rightarrow \mathbb{R}^K\otimes \mathbb{R}^K$ is the second order signature that is prescribed and satisfies Chen's relation $\mathbb{Z}_{s,t} = \mathbb{Z}_{s,u} + \mathbb{Z}_{u,t} + Z_{s,u}\otimes Z_{u,t}$, and the symmetry condition $\mathbb{Z}_{s,t} + \mathbb{Z}_{s,t}^* = Z_{s,t}\otimes Z_{s,t}$ for all $s,t$. This symmetry condition gives chain rule from ordinary calculus when working with geometric rough path. 

Let $Q$ be a $N$-dimensional configuration manifold and $T^*Q$ be the cotangent bundle of $Q$ with canonical coordinates. For simplicity, in this work we assume $Q$ is a vector space such that $Q \cong \mathbb{R}^N$ and $T^*Q \cong \mathbb{R}^N \times \mathbb{R}^N$ where we have identified covectors in $\left(\mathbb{R}^N\right)^*$ with vectors in $\mathbb{R}^N$ under the Euclidean pairing. Let $\mcal{D}^{2\alpha}_Z(W)$ be the space of controlled path with respect to $\mathbf{Z}$ taking values in an arbitrary Banach space $W$. Let $q, p \in \mcal{D}^{2\alpha}_Z(\mathbb{R}^N)$, or equivalently, $(q,p)\in \mcal{D}^{2\alpha}_Z(T^*Q)$ be controlled paths. 

Let $H :T^*Q \rightarrow \mathbb{R}$ and $\mcal{H} :T^*Q \rightarrow \mathbb{R}^K$ be Hamiltonian functionals. Here, $H$ is the Hamiltonian whose Hamiltonian vector field is integrated against $t$ and $\mcal{H} = (H_1,\ldots, H_K)^T$ where $H_i :T^*Q \rightarrow \mathbb{R}$ for $i = 1,\ldots K$ are the family of Hamiltonians that is integrated against $\mathbf{Z}$. For the purpose of this work we assume that $H \in C^\infty(T^*Q)$ and $\mcal{H} \in C^\infty(T^*Q, \mathbb{R}^K)$ for simplicity.

For controlled paths $q,p$, consider the following Type-II action functional $\mathfrak{S}: \mcal{D}^{2\alpha}_Z(\mathbb{R}^N) \times \mcal{D}^{2\alpha}_Z( \mathbb{R}^N) \rightarrow \mathbb{R}$:
\begin{align}
    \mathfrak{S}(q(\cdot), p(\cdot)) = \scp{p(t_1)}{q(t_1)} - \int_{t_0}^{t_1}\scp{p(t)}{dq(t)} + \int_{t_0}^{t_1}H(q(t), p(t))\,dt + \int_{t_0}^{t_1}\mcal{H}(q(t), p(t))d\mathbf{Z}_t\,. \label{eq:rough type 2 action}
\end{align}
Here, the integral $\int_{t_0}^{t_1}\scp{p(t)}{d q(t)}$ is well-defined as a rough integral since $q,p$ are controlled by $\mathbf{Z}$ and $p$ is taken as a linear map, $p \in \mcal{D}^{2\alpha}_{Z}(\mcal{L}(\mathbb{R}^N, \mathbb{R}))$. We remark that the integration against $q$ is dependent on the currently unknown Gubinelli derivative of $q$ which are to be determined using the variational principle in Theorem \ref{thm:RVP}, not just the path itself. 
Since the controlled paths $q,p$ are continuous in the compact interval $[t_0, t_1]$ and their codomain are also compact, when we evaluate $H$ and $\mcal{H}$ on $q,p$, we use the same notations, $H$ and $\mcal{H}$, for their compactly supported analogues on the codomains of $q,p$. The composition Lemma \cite[Lem. 7.3]{FH2020} implies that both Hamiltonians, $\mcal{H}(q,p)$ and $H(q,p)$, are controlled by $\mathbf{Z}$ where $\mcal{H}(q,p)$ is taken as an element of $\mcal{D}^{2\alpha}_Z(\mcal{L}(\mathbb{R}^K, \mathbb{R}))$ in the last term of \eqref{eq:rough type 2 action}. In what follows, for notational conveniences, we will use subscript notation $(\cdot)_{t}$ for a path to mean the evaluation of the path at time $t$, e.g., $q_t := q(t)$. Additionally, we will use the notations
\begin{align*}
    \frac{\p H}{\p q_t} := \frac{\p H}{\p q}(q_t, p_t)\,,\quad \frac{\p H}{\p p_t} := \frac{\p H}{\p p}(q_t, p_t)\,.
\end{align*}
Similar notations are made for the partial derivatives of $\mcal{H}$. 

Now let $q,p \in \mcal{D}^{2\alpha}_Z(\mathbb{R}^N)$ be arbitrary controlled paths satisfying boundary conditions $q_{t_0} = a$ and $p_{t_1} = b$ for some given parameters $a, b \in \mathbb{R}^N$. Consider class of variation of $q,p$ parameterised by $\epsilon \in (-1, 1)$ defined as
\begin{align}
    (q^\epsilon, p^\epsilon) := (q + \epsilon \delta q, p + \epsilon \delta p)\,, \label{eq:smooth var}
\end{align}
for arbitrarily chosen smooth path $(\delta q, \delta p) \in C^\infty([t_0, t_1];\mathbb{R}^N)$ satisfying $\delta q(t_0) = \delta p(t_1) = 0$ such that we have $q^\epsilon_{t_0} = a$ and $p^\epsilon_{t_1} = b$ for all $\epsilon \in (-1, 1)$. Since the smooth paths are trivially controlled by $\mathbf{Z}$ with vanishing Gubinelli derivative, the class of perturbed path defined by \eqref{eq:smooth var} are a $\eps$-parameterised family of controlled paths.
\begin{theorem}[Rough Type-II variational principle in phase space]\label{thm:RVP}
    If the controlled path $(q,p)$ is a critical point of the action functional $\mathfrak{S}$ satisfying 
    \begin{align}
        \delta \mathfrak{S}(q(\cdot), p(\cdot)) = \frac{d}{d\epsilon} \bigg|_{\epsilon = 0} \mathfrak{S}(q^\epsilon(\cdot), p^\epsilon(\cdot)) = 0\,, \label{eq:rough type 2 action variations}
    \end{align}
    for all variations defined by \eqref{eq:smooth var}. Then, $(q,p)$ satisfy the rough Hamilton's equations
    \begin{align}
        q_t - q_{t_0} = \int_{t_0}^t \left(\frac{\p H}{\p p_r}\,dr + \frac{\p \mcal{H}}{\p p_r}\, d\mathbf{Z}_r\right) \,, \qquad p_{t_1} - p_t = - \int_t^{t_1} \left(\frac{\p H}{\p q_r}\,dr + \frac{\p \mcal{H}}{\p q_r}\, d\mathbf{Z}_r\right)\,.  \label{eq:rough ham eq}
    \end{align}
    for all $t \in [t_0, t_1]$. Conversely, every controlled path solution $q,p$ of the rough Hamilton's equation \eqref{eq:rough ham eq} is a stationary point of the action \eqref{eq:rough type 2 action} under variations of the form \eqref{eq:smooth var}.
\end{theorem}
\begin{proof}
Noting that $H(q^\eps, p^\eps)$ and $\mcal{H}(q^\eps, p^\eps)$ are controlled paths depending on the parameter $\eps$, interchanging the time integral with $\delta$, we compute the $\epsilon$-derivatives of $\mathfrak{S}$ to have
    \begin{align}
        \begin{split}
            \delta \mathfrak{S}(q(\cdot), p(\cdot)) &= \scp{p_{t_1}}{\delta q_{t_1}} - \int_{t_0}^{t_1} \scp{p_t}{d \delta q_t} - \int_{t_0}^{t_1}\scp{\delta p_t}{d q_t}\\
            & \qquad + \int_{t_0}^{t_1} \scp{\delta q_t}{\frac{\p H}{\p q_t}\,dt + \frac{\p \mcal{H}}{\p q_t}\,d\mathbf{Z}_t } + \int_{t_0}^{t_1} \scp{\delta p_t}{\frac{\p H}{\p p_t}\,dt + \frac{\p \mcal{H}}{\p p_t}\,d\mathbf{Z}_t }\\
            & =  \int_{t_0}^{t_1} \scp{\delta q_t}{d p_t + \frac{\p H}{\p q_t}\,dt + \frac{\p \mcal{H}}{\p q_t}\,d\mathbf{Z}_t } + \int_{t_0}^{t_1} \scp{\delta p_t}{- d q_t + \frac{\p H}{\p p_t}\,dt + \frac{\p \mcal{H}}{\p p_t}\,d\mathbf{Z}_t }\,.
        \end{split}
    \end{align}
    Apply the rough fundamental lemma of calculus of variations \ref{foundamental coro} component-wise to obtain the result. 

    For the converse statement. Substituting in the rough Hamilton's equation \eqref{eq:rough ham eq} into the first variation \eqref{eq:rough type 2 action variations} immediately yields the result.
\end{proof}
\begin{remark}
    Writing out the components, the bulk part of the rough Hamilton's equations are
    \begin{align*}
        &dq^i_t - \frac{\p H}{\p p^i_t}\,dt -\frac{\p H_j}{\p p^i_t}\, d\mathbf{Z}^j_t = 0\,, \quad \forall i = 1,\ldots, N\,,\\
        &d p^i_t + \frac{\p H}{\p q^i_t}\,dt + \frac{\p H_j}{\p q^i_t}\, d\mathbf{Z}^j_t  = 0\,, \quad \forall i = 1,\ldots, N\,,
    \end{align*}
    where Einstein summation is assumed for the $j$-index.
\end{remark}
\begin{remark}
    The choice of imposing initial and terminal conditions on $q$ and $p$ respectively and consequently the class of variations $(\delta q, \delta p)$ such that $\delta q(t_0) = \delta p(t_1) = 0$ can be modified to use different boundary conditions. E.g., Type I boundary condition and Type-III boundary conditions. For Type I variational principle in phase space, we consider arbitrary paths $q,p$ satisfying $q_{t_0} = a$, $q_{t_1} = \wt{a}$ and the variations of $q,p$ parameterised by $\eps$ as smooth paths $(\delta q, \delta p)$ satisfying $\delta q(t_0) = \delta q(t_1) = 0$. The variational principle is expressed as
    \begin{align}
    \mathfrak{S}_I(q(\cdot), p(\cdot)) = \int_{t_0}^{t_1}\scp{p(t)}{dq(t)} - \int_{t_0}^{t_1}H(q(t), p(t))\,dt - \int_{t_0}^{t_1}\mcal{H}(q(t), p(t))d\mathbf{Z}_t\,,
    \end{align}
    which gives the rough Hamilton's equations
    \begin{align*}
        &q_{t} - q_{s} = \int_{s}^{t} \left(\frac{\p H}{\p p_r}\,dr + \frac{\p \mcal{H}}{\p p_r}\, d\mathbf{Z}_r\right) \,, \quad \text{for} \quad q_{t_0} = a\,, q_{t_1} = \wt{a}\,,\\
        &p_t - p_s = -\int_s^t \left(\frac{\p H}{\p q_r}\,dr + \frac{\p \mcal{H}}{\p q_r}\, d\mathbf{Z}_r\right)\,,
    \end{align*}
    for all $s<t \in [t_0, t_1]$ after taking variations and applying \ref{foundamental coro}. 
    For Type-III variational principle in phase space, the paths $q,p$ are chosen to satisfy $q_{t_1} = a$, $p_{t_0} = b$ and the variations satisfying $\delta q(t_1) = \delta p(t_0) = 0$. In this setting, the variational principle is expressed as
    \begin{align}
    \mathfrak{S}_{III}(q(\cdot), p(\cdot)) = \scp{p(t_0)}{q(t_0)} +  \int_{t_0}^{t_1}\scp{p(t)}{dq(t)} - \int_{t_0}^{t_1}H(q(t), p(t))\,dt - \int_{t_0}^{t_1}\mcal{H}(q(t), p(t))d\mathbf{Z}_t\,,
    \end{align}
    which gives the rough Hamilton's equations
    \begin{align*}
        q_{t_1} - q_{t} = \int_{t}^{t_1} \left(\frac{\p H}{\p p_r}\,dr + \frac{\p \mcal{H}}{\p p_r}\, d\mathbf{Z}_r\right) \,, \qquad p_t - p_{t_0} = -\int_{t_0}^t \left(\frac{\p H}{\p q_r}\,dr + \frac{\p \mcal{H}}{\p q_r}\, d\mathbf{Z}_r\right)\,,
    \end{align*}
    for all $t \in [t_0, t_1]$. For adjoint systems and applications to optimisation and optimal control, Type-II boundary conditions is the most natural choice.
\end{remark}
The rough Hamilton's equation \eqref{eq:rough ham eq} arising from Type-II variational principle is in the form of Hamiltonian boundary value problem with split boundary conditions which does not have a general theory of solutions in contrast to initial value problems. One may have no solutions, unique solution or multiple solutions depending on the type of boundary conditions and Hamiltonians \cite{Pontryagin2018}. In the existing literatures, there are systematic studies of the solutions and bifurcation theories of Hamiltonian boundary value problems e.g., in \cite{Eliashberg1998, MO2018}, where the authors proceeded by translating the problem into the intersection of Lagrangian submanifolds \cite{Weinstein1971, Weinstein1973}. For adjoint systems with Type-II boundary conditions, the solution properties simplify which we discuss in Section \ref{sec:cts rough adjoints}.
For the remainder of this section, we will proceed under the following assumption
\begin{assumption}\label{assump: non degen}
    We assume that a non-degenerate solution branch exists on an open set around the Type-II boundary data $(a,b)$ for which $q_{t_0} = a$, $p_{t_1} = b$. Let $\Phi_{t,s}$ be the two parameter rough flow of \eqref{eq:rough ham eq}, the non-degeneracy of the Type-II boundary datum assumption implies that the map $\phi^p$ defined by $\phi^p(p_{t_0}) := \pi_p \Phi_{t_1, t_0}(q_0, p_{t_0})$ is a local $C^\infty$-diffeomorphism for all $p_0$ in some open set around $a$. Here, $\phi^p$ is local to the preceding assumed open set around the boundary data and the time window $[t_0, t_1]$.
\end{assumption}
Let $\wt{q} = \wt{q}(a,b), \wt{p} = \wt{p}(a,b)$ be the solution defined by $(a,b)$. Due to smoothness of the Hamiltonians and Assumption \ref{assump: non degen}, the It\^o--Lyons map is continuously differentiable in the sense of Fr\'echet \cite[Thm. 11.6]{FV10}. Thus, we have that $\wt{q}$ and $\wt{p}$ are differentiable with respect to the boundary data which we will use in the following propositions. In fact, the rough flow generating by the solutions $\wt{q}$ and $\wt{p}$ is a rough flow of local $C^\infty$-diffeomorphism \cite[Thm. 11.14, 11.15]{FV10}.

On the local branch defined in the preceding paragraph, a rough analogue of the Type-II generating function can be found from the functional $\mathfrak{S}$ using Theorem \ref{thm:RVP}. Let $\mcal{S} : \mathbb{R}^N \times \mathbb{R}^N \rightarrow \mathbb{R}$ be a function given by the extremum of the functional $\mathfrak{S}$ defined in \eqref{eq:rough type 2 action} over the family of variations defined in $\eqref{eq:smooth var}$. That is, for arbitrary $a,b \in \mathbb{R}^N$, 
\begin{align}\label{eq:smooth gen func}
\begin{split}
    \mcal{S}(a,b) &:= \underset{\substack{(q^\epsilon,\, p^\epsilon), \\q_{t_0} = a,\, p_{t_1}=b }}{\text{ext}}\mathfrak{S}(q^\epsilon(\cdot), p^\epsilon(\cdot)) = \mathfrak{S}(\wt{q}, \wt{p})\\
    &= \scp{\wt{p}_{t_1}}{\wt{q}_{t_1}} - \int_{t_0}^{t_1}\scp{\wt{p}_t}{d \wt{q}_t} + \int_{t_0}^{t_1}H(\wt{q}_t, \wt{p}_t)\,dt + \int_{t_0}^{t_1}\mcal{H}(\wt{q}_t, \wt{p}_t) d\mathbf{Z}_t \,.
\end{split}
\end{align}
We note that the differentiability results of $\wt{q}, \wt{p}$ implies that $\mcal{S}$ is also $C^\infty$-Fr\'echet. 
\begin{proposition}\label{prop:gen func}
    The function $\mathcal{S}$ is a Type-II generating function of the rough Hamiltonian flow where the mapping $(\wt{q}_{t_0},\, \wt{p}_{t_0}) \rightarrow (\wt{q}_{t_1},\, \wt{p}_{t_1})$ is implicitly given by 
    \begin{align}\label{eq:cont sym mapping}
        \wt{p}_{t_0} = D_1 \mcal{S}(a, b),\, \quad \text{and}\quad \wt{q}_{t_1} = D_2 \mcal{S}(a,b)\,.
    \end{align}
\end{proposition}
\begin{proof}
    We directly compute the derivatives of $\mcal{S}(a,b)$ with respect to the parameters $a,b$ keeping in mind that $\wt{q}_t = \wt{q}_t(a,b)$ and $\wt{p}_t = \wt{p}_t(a,b)$ following from the boundary conditions. 
    \begin{align*}
        D_1 \mcal{S}(a,b) &= \frac{\p \wt{p}_{t_1}}{\p a}\wt{q}_{t_1} + \wt{p}_{t_1}\frac{\p \wt{q}_{t_1}}{\p a} - \int_{t_0}^{t_1}\left[\scp{\frac{\p \wt{p}_t}{\p a}}{d \wt{q}_t} + \scp{\wt{p}_t}{d \frac{\p \wt{q}_t}{\p a}} \right]\\
        & \qquad + \int_{t_0}^{t_1} \left[\frac{\p \wt{p}_t}{\p a} \frac{\p H}{\p \wt{p}_t} + \frac{\p \wt{q}_t}{\p a} \frac{\p H}{\p \wt{q}_t}\right]dt + \int_{t_0}^{t_1} \left[\frac{\p \wt{p}_t}{\p a}\frac{\p \mcal{H}}{\p \wt{p}_t} + \frac{\p \wt{q}_t}{\p a}\frac{\p \mcal{H}}{\p \wt{q}_t}\right]d\mathbf{Z}_t\\
        &= \frac{\p \wt{p}_{t_1}}{\p a}\wt{q}_{t_1} + \wt{p}_{t_0}\frac{\p \wt{q}_{t_0}}{\p a} - \int_{t_0}^{t_1}\left[\scp{\frac{\p \wt{p}_t}{\p a}}{d \wt{q}_t} - \scp{d \wt{p}_t}{\frac{\p \wt{q}_t}{\p a}}\right] \\
        & \qquad + \int_{t_0}^{t_1} \left[\frac{\p \wt{p}_t}{\p a} \frac{\p H}{\p \wt{p}_t} + \frac{\p \wt{q}_t}{\p a} \frac{\p H}{\p \wt{q}_t}\right]dt + \int_{t_0}^{t_1} \left[\frac{\p \wt{p}_t}{\p a}\frac{\p \mcal{H}}{\p \wt{p}_t} + \frac{\p \wt{q}_t}{\p a}\frac{\p \mcal{H}}{\p \wt{q}_t}\right]d\mathbf{Z}_t\\
        &= \wt{p}_{t_0} + \int_{t_0}^{t_1}\frac{\p \wt{q}_t}{\p a} \left[d \wt{p}_t + \frac{\p H}{\p \wt{q}_t}dt + \frac{\p \mcal{H}}{\p \wt{q}_t}d\mathbf{Z}_t\right] + \int_{t_0}^{t_1} \frac{\p \wt{p}_t}{\p a}\left[- d \wt{q}_t + \frac{\p H}{\p \wt{p}_t}dt + \frac{\p \mcal{H}}{\p \wt{p}_t}d\mathbf{Z}_t\right]\\
        &= \wt{p}_{t_0}\,.
    \end{align*}
    In the first equality we have used the smoothness property of composition of controlled paths and the chain rule property of integration against geometric rough path; In the second equality we have used integration by parts; In the third equality we have used the boundary conditions of $(\wt{q}_t, \wt{p}_t)$ and the last equality is implied by the rough Hamilton's equations \eqref{eq:rough ham eq}. A similar calculation gives the corresponding result with $D_2 \mcal{S}(a,b)$.
\end{proof}
\begin{corollary}\label{cor:symp form conservation}
Let $\omega = \sum_{i=1}^N \rmd p_i \wedge \rmd q^i$ be the canonical symplectic form on $T^*Q$, and let $\rmd$ denote the exterior derivative on the branch of boundary data $(a,b)$ fixed above. The symplectic $2$-form is preserved along the solution of the rough Hamilton's equations \eqref{eq:rough ham eq} at the boundary points $t = t_0, t_1$,
\begin{align}
    \rmd \wt{p}_{t_0} \wedge \rmd \wt{q}_{t_0} = \rmd \wt{p}_{t_1} \wedge \rmd \wt{q}_{t_1}\,.
\end{align}
Furthermore, the rough Hamiltonian flow $\Phi_{t_1,t_0}$ satisfies $\Phi^*_{t_1,t_0}\omega = \omega$.
\end{corollary}
\begin{proof}
    By Proposition \ref{prop:gen func}, together with the identifications $a = \wt{q}_{t_0}$ and $b = \wt{p}_{t_1}$, the differential of $\mcal{S}$ is the one-form
    \begin{align*}
        \rmd \mcal{S} = \wt{p}_{t_0}\,\rmd \wt{q}_{t_0} + \wt{q}_{t_1}\,\rmd \wt{p}_{t_1}\,.
    \end{align*}
    Since $\mcal{S}$ is twice continuously differentiable on the branch, $\rmd^2 \mcal{S} = 0$, so that
    \begin{align*}
        0 = \rmd \wt{p}_{t_0}\wedge \rmd \wt{q}_{t_0} + \rmd \wt{q}_{t_1}\wedge \rmd \wt{p}_{t_1} = \rmd \wt{p}_{t_0}\wedge \rmd \wt{q}_{t_0} - \rmd \wt{p}_{t_1}\wedge \rmd \wt{q}_{t_1}\,,
    \end{align*}
    which is the first assertion. Define the maps $E_0(a,b) := \left(a, \wt{p}_{t_0}(a,b)\right)$ and $E_1(a,b) := \left(\wt{q}_{t_1}(a,b), b\right)$ where $E_0^*\omega = E_1^*\omega$ holds following preceding arguments. By definition, we have $\Phi_{t_1,t_0}\circ E_0 = E_1$ and $E_0^*\left(\Phi^*_{t_1,t_0}\omega\right) = E_1^*\omega = E_0^*\omega$. By the regularity of the vector fields $\nabla H_i$ and the regular branch assumption \ref{assump: non degen}, the maps $E_0$ and $E_1$ are local $C^\infty$-diffeomorphisms. Applying $(E^{-1}_0)^*$ to both sides yields $\Phi^*_{t_1,t_0}\omega = \omega$.
\end{proof}

We now allow the two endpoints to vary. Let $t_0\leq s<t\leq t_1$ and $q,p\in\mathbb{R}^N$ be boundary value data. We denote with $(\ob{q},\ob{p})$ as the solution of \eqref{eq:rough ham eq} in the time domain $[s,t]$ satisfying $\ob{q}_s=q$ and $\ob{p}_t=p$. We further assume that the solution branch defined by the an open set around the boundary values $q$ and $p$ can be extended to $[t_0, t_1]$. From the smoothness of the Hamiltonian vector fields in \eqref{eq:rough ham eq}, the rough Hamiltonian flow $\Phi_{t,s}$ is a rough flow of local diffeomorphism.
The associated two-parameter Type-II generating function is
\begin{align}
\begin{split}
    \mcal S_{t,s}(q,p)
    &:={\scp{p}{\ob{q}_t}}
       -\int_s^t\scp{\ob{p}_r}{d\ob{q}_r}
       +\int_s^t H(\ob{q}_r,\ob{p}_r)\,dr 
       +\int_s^t\mcal H(\ob{q}_r,\ob{p}_r)\,d\mathbf{Z}_r,
       \label{eq:two parameter generating function}
\end{split}
\end{align}
where $\mcal S_{s,s}(q,p) = \scp{p}{q}$. The boundary solution and $\mcal S_{t,s}$ are then continuously differentiable in $(q,p)$, and Proposition \ref{prop:gen func} applied on the subinterval $[s,t]$ gives the following relations.
\begin{align}
    \frac{\p\mcal S_{t,s}}{\p q}(q,p)=\ob{p}_s,
    \qquad
    \frac{\p\mcal S_{t,s}}{\p p}(q,p)=\ob{q}_t.
    \label{eq:two parameter Type II relations}
\end{align}
\begin{proposition}\label{prop:rough HJ}
On every nondegenerate branch described above satisfying Assumption \ref{assump: non degen}, for fixed $s$ and $q$, set $S^{s,q}_t(p):=\mcal S_{t,s}(q,p)$. Then, $S^{s,q}$ satisfies the rough Hamilton--Jacobi equation
\begin{align}
    d S^{s,q}_t(p)
    &=H\left(\frac{\p S^{s,q}_t}{\p p}(p),p\right)\,dt
      +\mcal H\left(\frac{\p S^{s,q}_t}{\p p}(p),p\right)\,d\mathbf Z_t,
      \qquad S^{s,q}_s(p)=\scp{p}{q}\,.
      \label{eq:rough HJ 1}
\end{align}
In above and below the differential form is shorthand for an identity between controlled paths: for every compact $[u,v]$ contained in the branch,
\begin{align*}
    S^{s,q}_v(p) - S^{s,q}_u(p)
    = \int_u^v H\left(\frac{\p S^{s,q}_r}{\p p}(p),p\right)\,dr
      + \int_u^v \mcal H\left(\frac{\p S^{s,q}_r}{\p p}(p),p\right)\,d\mathbf Z_r\,,
\end{align*}
the second integral being the rough integral against $\mathbf Z$ of the controlled path $r \rightarrow \mcal H\left(\p S^{s,q}_r/\p p\,(p),\,p\right)$. Similarly, fix $t$ and $p$, and set $\wt S^{t,p}_s(q):=\mcal S_{t,s}(q,p)$. Then
\begin{align}
    d \wt S^{t,p}_s(q)
    &=-H\left(q,\frac{\p\wt S^{t,p}_s}{\p q}(q)\right)\,ds
      -\mcal H\left(q,\frac{\p\wt S^{t,p}_s}{\p q}(q)\right)\,d\mathbf Z_s,
      \qquad \wt S^{t,p}_t(q)=\scp{p}{q}.
      \label{eq:rough HJ 2}
\end{align}
that is, for every compact $[u,v]$ contained in the branch,
\begin{align*}
    \wt S^{t,p}_v(q) - \wt S^{t,p}_u(q)
    = -\int_u^v H\left(q,\frac{\p\wt S^{t,p}_r}{\p q}(q)\right)\,dr
      - \int_u^v \mcal H\left(q,\frac{\p\wt S^{t,p}_r}{\p q}(q)\right)\,d\mathbf Z_r\,.
\end{align*}
\end{proposition}
\begin{proof}
We will prove \eqref{eq:rough HJ 1} and \eqref{eq:rough HJ 2} follows by a similar computation. First we note that $S^{s,q}_t$ is a controlled path following the definition \eqref{eq:two parameter generating function} and the fact that integrals of controlled paths are once again controlled.
Via the rough chain rule for geometric rough path \cite[Thm. 3.3]{CHLN2022}, we have for all terminal value $p$ in the open set of boundary value data where a regular solution branch exists,
\begin{align*}
    d \left[S^{s,q}_t(\ob{p}_t)\right] &= \left(d S^{s,q}_t\right)(\ob{p}_t)- \scp{\frac{\p S^{s,q}_t}{\p \ob{p}_t}}{\frac{\p H}{\p \ob{q}_t}}dt - \scp{\frac{\p S^{s,q}_t}{\p \ob{p}_t}}{\frac{\p \mcal{H}}{\p \ob{q}_t}}d\mathbf{Z}_t \\
    & = \left(d S^{s,q}_t\right)(\ob{p}_t)- \scp{\frac{\p S^{s,q}_t}{\p p}}{\frac{\p H}{\p \ob{q}_t}}dt - \scp{\frac{\p S^{s,q}_t}{\p p}}{\frac{\p \mcal{H}}{\p \ob{q}_t}}d\mathbf{Z}_t \\
    &= \left(d S^{s,q}_t\right)(\ob{p}_t)- \scp{\ob{q}_t}{\frac{\p H}{\p \ob{q}_t}}dt - \scp{\ob{q}_t}{\frac{\p \mcal{H}}{\p \ob{q}_t}}d\mathbf{Z}_t\,,
\end{align*}
where we have used the generating function relation \eqref{eq:two parameter Type II relations} and the fact that on solution, $\ob{p}_t = p$. Directly computing the time increment of $S^{s,q}_t(\ob{p}_t)$ using \eqref{eq:two parameter generating function} gives
\begin{align*}
    d S^{s,q}_t(\ob{p}_t) = \scp{d \ob{p}_t}{\ob{q}_t} + H(\ob{q}_t, \ob{p}_t)\,dt + \mcal{H}(\ob{q}_t, \ob{p}_t)\,d\mathbf{Z}_t\,.
\end{align*}
Thus, we have 
\begin{align*}
    d S^{s,q}_t(\ob{p}_t) & = H(\ob{q}_t, \ob{p}_t)\,dt + \mcal{H}(\ob{q}_t, \ob{p}_t)\,d\mathbf{Z}_t =  H\left(\frac{\p S^{s,q}_t}{\p \ob{p}_t}, \ob{p}_t\right)\,dt + \mcal{H}\left(\frac{\p S^{s,q}_t}{\p \ob{p}_t}, \ob{p}_t\right)\,d\mathbf{Z}_t\,.
\end{align*}
and the rough Hamilton-Jacobi equation \eqref{eq:rough HJ 1} follow for all $p$ since $p \rightarrow \ob{p}_t$ is a local diffeomorphism by assumption. 

\end{proof}
For a discussion on the solution properties of the rough Hamilton-Jacobi equations \eqref{eq:rough HJ 1} and \eqref{eq:rough HJ 2}, see e.g., \cite{Friz2017}

\paragraph{Variational equations.}
The variational equations, also known as the Jacobi fields equations of the rough Hamilton's equations \eqref{eq:rough ham eq} are linearised equations for the perturbations of the initial conditions for $q$ and terminal condition for $p$. Let $\gamma = (\wt{q}, \wt{p}) \in \mcal{D}^{2\alpha}_Z(T^*Q)$ be a reference solution of the RDE \eqref{eq:rough ham eq} with boundary data conditions $(a,b)$. Consider the perturbations 
\begin{align*}
    a^\varepsilon = a + \varepsilon \delta a\,,\quad  b^\varepsilon = b + \varepsilon \delta b\,\quad \text{where}\quad \delta a\,,\delta b \in \mathbb{R}^N\,,
\end{align*}
where the solution $\wt{q}^\varepsilon, \wt{p}^\varepsilon$ associated with the perturbed boundary data $(a^\varepsilon, b^\varepsilon)$ exists in the same solution branch as the reference. We express the solution $\wt{q}^\varepsilon, \wt{p}^\varepsilon$ as $\varepsilon$-parameterised perturbations of the reference,
\begin{align*}
    \wt{q}^\varepsilon = \wt{q} + \varepsilon \delta \wt{q} + \mcal{O}(\varepsilon^2)\,,\quad \wt{p}^\varepsilon = \wt{p} + \varepsilon \delta \wt{p} + \mcal{O}(\varepsilon^2)\,.
\end{align*}
where $\delta \wt{q}, \delta \wt{p} \in \mcal{D}^{2\alpha}_Z(\mathbb{R}^N)$.
Geometrically, we can identify the reference solution as $(\wt{p}, \wt{q})= \gamma \in \mcal{D}^{2\alpha}_Z(T^*Q)$ and the perturbations as $(\delta \wt{q}, \delta \wt{p}) =: \xi \in \mcal{D}^{2\alpha}_Z(T_{\gamma}T^*Q)$. Directly applying the $\varepsilon$-derivative and noting that $\mathbf{Z}$ is geometric such that standard chain rule of ordinary calculus applies, we obtain the rough variational equations for the tangent vectors $\xi$,
\begin{align}
    \begin{split}
        \delta q_t - \delta q_{t_0} &= \int_{t_0}^t \left(\left[\frac{\p^2 H}{\p p_r \p q_r} \delta q_r + \frac{\p^2 H}{\p p_r \p p_r} \delta p_r\right]dr + \left[\frac{\p^2 \mcal{H}}{\p p_r \p q_r} \delta q_r + \frac{\p^2 \mcal{H}}{\p p_r \p p_r} \delta p_r\right]d\mathbf{Z}_r\right)\,, \\
        \delta p_{t_1} - \delta p_t &= -\int_{t}^{t_1} \left(\left[\frac{\p^2 H}{\p q_r \p p_r} \delta p_r + \frac{\p^2 H}{\p q_r \p q_r} \delta q_r\right]dr + \left[\frac{\p^2 \mcal{H}}{\p q_r \p p_r} \delta p_r + \frac{\p^2 \mcal{H}}{\p q_r \p q_r} \delta q_r\right]d\mathbf{Z}_r \right)\,,
    \end{split} \label{eq:rough var eq}
\end{align}
with boundary conditions $\delta q_{t_0} = \delta a$ and $\delta p_{t_1} = \delta b$. Let $\Phi_{t,s} : \mathbb{R}^2\times T^*Q \rightarrow T^*Q$ be the two-parameter rough flow such that $\Phi_{t,s} \gamma_s =\gamma_t = (\wt{q}_t, \wt{p}_t)$ whose existence is assumed. Then, the evolution of the tangent vectors $\xi$ can be expressed by the pushforward relation $\xi_t = \Phi_{t,s *} \xi_s$ for some initial condition $\xi_s \in T_{\gamma_s}T^*Q$. 
\begin{remark}
    The variational equations can alternatively be obtained from a second order variational principle where one require the first and second variation of \eqref{eq:rough type 2 action} to vanish identically. See, e.g., \cite{Jost2017, HHS2025a} for the derivation of the variational equations from variational principles in the deterministic setting.
\end{remark}
\begin{proposition}\label{prop: first poincare inv}
    Let $\xi^1 = (\delta \wt{q}^1, \delta \wt{p}^1)$ and $\xi^2 = (\delta \wt{q}^2, \delta \wt{p}^2)$ be two solutions of the variational equations \eqref{eq:rough var eq} following $(\wt{q}, \wt{p})$. Then,
    \begin{align}
        \omega_{\gamma_t}(\xi^1_t, \xi^2_t) = \omega_{\gamma_s}(\xi^1_s, \xi^2_s)\,,\quad (s,t) \in [t_0, t_1]^2\,,
    \end{align}
    where the notation $\omega_{\gamma_{(\cdot)}}$ is the evaluation of the symplectic form at the point $\gamma_{(\cdot)} \in T^*Q$.
\end{proposition}
\begin{proof}
    The variational equations in \eqref{eq:rough var eq} are linear rough differential equations whose coefficients are $C^\infty$ function composed with the reference solution $(\wt{q}, \wt{p})$. Restricting to the image of the reference solution, the coefficients define a locally bounded operator valued controlled paths by the composition lemma \cite[Lem. 7.3]{FH2020} since $\operatorname{Hess}(H)$ is $C^\infty$ and $\operatorname{Hess}(\mcal{H})$ is $C^\infty$ in its arguments component-wise. We proceed via direct computation using the solutions $\xi^1$ and $\xi^2$. Let
    \begin{align*}
        \mathbb{J} := \begin{pmatrix} 0 & I_N \\ -I_N & 0\end{pmatrix}\,,
        \qquad
        \mathrm{Hess}(F)(\gamma_r) := \begin{pmatrix}
            \p^2 F/\p q_r\,\p q_r & \p^2 F/\p q_r\,\p p_r \\
            \p^2 F/\p p_r\,\p q_r & \p^2 F/\p p_r\,\p p_r
        \end{pmatrix}\,,
    \end{align*}
    for $F = H, H_1, \ldots, H_K$ where $\mathbb{J}^{-1} = -\mathbb{J}$, $\mathbb{J}^T = -\mathbb{J}$, $\mathbb{J}^2 = -I_{2N}$. In this notation the variational equations \eqref{eq:rough var eq} can be expressed as
    \begin{align*}
        d\xi_r = \mathbb{J}\,\operatorname{Hess}(H)(\gamma_r)\,\xi_r\,dr + \mathbb{J}\,\operatorname{Hess}(\mcal{H})(\gamma_r)\,\xi_r\,d\mathbf{Z}_r\,,
    \end{align*}
    Since $\mathbf{Z}$ is geometric and $\xi^1, \xi^2$ are controlled by $\mathbf{Z}$, the chain rule of ordinary calculus applies to have
    \begin{align*}
        d\omega_{\gamma_r}(\xi^1_r, \xi^2_r) &= d \scp{\xi^1_r}{\mathbb{J}^{-1}\xi^2_r} = \scp{d\xi^1_r}{\mathbb{J}^{-1}\xi^2_r} + \scp{\xi^1_r}{\mathbb{J}^{-1}d\xi^2_r}\,,\\
        &=\scp{\mathbb{J\operatorname{Hess}}(H)(\gamma_r)\xi^1\,dr + \mathbb{J\operatorname{Hess}}(\mcal{H})(\gamma_r)\xi^1\,d\mathbf{Z}_r }{\mathbb{J}^{-1}\xi^2_r} \\
        & \qquad \qquad  + \scp{\xi^1_r }{\mathbb{J}^{-1}\left(\mathbb{J\operatorname{Hess}}(H)(\gamma_r)\xi^2\,dr + \mathbb{J\operatorname{Hess}}(\mcal{H})(\gamma_r)\xi^2\,d\mathbf{Z}_r\right)}\\
        &= \scp{\operatorname{Hess}(H)(\gamma_r)\xi^1}{\xi^2_r}\,dr - \scp{\operatorname{Hess}(H)(\gamma_r)\xi^2}{\xi^1_r}\,dr \\
        & \qquad \qquad + \scp{\operatorname{Hess}(\mcal{H})(\gamma_r)\xi^1}{\xi^2_r}\,d\mathbf{Z}_r - \scp{\operatorname{Hess}(\mcal{H})(\gamma_r)\xi^2}{\xi^1_r}\,d\mathbf{Z}_r \\
        & = 0\,.
    \end{align*}
    Thus, $\omega_{\gamma_r}(\xi^1_r, \xi^2_r)$ has zero increment for $r \in [t_0, t_1]$. Under assumptions that rough Hamiltonian flow $\Phi_{s, t}$ exists from Corollary \ref{cor:symp form conservation}, the proof is one line:
    \begin{align*}
        \omega_{\gamma_t}(\xi^1_t, \xi^2_t) = \omega_{\Phi_{t,s}\gamma_s}(\left(\Phi_{t,s}\right)_*\xi^1_s, \left(\Phi_{t,s}\right)_*\xi^2_s) = (\left(\Phi_{t,s}\right)^*\omega_{\gamma_s}) (\xi^1_s, \xi^2_s) = \omega_{\gamma_s}(\xi^1_s, \xi^2_s)\,.
    \end{align*}
\end{proof}

\subsection{Applications to adjoint systems}\label{sec:cts rough adjoints}
An important class of rough path driven Hamiltonian system in the form of \eqref{eq:rough ham eq} is the adjoint system. Consider a forward RDE of the form
\begin{align}
    d q_t = f(q)\, dt + \sigma(q)d\mathbf{Z}_t\,,\quad q_{t_0} = a\,, \label{eq:rough forward eq}
\end{align}
where $f \in C^\infty(\mathbb{R}^N, \mathbb{R}^N)$, $\sigma \in C^\infty(\mathbb{R}^N, \mcal{L}(\mathbb{R}^K, \mathbb{R}^N))$ and $\mathbf{Z} \in \mcal{C}^\alpha_g([t_0, t_1],\mathbb{R}^K)$ is a fixed geometric rough path for $\alpha \in (\frac{1}{3}, \frac{1}{2}]$. Let $\wt{q}$ be a solution of the previous RDE with initial condition $\wt{q}_{t_0} = a$, the associated variational equation for the quantity $\delta q \in \mcal{D}^{2\alpha}_Z(\mathbb{R}^N)$ is given by
\begin{align}\label{eq:rough variational eq}
    d \delta q_t = \frac{\p f}{\p \wt{q}_t}\delta q_t\, dt + \frac{\p \sigma}{\p \wt{q}_t}\delta q_t \,d\mathbf{Z}_t\,,\quad \delta q_{t_0} = \delta a\,.
\end{align}
The pathwise adjoint equation can be obtained by considering the degenerate Hamiltonians
\begin{align}\label{eq:linear ham}
    H(q, p) = \scp{p}{f(q)}\,, \quad \text{and} \quad  \mcal{H}(q,p) = \scp{p}{\sigma(q)}\,.
\end{align}
Here, we have abused the notation $\scp{p}{\sigma(q)}$ to mean the contraction $[\scp{p}{\sigma(q)}]_j := \sum_{i=1}^N p_i \sigma_{ij}(q) \in \mathbb{R}^K$. The rough adjoint system can be seen as a particular instance of the rough Hamilton's equation \eqref{eq:rough ham eq},
\begin{align} \label{eq:rough adjoint system}
\begin{split}
q_t - q_{t_0} &= \int_{t_0}^t f(q_r)\,dr + \sigma(q_r)\, d\mathbf{Z}_r \,, \\
p_{t_1} - p_t &= -\int_t^{t_1} \bigg[\left(\frac{\p f}{\p q_r}\right)^*p_r\,dr + \left(\frac{\p \sigma}{\p q_r}\right)^* p_r\, d\mathbf{Z}_r\bigg]\,.
\end{split}
\end{align}
\begin{proposition}\label{prop:existence solns}
    Let $\mathbf{Z} \in \mcal{C}^\alpha_g([t_0, t_1],\mathbb{R}^K)$ be fixed and let $f$ and $\sigma$ be as in \eqref{eq:rough forward eq}. For every $a \in \mathbb{R}^N$ where the forward RDE \eqref{eq:rough forward eq} has unique local solution up to $t = t_1$ with initial condition $\wt{q}_{t_0} = a$, the adjoint system \eqref{eq:rough adjoint system} has a unique global solution $(\wt{q}, \wt{p}) \in \mcal{D}^{2\alpha}_Z(T^*Q)$, defined on all of $[t_0,t_1]$, with boundary condition $\wt{p}_{t_1} = b$.
\end{proposition}
\begin{proof}
Let $\wt q \in \mcal{D}^{2\alpha}_Z(\mathbb{R}^N)$ denote the local solution with $\wt q_{t_0}=a$ of \eqref{eq:rough forward eq}. We have that 
\begin{align*}
    A_r &:=\left(\frac{\p f}{\p \wt{q}_r}\right)^* = \left(\frac{\p f}{\p q} \circ \wt{q}_r\right)^* \in \mcal{D}^{2\alpha}_Z(\mcal{L}(\mathbb R^N,\mathbb R^N))\,,\\
    B_r&:=\left(\frac{\p\sigma}{\p \wt{q}_r}\right)^* = \left(\frac{\p\sigma}{\p q} \circ \wt{q}_r\right)^* \in \mcal{D}^{2\alpha}_Z(\mcal{L}(\mathbb R^N,\mcal{L}(\mathbb R^K, \mathbb{R}^N)))\,.
\end{align*}
are linear operator-valued path via the composition lemma \cite[Lemma 7.3]{FH2020} with finite operator norm due since $f, \sigma $ are continuous differentiable and are evaluated at the solution $\wt{q}$. Let $\vartheta(r) = t_1 + t_0 - r$ and consider the time reversal $\cev{\mathbf{Z}}_t := \mathbf{Z}_{\vartheta(t)}$. One can show that $\cev{\mathbf{Z}} \in \mcal{C}^\alpha_g([t_0, t_1], \mathbb{R}^K)$ and through the standard theory of linear RDEs, we have global solution of the time reversed RDE
\begin{align*}
    \cev{p}_{\vartheta^{-1}(t)} - \cev{p}_{t_0} = \int^{\vartheta^{-1}(t)}_{t_0} A_{\vartheta(r)} \cev{p}_r \,dr - \int^{\vartheta^{-1}(t)}_{t_0} B_{\vartheta(r)} \cev{p}_r \,d\cev{\mathbf{Z}}_r\,,
\end{align*}
where $\cev{p} \in \mcal{D}^{2\alpha}_{\cev{Z}}(\mathbb{R}^N)$ with initial condition $\cev{p}_{t_0} = b$. It remains to show that $\wt{p}_t := \cev{p}_{\vartheta^{-1}(t)}$ satisfy \eqref{eq:rough adjoint system} which can verified by applying the change of variable transformation for rough integrals,
\begin{align*}
    \cev{p}_{\vartheta^{-1}(t)} - \cev{p}_{t_0} &= 
    \int^{\vartheta^{-1}(t)}_{t_0} A_{\vartheta(r)} \cev{p}_r \,dr - \int^{\vartheta^{-1}(t)}_{t_0} B_{\vartheta(r)} \cev{p}_r \,d\cev{\mathbf{Z}}_r \\
    &= \int^{t}_{\vartheta(t_0)} A_{u} \cev{p}_{\vartheta^{-1}(u)} \,du - \int^{t}_{\vartheta(t_0)} B_{u} \cev{p}_{\vartheta^{-1}(u)} \,d\cev{\mathbf{Z}}_{\vartheta^{-1}(u)} \\
    &= \int^{t_1}_{t} A_{u} \wt{p}_{u} \,du + \int^{t_1}_{t} B_{u} \wt{p}_{u} \,d\mathbf{Z}_{u}  = \wt{p}_{t} - \wt{p}_{t_1} \,.
\end{align*}
\end{proof}
For adjoint systems, Proposition \ref{prop:existence solns} implies that Assumption \ref{assump: non degen} is satisfied automatically when the unique local solution of the forward RDE \eqref{eq:rough forward eq} exists in $[t_0, t_1]$. The strictly linear form of the Hamiltonians $H$ and $\mcal{H}$ implies additional conservation laws distinct from the conservation of the symplectic form in Corollary \ref{cor:symp form conservation} and the invariance of the symplectic form along the variational equations in Proposition \ref{prop: first poincare inv}. 
\begin{proposition}\label{prop:canonical one form conservation}
    The canonical (Liouville) $1$-form $\theta := p\,\rmd q$ is preserved along the solution of the rough adjoint system \eqref{eq:rough adjoint system} at the boundary points $t = t_0, t_1$. That is,
    \begin{align}
        \wt{p}_{t_0}\rmd \wt{q}_{t_0} = \wt{p}_{t_1}\rmd \wt{q}_{t_1}\,.
    \end{align}
\end{proposition}
\begin{proof}
    Let $(\wt{p}, \wt{q}) \in \mcal{D}^{2\alpha}_Z(T^*Q)$ be solutions of \eqref{eq:rough adjoint system} with boundary conditions $\wt{q}_{t_0} = a,\, \wt{p}_{t_1} = b$. Since $H$ and $\mcal{H}$ are linear in $p$, the associated generating function $\mcal{S}(a,b)$ simplifies to $\mcal{S}(a,b) = \scp{\wt{p}_{t_1}}{\wt{q}_{t_1}}$ and it simultaneously satisfies the following two equalities
    \begin{align*}
        \frac{\p \mcal{S}}{\p a} = \wt{p}_{t_0} = \scp{\wt{p}_{t_1}}{\frac{\p \wt{q}_{t_1}}{\p a}} \,,
    \end{align*}
    where the second equality is due to $\wt{p}_{t_1} = b$ where $b$ is a parameter.  Thus, noting that $a = \wt{q}_{t_0}$ is also a parameter, we have that
    \begin{align*}
        \wt{p}_{t_0}\rmd \wt{q}_{t_0} = \frac{\p \mathcal{S}}{\p a}\rmd \wt{q}_{t_0} = \scp{\wt{p}_{t_1}}{\frac{\p \wt{q}_{t_1}}{\p a}}\rmd \wt{q}_{t_0} = \wt{p}_{t_1}\rmd \wt{q}_{t_1}\,.
    \end{align*}
    Here, in the last equality we have used the fact that $\wt{q}_{t_1}$ is not a function of $b$ as the $q$ equation in \eqref{eq:rough adjoint system} is decoupled from the $p$ equation.
\end{proof}
\begin{corollary}
    Let $\delta \wt{q}$ be the solution of the variational equation \eqref{eq:rough variational eq} associated with the solution $\wt{q}$ for the rough adjoint equations \eqref{eq:rough adjoint system}. Then we have the conservation law
    \begin{align}
        \scp{\wt{p}_{t_1}}{\delta \wt{q}_{t_1}} = \scp{\wt{p}_{t_0}}{\delta \wt{q}_{t_0}}\,. \label{eq:1-form inner prod conservation}
    \end{align}
\end{corollary}
\begin{proof}
    Via direct computation, we have
    \begin{align*}
        d\scp{\wt{p}_r}{\delta\wt{q}_r}
        &= \scp{d\wt{p}_r}{\delta\wt{q}_r} + \scp{\wt{p}_r}{d\delta\wt{q}_r}\\
        &= -\scp{\left(\frac{\p f}{\p \wt{q}_r}\right)^*\wt{p}_r\,dr + \left(\frac{\p \sigma}{\p \wt{q}_r}\right)^* \wt{p}_r\,d\mathbf{Z}_r}{\delta\wt{q}_r} + \scp{\wt{p}_r}{\left(\frac{\p f}{\p \wt{q}_r}\right)\delta\wt{q}_r \,dr + \left(\frac{\p \sigma}{\p \wt{q}_r}\right)\delta\wt{q}_r \,d\mathbf{Z}_r}\\
        & = 0\,.
    \end{align*}
    Hence the increment of $\scp{\wt{p}_r}{\delta\wt{q}_r}$ vanish for all $r \in [t_0,t_1]$. 
    A more elegant proof present itself since the rough Hamiltonian flow $\Phi_{s,t}$ is a rough flow of local diffeomorphism. Then, Proposition \ref{prop:canonical one form conservation} implies $\Phi_{t_0,t_1}^*\theta = \theta$ and writing $\scp{\wt{p}_t}{\delta\wt{q}_t} =: \theta_{\gamma_t}(\delta \wt{q}_t)$, we have
    \begin{align*}
        \theta_{\gamma_t}(\delta \wt{q}_t) = \theta_{\Phi_{t,s}\gamma_s}(\left(\Phi_{t,s}\right)_*\delta \wt{q}_s) = (\left(\Phi_{t,s}\right)^*\theta_{\gamma_s}) (\delta \wt{q}_s) = \theta_{\gamma_s}(\delta \wt{q}_s)\,.
    \end{align*}
\end{proof}

In its current form, the adjoint system \eqref{eq:rough adjoint system} can only accommodate applications in which a terminal cost is present. We consider an augmented version of the rough adjoint system by including terms corresponding to running costs by considering the following Hamiltonians that are \emph{affine} (rather than linear) in the momentum variable $p$,
\begin{align}
    H(q, p) = \scp{p}{f(q)} + L(q)\,, \quad \text{and} \quad  \mcal{H}(q,p) = \scp{p}{\sigma(q)} + \mathfrak{L}(q)\,, \label{eq:affine hams}
\end{align}
Here, $L \in C^\infty(\mathbb{R}^N, \mathbb{R})$ and $\mathfrak{L} \in C^\infty(\mathbb{R}^N, \mcal{L}(\mathbb{R}^K, \mathbb{R}))$ are the running costs to be integrated against $t$ and $\mathbf{Z}$, respectively. 
The rough adjoint systems with the augmented Hamiltonians in \eqref{eq:affine hams} are given by
\begin{align}\label{eq:rough adjoint aug}
    \begin{split}
        q_t - q_{t_0} &= \int_{t_0}^t f(q_r)\,dr + \sigma(q_r)\, d\mathbf{Z}_r \,, \\
        p_{t_1} - p_t &= - \int_t^{t_1} \left(\left[\left(\frac{\p f}{\p q_r}\right)^*p_r + \frac{\p L}{\p q_r}\right]\,dr + \left[\left(\frac{\p \sigma}{\p q_r}\right)^* p_r + \frac{\p \mathfrak{L}}{\p q_r}\right]\, d\mathbf{Z}_r\right)\,.
    \end{split}
\end{align}
Under the current assumptions on $H$ and $\mcal{H}$, a repeat of Proposition \ref{prop:existence solns} yields that \eqref{eq:rough adjoint aug} have unique solutions in $[t_0, t_1]$ provided that the assumptions of Proposition \ref{prop:existence solns} holds. 
\begin{proposition}\label{prop:aug 1form conservation}
    Let $(\wt{q}, \wt{p})$ be the solution to the rough adjoint system \eqref{eq:rough adjoint aug} and let $\delta\wt{q}$ be the solution to the associated variational equation \eqref{eq:rough variational eq}. Then we have the quasi conservation law
    \begin{align}
        \scp{\wt{p}_{t_1}}{\delta \wt{q}_{t_1}} - \scp{\wt{p}_{t_0}}{\delta \wt{q}_{t_0}} = -\int_{t_0}^{t_1} \Bigg(\scp{\frac{\p L}{\p \wt{q}_r}}{\delta \wt{q}_r}\,dr + \scp{\frac{\p \mathfrak{L}}{\p \wt{q}_r}}{\delta \wt{q}_r}\,d\mathbf{Z}_r\Bigg)\,.\label{eq:quasi 1-form conservation}
    \end{align}
\end{proposition}
\begin{proof}
Via direct computation, we have
    \begin{align*}
        d\scp{\wt{p}_r}{\delta\wt{q}_r}
        &= \scp{d\wt{p}_r}{\delta\wt{q}_r} + \scp{\wt{p}_r}{d\delta\wt{q}_r}\\
        &= -\scp{\left(\frac{\p f}{\p \wt{q}_r}\right)^*\wt{p}_r\,dr + \frac{\p L}{\p \wt{q}_r}dr + \left(\frac{\p \sigma}{\p \wt{q}_r}\right)^* \wt{p}_r\,d\mathbf{Z}_r + \frac{\p \mathfrak{L}}{\p \wt{q}_r}d\mathbf{Z}_r}{\delta\wt{q}_r} \\
        & \qquad \qquad + \scp{\wt{p}_r}{\left(\frac{\p f}{\p \wt{q}_r}\right)\delta\wt{q}_r \,dr + \left(\frac{\p \sigma}{\p \wt{q}_r}\right)\delta\wt{q}_r \,d\mathbf{Z}_r}\\
        & = -\scp{\frac{\p L}{\p \wt{q}_r}dr + \frac{\p \mathfrak{L}}{\p \wt{q}_r}d\mathbf{Z}_r}{\delta \wt{q}_r}\,.
    \end{align*}
    Integrating between $r\in[t_0,t_1]$ yields the result.
\end{proof}

\begin{remark}
    Another proof directly uses the Type-II generating function. Let $\mcal{S} : \mathbb{R}^N\times \mathbb{R}^N \rightarrow \mathbb{R}$ be the extremum of the functional $\mathfrak{S}$ defined in \eqref{eq:rough type 2 action} with $\tau$-parametrised boundary conditions. That is, let $a_{(\cdot)}, b_{(\cdot)} \in C^\infty([0,1),\mathbb{R}^N)$ be arbitrary with $(a_0, b_0) = (a, b)$ and define
    \begin{align}
        \begin{split}
            \mathcal{S}(a_\tau, b_\tau) &:= \underset{\substack{(q^\epsilon,\, p^\epsilon), \\q_{t_0} = a_\tau,\, p_{t_1}=b_\tau }}{\text{ext}}\mathfrak{S}(q^\epsilon(\cdot), p^\epsilon(\cdot)) = \mathfrak{S}(\wh{q}, \wh{p})\\
    &= \scp{\wh{p}_{t_1}}{\wh{q}_{t_1}} - \int_{t_0}^{t_1}\scp{\wh{p}_t}{d \wh{q}_t} + \int_{t_0}^{t_1}H(\wh{q}_t, \wh{p}_t)\,dt + \int_{t_0}^{t_1}\mcal{H}(\wh{q}_t, \wh{p}_t) d\mathbf{Z}_t \,.
        \end{split}
    \end{align}
    Here, $(\wh{q}, \wh{p}) : [0,1) \rightarrow \mcal{D}^{2\alpha}_Z(T^*Q)$ are $\tau$ parameterised controlled paths that are solutions to the rough Hamilton's equations \eqref{eq:rough ham eq} with boundary conditions $\wh{q}_{t_0} = a_\tau$ and $\wh{p}_{t_1} = b_\tau$. When $\tau =0$, we have the solutions $(\wt{q}, \wt{p}) = (\wh{q}, \wh{p})$. Taking $\delta^\tau := \frac{d}{d\tau}\big|_{\tau = 0}$, we have $\delta^\tau\wh{q} = \delta \wt{q}$ and $\delta^\tau \wh{p} = \delta\wt{p}$ satisfying the variational equations \eqref{eq:rough var eq}. One can show the implicit symplectic mapping 
    \begin{align*}
        \wh{q}_{t_1} = D_2 \mcal{S}(a_\tau, b_\tau)\,,\quad \text{and}\quad \wh{p}_{t_0} = D_1 \mcal{S}(a_\tau, b_\tau)\,,
    \end{align*}
    holds for all $\tau \in [0,1)$. Fixing $b_\tau = b$ as a constant, we have 
    \begin{align*}
        \delta^\tau \mcal{S}(a_\tau, b) = \scp{\wt{p}_{t_0}}{\delta^\tau a_\tau} &= \scp{\wt{p}_{t_0}}{\delta \wt{q}_{t_0}}\,,
    \end{align*}
    where we made the identification $\delta^\tau a_\tau := \delta \wt{q}_{t_0}$ as the initial condition of the variational equation \eqref{eq:rough var eq}. Inserting the particular choice of the affine Hamiltonians \eqref{eq:affine hams}, $\mcal{S}$ becomes 
    \begin{align*}
        \mathcal{S}(a_\tau, b) = \scp{b}{\wh{q}_{t_1}} + \int_{t_0}^{t_1}L(\wh{q}_r)\,dr + \mathfrak{L}(\wh{q}_r)\,d\mathbf{Z}_r\,, 
    \end{align*}
    and we compute $\tau$-derivative as
    \begin{align*}
        \begin{split}
            \delta^\tau S(a_\tau, b) &= \scp{b}{\delta^\tau \wh{q}_{t_1}} + \int_{t_0}^{t_1} \scp{\frac{\p L}{\p \wh{q}_r}}{\delta^\tau \wh{q}_r}\biggr|_{\tau=0}\,dr + \scp{\frac{\p \mathfrak{L}}{\p \wh{q}_r}}{\delta^\tau \wh{q}_r}\biggr|_{\tau=0}\,d\mathbf{Z}_r\\
            &= \scp{\wt{p}_{t_1}}{\delta \wt{q}_{t_1}} + \int_{t_0}^{t_1} \scp{\frac{\p L}{\p \wt{q}_r}}{\delta\wt{q}_r}\,dr + \scp{\frac{\p \mathfrak{L}}{\p \wt{q}_r}}{\delta \wt{q}_r}\,d\mathbf{Z}_r\,.
        \end{split}
    \end{align*}
    Equating the $\delta^\tau$-derivatives of $\mcal{S}$ to have the result of Proposition \ref{prop:aug 1form conservation}. We will make use of a similar computation in the next section to show the class of rough Galerkin integrators preserves these quasi-conservation laws.
\end{remark}

\subsection{Adjoint sensitivities.}\label{sec:cts rough adjoints sensitivities}
The pathwise (quasi)-conservation laws given in Propositions \ref{prop:canonical one form conservation} and \ref{prop:aug 1form conservation} imply a natural method to compute sensitivities (gradients) of cost function with respect to initial conditions and parameters when the states evolves following a rough dynamical system. Moreover, when the driving rough path is a random path, i.e., $\mathbf{Z} = \mathbf{Z}(\omega)$ where $\omega \in \Omega$, the sample of the probability space $(\Omega, \mcal{F}, \mathbb{P})$, one may consider the expectation of a cost function over that probability space. A canonical example of random geometric rough path is the Stratonovich lifted Brownian motion $\mathbf{B}^{Strat} = (B, \mathbb{B}^{Strat})$. In this subsection, we will demonstrate adjoint sensitivities for RDEs driven by random paths through the example of Stratonovich lifted Brownian motion. For the RDE
\begin{align}
    dq_t = f(q_t)\,dt + \sigma(q_t)\, d\mathbf{B}^{Strat}_t \,, 
    \label{eq:basic rde strat}
\end{align}
we consider the computation of adjoint sensitivity of the expectation of the cost function $\mcal{L}$ defined as the sum of terminal cost $C\in C^\infty(Q)$ and running costs $L \in C^\infty(Q)$, $\mathfrak{L}\in C^\infty(Q, \mcal{L}(\mathbb{R}^K,\mathbb{R}))$, that is, 
\begin{align}
    \mcal{L}[q] = C(q_{t_1}) + \int_{t_0}^{t_1} L(q_t)\,dt + \mathfrak{L}(q_t)\,d\mathbf{B}^{Strat}_t\,.
\end{align}

\paragraph{Initial condition sensitivity.}
For each driving Stratonovich enhanced Brownian motion, following \cite{Sanz-Serna2016,TL2024}, we consider the action 
\begin{align}
    J = C(q_{t_1}) - \scp{p_0}{q_{t_0} - q_0} - \int_{t_0}^{t_1} \left[\scp{p_t}{dq_t - f(q_t)\,dt - \sigma(q_t)\, d\mathbf{B}^{Strat}_t}-L(q_t)\,dt - \mathfrak{L}(q_t)\,d\mathbf{B}^{Strat}_t\right]\,,\label{eq:sensitivity J}
\end{align}
where $q_0$ will become the initial condition for the $q$ evolution. Taking variations of $J$ without assuming vanishing boundary conditions for the variations of $q_t$ and $p_t$, we have
\begin{align*}
\begin{split}
    \delta J &= \scp{\frac{\p C}{\p q_{t_1}} - p_{t_1}}{\delta q_{t_1}} + \scp{p_{t_0} - p_0}{\delta q_{t_0}} + \scp{p_0}{\delta q_0} - \scp{\delta p_0}{q_{t_0} - q_0}\\
    & \qquad - \int_{t_0}^{t_1}\scp{\delta p_t}{dq_t - f(q_t)\,dt - \sigma(q_t)\, d\mathbf{B}^{Strat}_t}\\
    & \qquad + \int_{t_0}^{t_1}\scp{d p_t + \left[\left(\frac{\p f}{\p q_t}\right)^* p_t + \frac{\p L}{\p q_t}\right]\,dt + \left[\frac{\p \mathfrak{L}}{\p q_t} + \left(\frac{\p \sigma}{\p q_t}\right)^* p_t \right]\,d\mathbf{B}^{Strat}_t}{\delta q_t}\,.
\end{split}
\end{align*}
Assuming that the terminal conditions $p_{t_1} = \frac{\delta C}{\delta q_{t_1}}$ and $p_0 = p_{t_0}$, initial condition $q_0 = q_{t_0}$, and the adjoint system augmented with running cost 
\begin{align}
    &dq_t = f(q_t)\,dt + \sigma(q_t)\, d\mathbf{B}^{Strat}_t\,, \label{eq:rde forward}\\
    &d p_t = -\left[ \left(\frac{\p f}{\p q_t}\right)^* p_t \,dt + \frac{\p L}{\p q_t} \,dt + \left(\frac{\p \sigma}{\p q_t}\right)^* p_t \,d\mathbf{B}^{Strat}_t + \frac{\p \mathfrak{L}}{\p q_t}\,d\mathbf{B}^{Strat}_t \right]\,, \label{eq:rde adjoint backward}
\end{align}
all hold, we have that $\delta J = \scp{p_0}{\delta q_0}$ and $J$ can be evaluated as 
\begin{align}
    J = C(q_{t_1}) + \int_{t_0}^{t_1} L(q_t)\,dt  + \mathfrak{L}(q_t)\,d\mathbf{B}^{Strat}_t = \mcal{L}[q] 
\end{align}
such that the adjoint sensitivity of $\mcal{L}$ with respect to the initial condition $q_0 = q_{t_0}$ is given by
\begin{align}
    \frac{\delta \mcal{L}[q]}{\delta q_{t_0}} = \frac{\delta J}{\delta q_0} = p_0 = p_{t_0}\,,\label{eq:init cond sensitivity}
\end{align}
for every realisation of the Brownian path $\mathbf{B}^{Strat}$. To obtain $\delta J/\delta q_0$, one first solves \eqref{eq:rde forward} with initial condition $q_0$ to time $t_1$, then solve \eqref{eq:rde adjoint backward} with condition $p_{t_1} = \delta C/\delta q_{t_1}$ backwards in time to obtain $p_{t_0}$ using the solution of $q$ in the forward simulation, thereby obtaining the sensitivity of $J$ for a fixed path $\mathbf{B}^{Strat}$. One may recast the pathwise adjoint sensitivity into the context of Proposition \ref{prop:aug 1form conservation}. Starting with the quasi-conservation law \eqref{eq:quasi 1-form conservation}, setting $\wt{p}_{t_1} = \delta C/\delta q_{t_1}$ directly yields \eqref{eq:init cond sensitivity}. 

To obtain the sensitivity of $\mathbb{E}[J]$ with respect to $q_0$, we have 
\begin{align}
    \frac{\delta }{\delta q_0} \mathbb{E}[J] = \mathbb{E}\left[\frac{\delta J}{\delta q_0} \right] = \mathbb{E}[p_0]\,.
\end{align}
The conditions under which the commutation of derivatives and expectations holds are non-trivial as it require Dominated Convergence theorem to apply to the derivatives of $p_0$ whose moments are controlled by the RDE \eqref{eq:rough adjoint aug}. We refer to \cite{CLL2013} for a discussion on the moments bounds of RDEs driven by Gaussian rough paths and proceed assuming that the Dominated Convergence theorem holds. 

We remark that this type of path-wise rough (stochastic) adjoint system has appeared in the literature before, most notably in \cite{Li2020}. For a comparison between (forward) path-wise sensitivity and Malliavin calculus, see e.g., \cite{Fourni1999}.

\paragraph{Parametric sensitivity.}
A similar computation demonstrates the sensitivity of the cost function $\mcal{L}$ against parameters. Let $\bs{\theta} \in \mathbb{R}^M$ be the vector of real valued parameters and consider a modification to the RDE \eqref{eq:basic rde strat} 
\begin{align}
    dq_t = f(q_t; \bs{\theta})\,dt + \sigma(q_t; \bs{\theta})\, d\mathbf{B}^{Strat}_t \,,
\end{align}
as well its adjoint equation with respect to $\mcal{L}$, (\eqref{eq:rde adjoint backward} with appropriate replacement of $f$ and $\sigma$ with their $\bs{\theta}$-dependent counterparts). Here, we assume that the maps
\begin{align*}
    \bs{\theta} \rightarrow f(\cdot; \bs{\theta})\,, \quad \bs{\theta} \rightarrow \sigma(\cdot; \bs{\theta})\,,
\end{align*}
are $C^\infty$-smooth.
Consider the action \eqref{eq:sensitivity J}, assuming $q_{t_0}$ and $q_0$ are independent of $\bs{\theta}$, directly taking $\bs{\theta}$-derivative we obtain
\begin{align*}
    \frac{d}{d\bs{\theta}}J &= \scp{\frac{\p C}{\p q_{t_1}}}{\frac{d}{d\bs{\theta}}q_{t_1}} - \scp{\frac{d}{d\bs{\theta}}p_0}{q_{t_0} - q_0} - \int_{t_0}^{t_1}\scp{\frac{d}{d\bs{\theta}}p_t}{dq_t - f(q_t;\bs{\theta})\,dt - \sigma(q_t;\bs{\theta})d\mathbf{B}^{Strat}_t}\\
    & \quad - \int_{t_0}^{t_1}\scp{p_t}{d \frac{d}{d\bs{\theta}}q_t - \frac{d}{d\bs{\theta}}f(q_t;\bs{\theta})\,dt - \frac{d}{d\bs{\theta}}\sigma(q_t;\bs{\theta})d\mathbf{B}^{Strat}_t} \\
    &\quad + \int_{t_0}^{t_1}\scp{\frac{\p L}{\p q_t}}{\frac{d}{d\bs{\theta}}q_t}dt + \scp{\frac{\p \mathfrak{L}}{\p q_t}}{\frac{d}{d\bs{\theta}}q_t}d\mathbf{B}^{Strat}_t\\
    &= \scp{\frac{\p C}{\p q_{t_1}}- p_{t_1}}{\frac{d}{d\bs{\theta}}q_{t_1}} - \scp{\frac{d}{d\bs{\theta}}p_0}{q_{t_0} - q_0} - \int_{t_0}^{t_1}\scp{\frac{d}{d\bs{\theta}}p_t}{dq_t - f(q_t;\bs{\theta})\,dt - \sigma(q_t;\bs{\theta})d\mathbf{B}^{Strat}_t}\\
    & \quad + \int_{t_0}^{t_1}\scp{d p_t + \left[\left(\frac{\p f}{\p q_t}\right)^* p_t + \frac{\p L}{\p q_t}\right]\,dt + \left[\frac{\p \mathfrak{L}}{\p q_t} + \left(\frac{\p \sigma}{\p q_t}\right)^* p_t \right]\,d\mathbf{B}^{Strat}_t}{\frac{d}{d\bs{\theta}}q_t} \\
    & \quad + \int_{t_0}^{t_1}\scp{p_t}{\frac{\p f}{\p \bs{\theta}}\,dt + \frac{\p \sigma}{\p \bs{\theta}}d\mathbf{B}^{Strat}_t}\,.
\end{align*}
Assuming that the adjoint system \eqref{eq:rde forward}-\eqref{eq:rde adjoint backward} are satisfied with initial condition $q_{t_0} = q_0$ and terminal condition $\frac{\p C}{\p q_{t_1}} = p_{t_1}$, we have that 
\begin{align}
    \frac{d}{d\bs{\theta}}J = \frac{d}{d\bs{\theta}}\mcal{L}[q] = \int_{t_0}^{t_1}\scp{p_t}{\frac{\p f}{\p \bs{\theta}}\,dt + \frac{\p \sigma}{\p \bs{\theta}}d\mathbf{B}^{Strat}_t}\,,\label{eq:path param sens}
\end{align}
for each realisation of the Brownian motion $\mathbf{B}^{Strat}$ as a rough path. Similar to the case of sensitivity to initial conditions, the ensemble sensitivity $\mathbb{E}[J]$ can be obtained by 
\begin{align}
    \frac{d}{d\bs{\theta}}\mathbb{E}[J] = \mathbb{E}\left[\frac{d}{d\bs{\theta}}J\right] = \mathbb{E}\left[\int_{t_0}^{t_1}\scp{p_t}{\frac{\p f}{\p \bs{\theta}}\,dt + \frac{\p \sigma}{\p \bs{\theta}}d\mathbf{B}^{Strat}_t}\right]\,.\label{eq:param expect sens}
\end{align}
\begin{remark}
Of course, one may cast parametric sensitivity into the form of initial condition sensitivity by enlarging the state space, $q \rightarrow \underline{q} = (q, \bs{\theta}) \in \mcal{D}_Z^{2\alpha}(\mathbb{R}^{N+M})$ as well as the adjoint variables accordingly, $p \rightarrow \underline{p} = (p, \bs{\phi}) \in \mcal{D}_Z^{2\alpha}(\mathbb{R}^{N+M})$. Assuming that $\bs{\theta}$ is constant in time, we consider the augmented action
\begin{align*}
    J &= C(q_{t_1}) +\int_{t_0}^{t_1} \left[L(q_t)\,dt + \mathfrak{L}(q_t)\,d\mathbf{B}^{Strat}_t\right] - \scp{\bs{\phi}_0}{\bs{\theta}_{t_0} - \bs{\theta}_0} - \int_{t_0}^{t_1} \scp{\bs{\phi}_t}{d\bs{\theta}_t} \\ 
    & \qquad  - \scp{p_0}{q_{t_0} - q_0} - \int_{t_0}^{t_1} \scp{p_t}{dq_t - f(q_t;\bs{\theta})\,dt - \sigma(q_t;\bs{\theta})\, d\mathbf{B}^{Strat}_t}  \,.
\end{align*}
Here, $\bs{\theta}_0$ is the constant determining the value of the parameters. The $\bs{\phi}$ variations recovers $d\bs{\theta} = 0$ and the $\bs{\theta}_t$ variations yields the adjoint dynamics of $\bs{\phi}$ whose terminal condition is $0$.
\begin{align*}
    d\bs{\phi}_t+\scp{p_t}{\frac{\p f}{\p \bs{\theta}}}\,dt + \scp{p_t}{\frac{\p \sigma}{\p \bs{\theta}}}\,d\mathbf{B}^{Strat}_t = 0\,.
\end{align*}
Following a similar argument as before and substituting extended adjoint system, initial conditions for $q, \bs{\theta}$ and terminal conditions for $p, \bs{\phi}$, we obtain that $\frac{d}{d\bs{\theta}}J = \bs{\phi}_0$ which agrees with \eqref{eq:path param sens}.
\end{remark}

\section{Variational discretisation}\label{sec:variational discretisation}
As pointed out in \cite{Sanz-Serna2016}, in the deterministic case, the accuracy of the adjoint sensitivity is directly related to the conservative properties of the numerical integrator used. In this section, we extend the geometric integration of deterministic adjoint systems to rough adjoint systems \eqref{eq:rde forward}-\eqref{eq:rde adjoint backward}. To this end, we consider a variational discretisation using a discrete Type-II variational principle similar to that found in \cite{LZ2011,HT2018,TL2024}. We construct a class of rough Galerkin integrators for the rough Hamilton's equations \eqref{eq:rough ham eq}, demonstrate their conservative properties and their equivalence to Rough Symplectic Partitioned Runge--Kutta (RSPRK) methods, before specialising to rough adjoint systems and their (quasi)-conservation laws. We also briefly discuss the rate of convergence of these methods.

\subsection{Rough Galerkin discretisation}\label{subsec:rough galerkin}
Following \cite{LZ2011, HT2018}, we consider a Galerkin type discretisation of the controlled path $q$ by projecting it into a finite dimensional polynomial space of degree $s$. Consider an uniform partitioning of $[t_0, t_1]$ into intervals $[t_k, t_{k+1}]$ for $k=0, \ldots, n-1$ such that $t_{k+1} - t_k =: \Delta t$ is a constant. Let $Z_{t_k, t_{k+1}} := Z_{t_{k+1}} - Z_{t_k}$ denote the increment of the driving rough path on $[t_k, t_{k+1}]$. Define the piecewise smooth path $Z^{\Delta t}$ as
\begin{align}
    Z^{\Delta t}_t := Z_{t_k} + \frac{t-t_k}{\Delta t}\left(Z_{t_{k+1}} - Z_{t_k}\right)\,,\quad t \in[t_k, t_{k+1}]\,, \label{eq:piecewise smooth path}
\end{align}
that interpolates those increments. We assume that the canonical lift of $Z^{\Delta t}$, $\mathbf{Z}^{\Delta t} = (Z^{\Delta t}, \mathbb{Z}^{\Delta t})$, converge to $\mathbf{Z}$ in the homogenous $\alpha$-H\"older rough path metric as $\Delta t \rightarrow 0$. That is, $\rho^\alpha_g(\mathbf{Z}, \mathbf{Z}^{\Delta t}) = \mcal{O}(\Delta t^{r_0})$ for some $r_0 > 0$. 

Let $\{d_\nu\}_{\nu=0}^s$ be control points satisfying $0 = d_0 < d_1 <\ldots < d_s = 1$ and let $\{l_{\mu}\}_{\mu=0}^s$ be Lagrange polynomial of degree $s$ defined on the control points $\{d_\nu\}$ satisfying $l_\mu(d_{\nu}) = \delta_{\mu\nu}$ where $\delta_{\mu\nu}$ is the Kronecker delta. The finite dimensional representation of $q$ and its derivative are given by
\begin{align}
    q_d(t_k + \eta \Delta t) = \sum_{\mu=0}^s q^\mu l_\mu(\eta)\,,\quad \dot{q}_d(t_k + \eta \Delta t) = \frac{1}{\Delta t}\sum_{\mu=0}^s q^\mu \dot{l}_\mu(\eta)\,, \label{eq:rough q interpolation}
\end{align}
where we have introduced the notation $q^\mu := q(t_k + d_\mu\Delta t)$ as the control values at the control points $d_\mu$. On an arbitrary interval $[t_k, t_{k+1}]$, we approximate the action functional \eqref{eq:rough type 2 action} using numerical quadratures over the quadrature points $\{c_i\}_{i=1}^r$. The quadratures weights are $\{b_i\}_{i=1}^r$ and $\{\ob{b}_i\}_{i=1}^r$ for integration over smooth path $t$ and rough path $\mathbf{Z}$ respectively.

We reserve the notation $\mathbf{Z}$ for the area enhanced rough path and $Z_{t_k, t_{k+1}}$ is taken as a $\mathbb{R}^K$ vector where multiplications with the Hamiltonian $\mcal{H}$ is the contraction over the index of rough path dimension, $\mcal{H}\cdot Z_{t_k, t_{k+1}} := H_k Z^k_{t_{k+1},t_k}$, similarly for the gradients of $\mcal{H}$.
Consider the discrete action
\begin{align}
    \mathfrak{S}(q_k, p_{k+1}) = p_{k+1}q^s - \Delta t \sum_{i=1}^r b_i\left( P_{i,k} \dot{q}_d(t_k + c_i \Delta t) - H(Q_{i,k}, P_{i,k})\right) + \sum_{i=1}^r \ob{b}_i\mcal{H}(Q_{i,k}, P_{i,k})\cdot Z_{t_k, t_{k+1}}\,,\label{eq:discrete action}
\end{align}
where $Q_{i,k} := q_d(t_k + c_i \Delta t)$ and $P_{i,k} := p(t_k + c_i \Delta t)$ which we take as the control values of the finite dimensional approximation of $p$ at the control points $\{c_i\}_{i=1}^r$ where $0\leq c_1 \leq \ldots\leq c_r \leq 1$. Here we note that $d_s = 1$ implies that $q^s = q_d(t_k + \Delta t) = q_{k+1}$ and $d_0 = 0$ implies that $q^0 = q_d(t_k) = q_k$.
\begin{remark}
    The discrete action only uses the level-one increment $Z_{t_k, t_{k+1}}$ alone without higher level of the signature of $\mathbf{Z}$. As demonstrated in Section \ref{subsec:convergence}, this choice imposes a order barrier on the class of Runge Kutta (RK) methods constructed via the Galerkin approach. 
\end{remark}
Given the states $q_k := q(t_k)$, $ p_{k+1}:= p(t_{k+1})$ as inputs, a discrete approximation of the generating function $\mcal{S}$ can be define as the extremum over the control points $\{q^\mu\}_{\mu=0}^s$ and $\{P_{i,k}\}_{i=1}^r$ of the below expression,
\begin{align*}
\begin{split}
    \mcal{S}_d(q_k, p_{k+1}) &= \underset{\substack{P_{i,k}, q^\mu,\\  q^0=q_k }}{\text{ext}}\left[p_{k+1}q^s - \Delta t \sum_{i=1}^r b_i\left( P_{i,k} \dot{q}_d(t_k + c_i \Delta t) - H(Q_{i,k}, P_{i,k})\right) + \sum_{i=1}^r \ob{b}_i\mcal{H}(Q_{i,k}, P_{i,k})\cdot Z_{t_k, t_{k+1}}\right]\,,
\end{split}
\end{align*}
We also impose a discrete version of the implicit symplectic mapping \eqref{eq:cont sym mapping}, 
\begin{align}\label{eq:discrete sym mapping}
    p_k = \frac{\p \mcal{S}_d}{\p q_k}(q_k, p_{k+1})\,,\quad q_{k+1} = \frac{\p \mcal{S}_d}{\p p_{k+1}}(q_k, p_{k+1})\,,
\end{align}
where $q_{k+1} := q(t_{k+1})$ and $ p_k:= p(t_k)$ are the outputs determined by \eqref{eq:discrete sym mapping} from the inputs $(q_k, p_{k+1})$. A straight forward computation gives the following conditions on $\{q^\mu\}_{\mu=0}^s$ and $\{P_{i,k}\}_{i=1}^r$,
\begin{subequations}\label{eq:Galerkin stat cond 1}
\begin{align}
    & - \Delta t b_i\left(\dot{q}_d(t_k + c_i\Delta t) - \frac{\p H}{\p p}(Q_{i,k}, P_{i,k})\right) + \ob{b}_i\frac{\p \mcal{H}}{\p p}(Q_{i,k}, P_{i,k})\cdot Z_{t_k, t_{k+1}} = 0\,,\quad i = 1,\ldots,r\,,\label{eq:galerkin stat 1}\\
    & -  \sum_{i=1}^r b_i\left( P_{i,k} \dot{l}_\mu(c_i) - \Delta t \frac{\p H}{\p q}(Q_{i,k}, P_{i,k})l_\mu(c_i)\right) + \sum_{i=1}^r \ob{b}_i\frac{\p \mcal{H}}{\p q}(Q_{i,k}, P_{i,k})l_\mu(c_i)\cdot Z_{t_k, t_{k+1}} = 0\,,\label{eq:galerkin stat 2}\\
    &\hspace{33em}\quad \mu=1,\ldots,s-1\,, \notag \\
    &p_{k+1} - \sum_{i=1}^r b_i\left( P_{i,k} \dot{l}_s(c_i) - \Delta t \frac{\p H}{\p q}(Q_{i,k}, P_{i,k})l_s(c_i)\right) + \sum_{i=1}^r \ob{b}_i\frac{\p \mcal{H}}{\p q}(Q_{i,k}, P_{i,k})l_s(c_i)\cdot Z_{t_k, t_{k+1}} = 0\,,\label{eq:galerkin stat 3}\\
    &p_k = - \sum_{i=1}^r b_i\left( P_{i,k} \dot{l}_0(c_i) - \Delta t \frac{\p H}{\p q}(Q_{i,k}, P_{i,k})l_0(c_i)\right) + \sum_{i=1}^r \ob{b}_i\frac{\p \mcal{H}}{\p q}(Q_{i,k}, P_{i,k})l_0(c_i)\cdot Z_{t_k, t_{k+1}} \,,\label{eq:galerkin stat 4}\\
    &q_{k+1} = q^s\,,\label{eq:galerkin stat 5}
\end{align}
\end{subequations}
which defines the rough Galerkin method. Let the variables under $\wt{(\cdot)}$ notations be the variables satisfying the extremum conditions \eqref{eq:Galerkin stat cond 1}, then the discrete generating function $\mcal{S}_d$ can be expressed simply as 
\begin{align*}
    \mcal{S}_d(q_k, p_{k+1})&= p_{k+1}\wt{q}^s - \Delta t \sum_{i=1}^r b_i\left( \wt{P}_{i,k} \dot{\wt{q}}_d(t_k + c_i \Delta t) - H(\wt{Q}_{i,k}, \wt{P}_{i,k})\right) + \sum_{i=1}^r \ob{b}_i\mcal{H}(\wt{Q}_{i,k}, \wt{P}_{i,k})\cdot Z_{t_k, t_{k+1}}\,,
\end{align*}
where $\dot{\wt{q}}_d(t_k + \eta \Delta t) := \frac{1}{\Delta t}\sum_{\mu=0}^s \wt{q}^\mu \dot{l}_\mu(\eta)$. 

The rough Galerkin integrator defined by the interpolation \eqref{eq:rough q interpolation} and the stationary conditions \eqref{eq:Galerkin stat cond 1} is a general class of Galerkin integrator whose solvability and convergence is unknown in general. We state one case to which is it solvable.
\begin{proposition}\label{prop: solvability Galerkin}
    When $r=s$ and the matrix $A_{i\mu}:=b_i\dot{l}_\mu(c_i)$ is nonsingular with $b_i \neq 0$, the rough Galerkin method defined by \eqref{eq:rough q interpolation}, \eqref{eq:Galerkin stat cond 1} is solvable for sufficiently small $\Delta t$ and $|Z_{t_k, t_{k+1}}|$. Furthermore, $\mcal{S}_d(q_k, p_{k+1})$ is smooth in the arguments $q_k$ and $p_{k+1}$.
\end{proposition}
\begin{proof}
Let $u := \bigl(q^1,\ldots,q^s, P_{1,k},\ldots,P_{r,k}\bigr) \in \mathbb{R}^{2Nr}$ be the vectors of unknowns and $q^0 = q_k$ being prescribed. Substituting \eqref{eq:rough q interpolation} into \eqref{eq:galerkin stat 1} and \eqref{eq:galerkin stat 1}--\eqref{eq:galerkin stat 3} become a system of $2Nr$ equations,
\begin{align*}
    G\left(u;\, q_k, p_{k+1}; \Delta t, Z_{t_k, t_{k+1}}\right) = 0\,,
\end{align*}
in which non-linearities arise exclusively from $H$ and $\mcal{H}$. At $(\Delta t, Z_{t_k, t_{k+1}}) = (0,0)$ the system is linear and decouples between $q^\mu$ and $P_{i,k}$,
\begin{align*}
    \sum_{\mu=1}^s \dot{l}_\mu(c_i)\, b_i\, q^\mu = -\dot{l}_0(c_i)\,b_i\,q_k\,,\quad i = 1,\ldots,s\,,
    \qquad
    \sum_{i=1}^r b_i \dot{l}_\mu(c_i)\, P_{i,k} = \delta_{\mu s}\, p_{k+1}\,,\quad \mu = 1,\ldots,s\,,
\end{align*}
which can be represented as $A_{i\mu}q^\mu = -\dot{l}_0(c_i)b_i q_k$ and $A^T_{\mu i}P_{i,k} = \delta_{\mu s}p_{k+1}$. Since $A_{i, \mu}$ is nonsingular and $b_i \neq 0$, the linear system have unique solution $u_0$. One can readily check that the Jacobian $\nabla_uG|_{u_0}$ is $\operatorname{diag}(-A, -A^T)\otimes I_N$, and the $G$ is smooth in the $q_k$, $p_{k+1}$ arguments due to the  smoothness of $H$ and $\mcal{H}$. Via the implicit function theorem, there exits a neighbourhood around $u_0$ on which, for sufficiently small $\Delta t$ and $Z_{t_k, t_{k+1}}$, the system $G = 0$ has a unique solution $u = u(q_k, p_{k+1}; \Delta t, Z_{t_k, t_{k+1}})$. Additionally, $\mcal{S}_d(q_k, p_{k+1})$ inherits the smoothness of $H$, $\mcal{H}$ composed with the smoothness of $u(q_k, p_{k+1}; \Delta t, Z_{t_k, t_{k+1}})$ in $q_k$ and $p_{k+1}$.
\end{proof}
When the rough Galerkin integrator is solvable, the discrete symplectic mapping \eqref{eq:discrete sym mapping} defines a forward map $(q_k, p_k) \rightarrow (q_{k+1}, p_{k+1})$ through the implicit function theorem that possesses discrete versions of the symplectic $2$-form conservation and the canonical $1$-form conservation laws of the continuous system.
\begin{proposition}\label{prop:discrete 2form conserv}
    The discrete generating function $\mcal{S}_d(q_k, p_{k+1})$ generates a discrete flow that preserves the symplectic form between time steps,
    \begin{align}\label{eq:discrete 2form conserv}
        \rmd p_k \wedge \rmd q_k = \rmd p_{k+1} \wedge \rmd q_{k+1}\,. 
    \end{align}
\end{proposition}
\begin{proof}
    Analogous to the continuous case. Taking the exterior derivative of the discrete generating function $\mcal{S}_d(q_k, p_{k+1})$ and use \eqref{eq:discrete sym mapping} yields the result.
\end{proof}

\begin{proposition}\label{prop:discrete 1 form CL}
    When both Hamiltonians, $H$ and $\mcal{H}$ are linear in $p$, wlog taking the form of \eqref{eq:linear ham}, the rough Galerkin integrator preserves the canonical (Liouville) $1$-form between time steps,
    \begin{align}
        p_k \,\rmd q_k = p_{k+1}\, \rmd q_{k+1}\,. \label{eq:discrete 1 form CL}
    \end{align}
\end{proposition}
\begin{proof}
    For linear Hamiltonians, the stationary condition \eqref{eq:Galerkin stat cond 1} becomes
    \begin{align*}
        - \Delta t b_i\left(\dot{\wt{q}}_d(t_k + c_i\Delta t) - f(\wt{Q}_{i,k})\right) + \ob{b}_i\sigma(\wt{Q}_{i,k})\cdot Z_{t_k, t_{k+1}} = 0\,,\quad i = 1,\ldots,r\,,
    \end{align*}
    where we have substituted the variational derivatives of the Hamiltonian. Multiplying by $\wt{P}_{i,k}$ and summing over $i$, we have that $\mcal{S}_d(q_k, p_{k+1}) = p_{k+1}\wt{q}^s$ as the terms summing over $i$ vanishes. Noting that \eqref{eq:discrete sym mapping} holds, we have that 
    \begin{align*}
        & p_k = \frac{\p \mcal{S}_d}{\p q_k} = p_{k+1}\frac{\p \wt{q}^s}{\p q_k} = p_{k+1}\frac{\p q_{k+1}}{\p q_k} \quad \Longrightarrow \quad p_k\, \rmd q_k = p_{k+1}\frac{\p q_{k+1}}{\p q_k} \,\rmd q_k = p_{k+1} \rmd q_{k+1}\,,
    \end{align*}
    where in the last equality we have used the fact that $q_{k+1} = q_{k+1}(q_k)$ is independent of $p_{k+1}$ for linear Hamiltonians.
\end{proof}

We extend the Galerkin approximation of the controlled path $(q, p)$ to a $\tau\in [0,1)$-parameterised family of controlled paths $(\wh{q}, \wh{p}) : [0,1) \rightarrow \mcal{D}^{2\alpha}_Z(T^*Q)$ to have the finite dimensional representation
\begin{align}
    q^\tau_d(t_k+\eta\Delta t) = \sum_{\mu=0}^s q^\mu_\tau l_\mu(\eta)\,,\quad \dot{q}^\tau_d(t_k+\eta\Delta t) = \frac{1}{\Delta t}\sum_{\mu=0}^s q^\mu_\tau \dot{l}_\mu(\eta)\,,
\end{align}
where we have used the same basis polynomials $l_\mu$ and $q^\mu_\tau := \wh{q}(t_k + d_\mu \Delta t)$ are the $\tau$ parameterised control values. We assume that the parameterisation is such that when $\tau = 0$, $(\wh{q}, \wh{p}) = (q,p)$ and we define Galerkin approximation to the variational paths $\delta q$ as well as its time derivative by
\begin{align}
     \delta q_d(t_k + \eta\Delta t) := \delta^\tau q^\tau_d(t_k + \eta\Delta t) = \sum_{\mu=0}^s\delta q^\mu l_\mu(\eta)\,, \quad \frac{d}{dt}\delta q_d(t_k + \eta \Delta t) = \frac{1}{\Delta t}\sum_{\mu=0}^s \delta q^\mu \dot{l}_\mu(\eta)\,.
\end{align}
where $\delta q^\mu := \delta^\tau \wh{q}(t_k + d_\mu\Delta t)$ are the control values. Additionally, we define the control values at quadrature points $\{c_i\}_{i=1}^r$ with $P^\tau_{i,k} := \wh{p}(t_k + c_i\Delta t)$ and $Q^\tau_{i,k} := \wh{q}(t_k + c_i\Delta t)$. Using the definition of $\mcal{S}_d$, for arbitrary $\wh{q}_k, \wh{p}_k, \wh{q}_{k+1},\wh{p}_{k+1} \in C^\infty([0,1], \mathbb{R}^N)$, we have that
\begin{align}
    \mcal{S}_d(\wh{q}_k, \wh{p}_{k+1})&= \wh{p}_{k+1}\wt{q}^s_\tau - \Delta t \sum_{i=1}^r b_i\left( \wt{P}^\tau_{i,k} \dot{\wt{q}}^\tau_d(t_k + c_i \Delta t) - H(\wt{Q}^\tau_{i,k}, \wt{P}^\tau_{i,k})\right) + \sum_{i=1}^r \ob{b}_i\mcal{H}(\wt{Q}^\tau_{i,k}, \wt{P}^\tau_{i,k})\cdot Z_{t_k, t_{k+1}}\,,
\end{align}
where $\{\wt{q}^\mu_\tau\}_{\mu=0}^s$ and $\{\wt{P}^\tau_{i,k}\}_{i=1}^r$ satisfy \eqref{eq:Galerkin stat cond 1} for all $\tau \in [0,1]$ with the appropriate replacement of $q_{k+1}$ and $p_k$ with $\wh{q}_{k+1}$ and $\wh{p}_k$ respectively. Applying the operator $\delta^\tau$ to \eqref{eq:Galerkin stat cond 1} reveals the rough Galerkin integrator of the linearised variables,
\begin{subequations}
    \label{eq:lin galerkin stat cond}
    \begin{align}
        & - \Delta t b_i\left(\frac{d}{dt}\delta q_d(t_k + c_i\Delta t) - \frac{\p^2 H}{\p p \p p}(Q_{i,k}, P_{i,k})\delta P_{i,k} - \frac{\p^2 H}{\p p \p q}(Q_{i,k}, P_{i,k})\delta Q_{i,k}\right) \notag\\
        & \qquad \qquad + \ob{b}_i\left(\frac{\p^2 \mcal{H}}{\p p \p p}(Q_{i,k}, P_{i,k})\delta P_{i,k} + \frac{\p^2 \mcal{H}}{\p p \p q}(Q_{i,k}, P_{i,k})\delta Q_{i,k}\right) \cdot Z_{t_k, t_{k+1}} = 0\,,\quad i = 1,\ldots,r\,,\\
        & -  \sum_{i=1}^r b_i\left( \delta P_{i,k} \dot{l}_\mu(c_i) - \Delta t \left(\frac{\p^2 H}{\p q\p q}(Q_{i,k}, P_{i,k})\delta Q_{i,k} + \frac{\p^2 H}{\p q\p p}(Q_{i,k}, P_{i,k})\delta P_{i,k}\right)l_\mu(c_i)\right) \notag\\
        &\qquad \qquad + \sum_{i=1}^r \ob{b}_i\left(\frac{\p^2 \mcal{H}}{\p q \p q}(Q_{i,k}, P_{i,k})\delta Q_{i,k} + \frac{\p^2 \mcal{H}}{\p q \p p}(Q_{i,k}, P_{i,k})\delta P_{i,k}\right)l_\mu(c_i)\cdot Z_{t_k, t_{k+1}} = 0\,, \quad \mu=1,\ldots,s-1\,,\\
        &\delta p_{k+1} - \sum_{i=1}^r b_i\left( \delta P_{i,k} \dot{l}_s(c_i) - \Delta t \left(\frac{\p^2 H}{\p q\p q}(Q_{i,k}, P_{i,k})\delta Q_{i,k} + \frac{\p^2 H}{\p q\p p}(Q_{i,k}, P_{i,k})\delta P_{i,k}\right)l_s(c_i)\right)\notag \\
        &\qquad \qquad + \sum_{i=1}^r \ob{b}_i\left(\frac{\p^2 \mcal{H}}{\p q \p q}(Q_{i,k}, P_{i,k})\delta Q_{i,k} + \frac{\p^2 \mcal{H}}{\p q \p p}(Q_{i,k}, P_{i,k})\delta P_{i,k}\right)l_s(c_i)\cdot Z_{t_k, t_{k+1}} = 0\,,\\
        &\delta p_k = - \sum_{i=1}^r b_i\left( \delta P_{i,k} \dot{l}_0(c_i) - \Delta t \left(\frac{\p^2 H}{\p q\p q}(Q_{i,k}, P_{i,k})\delta Q_{i,k} + \frac{\p^2 H}{\p q\p p}(Q_{i,k}, P_{i,k})\delta P_{i,k}\right)l_0(c_i)\right) \notag\\
        & \qquad \qquad + \sum_{i=1}^r \ob{b}_i\left(\frac{\p^2 \mcal{H}}{\p q \p q}(Q_{i,k}, P_{i,k})\delta Q_{i,k} + \frac{\p^2 \mcal{H}}{\p q \p p}(Q_{i,k}, P_{i,k})\delta P_{i,k}\right)l_0(c_i)\cdot Z_{t_k, t_{k+1}} \,,\\
        &\delta q_{k+1} = \delta q^s\,. 
\end{align}    
\end{subequations}
Here, $\delta Q_{i,k} := \delta q_d(t_k + c_i\Delta t)$ and $\delta P_{i,k} := \delta^\tau P^\tau_{i,k}$, the latter obtained by differentiating the independent stage control $P^\tau_{i,k}$, since no polynomial representation $p_d$ of the momentum was introduced. Having established the rough Galerkin method of the linearised variables, we have the following proposition
\begin{proposition}\label{prop:numerical 1-form conservation}
    When both Hamiltonians, $H$ and $\mcal{H}$ are affine in $p$, wlog taking the form of \eqref{eq:affine hams}, the rough Galerkin integrator \eqref{eq:Galerkin stat cond 1} together with associated integrator \eqref{eq:lin galerkin stat cond} for the linearised quantities, possess the following discrete quasi-conservation law,
    \begin{align}\label{eq:numerical 1-form conservation}
        \scp{p_k}{\delta q_k} = \scp{p_{k+1}}{\delta q_{k+1}} + \Delta t\sum_{i=1}^r b_i \scp{\frac{\p L}{\p q}(\wt{Q}_{i,k})}{\delta \wt{Q}_{i,k}} + \sum_{i=1}^r \ob{b}_i\scp{\frac{\p \mathfrak{L}}{\p q}(\wt{Q}_{i,k})}{\delta \wt{Q}_{i,k}}\cdot Z_{t_k, t_{k+1}}\,.
    \end{align}
    As a special case, when $L(q) = \mathfrak{L}(q) = 0$, we have the conservation law $\scp{p_{k+1}}{\delta q_{k+1}} = \scp{p_k}{\delta q_k}$.
\end{proposition}
\begin{proof}
Given some arbitrary constants $p_{k+1}, p_k \in \mathbb{R}^N$ and $\tau$-parameterised vectors $\wh{q}_k, \wh{q}_{k+1} \in C^\infty([0,1], \mathbb{R}^N)$, for affine Hamiltonians, where the control points $\{\wt{q}^\mu_\tau\}_{\mu=0}^s$ and $\{\wt{P}^\tau_{i,k}\}_{i=1}^r$ satisfy the conditions \eqref{eq:Galerkin stat cond 1}, the discrete generating function becomes
    \begin{align*}
    \mcal{S}_d(\wh{q}_k, p_{k+1})&= p_{k+1}\wt{q}^s_\tau + \Delta t \sum_{i=1}^r b_i L(\wt{Q}^\tau_{i,k}) + \sum_{i=1}^r \ob{b}_i\mathfrak{L}(\wt{Q}^\tau_{i,k})\cdot Z_{t_k, t_{k+1}}\,,
    \end{align*}
    Taking $\delta^\tau$ gives
    \begin{align*}
        \begin{split}
            \delta^\tau \mcal{S}_d(\wh{q}_k, p_{k+1}) &= p_{k+1}\delta^\tau \wt{q}^s_\tau + \Delta t\sum_{i=1}^rb_i\scp{\frac{\p L}{\p q}(\wt{Q}_{i,k})}{\delta \wt{Q}_{i,k}} +  \sum_{i=1}^r \ob{b}_i\scp{\frac{\p \mathfrak{L}}{\p q}(\wt{Q}_{i,k})}{\delta \wt{Q}_{i,k}}\cdot Z_{t_k, t_{k+1}}\\
            &= \scp{p_{k+1}}{\delta q_{k+1}} + \Delta t\sum_{i=1}^rb_i\scp{\frac{\p L}{\p q}(\wt{Q}_{i,k})}{\delta \wt{Q}_{i,k}} + \sum_{i=1}^r \ob{b}_i\scp{\frac{\p \mathfrak{L}}{\p q}(\wt{Q}_{i,k})}{\delta \wt{Q}_{i,k}}\cdot Z_{t_k, t_{k+1}}\,,\\
        \end{split}
    \end{align*}
    Using the implicit symplectic mapping, we have that 
    \begin{align*}
        \delta^\tau \mcal{S}_d(\wh{q}_k, p_{k+1}) = \scp{\frac{\p \mcal{S}_d}{\p \wh{q}_k}\bigg|_{\tau = 0}}{\delta^\tau \wh{q}_k} = \scp{p_k}{\delta q_k}\,,
    \end{align*}
    Equating the derivatives yields the result.
\end{proof}

\subsection{Rough symplectic partitioned Runge--Kutta method (RSPRK)}\label{subsec:rough RK}
The rough Galerkin method defined by the extremum conditions \eqref{eq:Galerkin stat cond 1} is equivalent to a rough analogue of symplectic partitioned Runge--Kutta methods of collocation type under the solvability conditions given in Proposition \ref{prop: solvability Galerkin}. In particular, 
\begin{subequations}\label{eq:RSPRK}
\begin{align}
    Q_{i,k} &= q_k + \Delta t \sum_{j=1}^s a_{ij} \frac{\p H}{\p p}(Q_{j,k}, P_{j,k}) + \sum_{j=1}^s \ob{a}_{ij}\frac{\p \mcal{H}}{\p p}(Q_{j,k}, P_{j,k})\cdot Z_{t_k, t_{k+1}}\,,\quad \forall i = 1,\ldots,s\,,\label{eq:RSPRK 1}\\
    P_{i,k} &= p_{k+1} + \Delta t \sum_{j=1}^s \beta_{ij} \frac{\p H}{\p q}(Q_{j,k}, P_{j,k}) + \sum_{j=1}^s \ob{\beta}_{ij}\frac{\p \mcal{H}}{\p q}(Q_{j,k}, P_{j,k})\cdot Z_{t_k, t_{k+1}}\,,\quad \forall i = 1,\ldots,s\,,\label{eq:RSPRK 2}\\
    q_{k+1} &= q_k + \Delta t \sum_{i=1}^s b_i \frac{\p H}{\p p}(Q_{i,k}, P_{i,k}) + \sum_{i=1}^s \ob{b}_i\frac{\p \mcal{H}}{\p p}(Q_{i,k}, P_{i,k})\cdot Z_{t_k, t_{k+1}}\,,\label{eq:RSPRK 3}\\
    p_k &= p_{k+1} + \Delta t \sum_{i=1}^s b_i \frac{\p H}{\p q}(Q_{i,k}, P_{i,k}) + \sum_{i=1}^s \ob{b}_i\frac{\p \mcal{H}}{\p q}(Q_{i,k}, P_{i,k})\cdot Z_{t_k, t_{k+1}}\,,\label{eq:RSPRK 4}
\end{align}
\end{subequations}
where the coefficients $\ob{a}_{ij}$, $\beta_{ij}$ and $\ob{\beta}_{ij}$ are defined from $a_{ij}$, $b_i$ and $\ob{b}_i$ as 
\begin{align}
    \ob{a}_{ij} = \frac{\ob{b}_j a_{ij}}{b_j}\,,\quad  \beta_{ij} = \frac{b_j a_{ji}}{b_i}\,,\quad \ob{\beta}_{ij} = \frac{\ob{b}_ja_{ji}}{b_i}\,,\label{eq:RSPRK coeff conds}
\end{align}
assuming that $b_i$, $\ob{b}_i$ are non-zero and the coefficients $a_{ij}$, $b_i$ and $\ob{b}_i$ are defined as follows. Let $r = s$ and let $\{\ob{l}_i\}_{i=1}^s$ be degree $s-1$ Lagrange polynomials defined on the control points $\{c_i\}_{i=1}^s$ such that $\ob{l}_i(c_j) = \delta_{ij}$ with the properties
\begin{align}\label{eq:collocation RK coeff def}
    \int_0^1 \ob{l}_i(\eta)\,d\eta = b_i\,,\quad \int_0^{c_i} \ob{l}_j(\eta)\,d\eta = a_{ij} \,, \quad \forall i,j=1,\ldots,s\,.
\end{align}
Consider the Galerkin discretisation of $\dot{q}$ in \eqref{eq:rough q interpolation}. Since $\dot{q}_d$ is a polynomial of degree $s-1$ as $\dot{l}_\mu$ are of degree $s-1$, there exists a unique interpolation by the polynomials $\{\ob{l}_i\}_{i=1}^s$,
\begin{align}
    \dot{q}_d(t_k + \eta \Delta t) =\sum_{i=1}^s \dot{q}_d(t_k + c_i \Delta t)\ob{l}_i(\eta)\,.
\end{align}
We remark that whilst $\dot{q}$ does not formally exist in the rough setting, we are instead interpolating the finite dimensional representation $\dot{q}_d$, obtained by taking derivatives of $q_d$, by the new basis elements $\ob{l}_i$. Integrating, we obtain
\begin{align*}
    q_d(t_k + \eta \Delta t) = q_d(t_k) + \Delta t \int_0^{\eta} \sum_{i=1}^s \dot{q}_d(t_k+c_i\Delta t)\ob{l}_i(\eta')\,d\eta' = q_d(t_k) +  \Delta t \sum_{i=1}^s \dot{q}_d(t_k+c_i\Delta t)\int_0^{\eta}\ob{l}_i(\eta')\,d\eta'
\end{align*}
For $\eta = 1$, note that by assumption $\int_0^1\ob{l}_i(\eta)\,d\eta = b_i$, $q_d(t_k) := q_k$ and $q_d(t_k + \Delta t) := q_{k+1}$, substituting the relation \eqref{eq:galerkin stat 1} yields \eqref{eq:RSPRK 3}. We obtain \eqref{eq:RSPRK 1} by setting $\eta = c_i$ for $i=1,\ldots, s$. To obtain \eqref{eq:RSPRK 4}, note that $\sum_{\mu=0}^sl_\mu(\eta) = 1$ such that $\sum_{\mu=0}^s \dot{l}(\eta) = 0$, summing \eqref{eq:galerkin stat 2}-\eqref{eq:galerkin stat 4} yields the result. To obtain \eqref{eq:RSPRK 2}, define the quantity $\{m_{j,\mu}\}_{\mu=0}^s$, $j=1,\ldots,s$ by
\begin{align*}
    \int_0^\eta \ob{l}_j(\eta')\,d\eta' - b_j = \sum_{\mu=0}^s m_{j,\mu}l_\mu(\eta)\,.
\end{align*}
As $\ob{l}_j$ are degree $s-1$ polynomials, the LHS of the previous expression is a degree $s$ polynomials and the coefficients $m_{j,\mu}$ are uniquely defined. Evaluating at $\eta= 0$, $\eta=1$, $\eta = c_i$ and differentiating, we obtain 
\begin{align*}
    &\eta = 0\,\Longrightarrow\, -b_j = \sum_{\mu=0}^s m_{j,\mu}l_\mu(0) = m_{j,0}\,,\qquad \eta = 1\,\Longrightarrow\, 0 = \sum_{\mu=0}^s m_{j,\mu}l_\mu(1) = m_{j,s}\,,\\
    &\eta = c_i\,\Longrightarrow\, a_{ij} - b_j = \sum_{\mu=0}^s m_{j,\mu}l_\mu(c_i)\,,\qquad \ob{l}_j(\eta) = \sum_{\mu=0}^s m_{j,\mu}\dot{l}_\mu(\eta)\,.
\end{align*}
For each $j \in [1,\ldots, s]$, multiply \eqref{eq:galerkin stat 2} by $m_{j,\mu}$ for all $\mu$, \eqref{eq:galerkin stat 3} by $m_{j,s}$, \eqref{eq:galerkin stat 4} by $m_{j,0}$ then summing, we obtain
\begin{align*}
    0&=p_{k+1}m_{j,s} - p_k m_{j,0} - \sum_{\mu=0}^s\sum_{i=1}^s b_i m_{j,\mu}\left(P_{i,k}\dot{l}_\mu(c_i) - \Delta t \frac{\p H}{\p q}(Q_{i,k}, P_{i,k})l_\mu(c_i)\right) \\
    &\hspace{20em} + \sum_{\mu=0}^s\sum_{i=1}^s \ob{b}_i m_{j,\mu}\frac{\p \mcal{H}}{\p q}(Q_{i,k}, P_{i,k})l_\mu(c_i)\cdot Z_{t_k, t_{k+1}}\\
    & = - p_k m_{j,0}- \sum_{i=1}^s b_i \left(P_{i,k}\ob{l}_j(c_i) - \Delta t \frac{\p H}{\p q}(Q_{i,k}, P_{i,k})(a_{ij}-b_j)\right) + \sum_{i=1}^s \ob{b}_i \frac{\p \mcal{H}}{\p q}(Q_{i,k}, P_{i,k})(a_{ij}-b_j)\cdot Z_{t_k, t_{k+1}}\\
    & = p_k b_j - P_{j,k}b_j + \sum_{i=1}^sb_i \Delta t\frac{\p H}{\p q}(Q_{i,k},P_{i,k})(a_{ij} - b_j) + \sum_{i=1}^s \ob{b}_i \frac{\p \mcal{H}}{\p q}(Q_{i,k}, P_{i,k})(a_{ij}-b_j)\cdot Z_{t_k, t_{k+1}}\,,
\end{align*}
Substituting in the equations for $p_k$ in terms of $p_{k+1}$ \eqref{eq:RSPRK 4} yields \eqref{eq:RSPRK 2}. 

We remark that particular cases of RSPRK methods have been considered before. E.g., in \cite{HHW2018}, the case of Gaussian rough path as the driving geometric rough path was investigated. In \cite{HT2018}, the case of Stratonovich Brownian motion was considered. 

Similarly to the case where the rough Galerkin method for the state variables $(q,p)$ is equivalent to the RSPRK method \eqref{eq:RSPRK}, the rough Galerkin method for the variational (linearised) variables $(\delta q, \delta p)$ can also be cast into a RSPRK method that is coupled to the solution of the state variables. The construction from is effectively the same and we obtain
\begin{subequations}\label{eq: var RSPRK}
    \begin{align}
    \delta Q_{i,k} &= \delta q_k + \Delta t \sum_{j=1}^s a_{ij} \left(\frac{\p^2H}{\p p \p p}(Q_{j,k}, P_{j,k})\delta P_{j,k} + \frac{\p^2H}{\p p \p q}(Q_{j,k}, P_{j,k})\delta Q_{j,k} \right) \notag\\
    & \qquad \qquad + \sum_{j=1}^s \ob{a}_{ij}\left(\frac{\p^2\mcal{H}}{\p p \p p}(Q_{j,k}, P_{j,k})\delta P_{j,k} + \frac{\p^2\mcal{H}}{\p p \p q}(Q_{j,k}, P_{j,k})\delta Q_{j,k} \right)\cdot Z_{t_k, t_{k+1}}\,,\quad \forall i = 1,\ldots,s\,,\\
    \delta P_{i,k} &= \delta p_{k+1} + \Delta t \sum_{j=1}^s \beta_{ij} \left(\frac{\p^2 H}{\p q \p p}(Q_{j,k}, P_{j,k})\delta P_{j,k} + \frac{\p^2 H}{\p q \p q}(Q_{j,k}, P_{j,k})\delta Q_{j,k}\right) \notag\\
    &\qquad \qquad + \sum_{j=1}^s \ob{\beta}_{ij}\left(\frac{\p^2 \mcal{H}}{\p q \p p}(Q_{j,k}, P_{j,k})\delta P_{j,k} + \frac{\p^2 \mcal{H}}{\p q \p q}(Q_{j,k}, P_{j,k})\delta Q_{j,k}\right)\cdot Z_{t_k, t_{k+1}}\,,\quad \forall i = 1,\ldots,s\,,\\
    \delta q_{k+1} &= \delta q_k + \Delta t \sum_{i=1}^s b_i \left(\frac{\p^2H}{\p p \p p}(Q_{i,k}, P_{i,k})\delta P_{i,k} + \frac{\p^2H}{\p p \p q}(Q_{i,k}, P_{i,k})\delta Q_{i,k} \right) \notag\\ 
    &\qquad \qquad + \sum_{i=1}^s \ob{b}_i\left(\frac{\p^2\mcal{H}}{\p p \p p}(Q_{i,k}, P_{i,k})\delta P_{i,k} + \frac{\p^2\mcal{H}}{\p p \p q}(Q_{i,k}, P_{i,k})\delta Q_{i,k} \right)\cdot Z_{t_k, t_{k+1}}\,,\\
    \delta p_k &= \delta p_{k+1} + \Delta t \sum_{i=1}^s b_i \left(\frac{\p^2 H}{\p q \p p}(Q_{i,k}, P_{i,k})\delta P_{i,k} + \frac{\p^2 H}{\p q \p q}(Q_{i,k}, P_{i,k})\delta Q_{i,k}\right) \notag\\
    &\qquad \qquad + \sum_{i=1}^s \ob{b}_i\left(\frac{\p^2 \mcal{H}}{\p q \p p}(Q_{i,k}, P_{i,k})\delta P_{i,k} + \frac{\p^2 \mcal{H}}{\p q \p q}(Q_{i,k}, P_{i,k})\delta Q_{i,k}\right)\cdot Z_{t_k, t_{k+1}}\,,
    \end{align}
\end{subequations}
where the RK coefficients are defined previously and the quantities $(Q_{i,k}, P_{i,k})$ are the solutions obtained from the RSPRK method \eqref{eq:RSPRK}. 

As the class of RSPRK methods of collocation type defined by \eqref{eq:RSPRK} are equivalent to the rough Galerkin methods defined by \eqref{eq:rough q interpolation} and \eqref{eq:Galerkin stat cond 1}, the discrete conservation laws proposed in Propositions \eqref{prop:discrete 2form conserv} -- \eqref{prop:numerical 1-form conservation} are preserved by RSPRK methods of collocation type whose coefficients are \emph{defined} by \eqref{eq:collocation RK coeff def}. Nevertheless, the discrete conservation laws in Propositions \eqref{prop:discrete 2form conserv} -- \eqref{prop:numerical 1-form conservation} holds for \emph{all} RSPRK methods in the form \eqref{eq:RSPRK} whose coefficients satisfy the classic symplecticity conditions 
\begin{align}
\begin{split}\label{eq:sym cond 2}
    b_i\alpha_{ij} + b_ja_{ji} = b_ib_j\,,\quad 
    \ob{b}_i\alpha_{ij} + b_j\ob{a}_{ji} = \ob{b}_ib_j\,,\\
    b_i\ob{\alpha}_{ij} + \ob{b}_ja_{ji} = b_i\ob{b}_j\,,\quad 
    \ob{b}_i\ob{\alpha}_{ij} + \ob{b}_j\ob{a}_{ji} = \ob{b}_i\ob{b}_j\,,
\end{split}
\end{align}
where $\alpha_{ij} := b_j - \beta_{ij}$ and $\ob{\alpha}_{ij} := \ob{b}_j - \ob{\beta}_{ij}$. One can easily verify \eqref{eq:RSPRK coeff conds} is equivalent to the above symplectic conditions. 
\begin{proposition}\label{Prop:RSPRK conservation laws}
    The RSPRK method in the form \eqref{eq:RSPRK} whose coefficient satisfy \eqref{eq:RSPRK coeff conds} possess the following conservation laws. For all smooth Hamiltonians, 
    \begin{align}
        \rmd p_k \wedge \rmd q_k = \rmd p_{k+1} \wedge \rmd q_{k+1}\,.
    \end{align}
     For linear in $p$ Hamiltonians, wlog taking the form of \eqref{eq:linear ham},
    \begin{align}
        p_k \,\rmd q_k = p_{k+1}\, \rmd q_{k+1}\,. 
    \end{align}
     For affine in $p$ Hamiltonians, wlog taking the form of \eqref{eq:affine hams},
     \begin{align}
        \scp{p_k}{\delta q_k} = \scp{p_{k+1}}{\delta q_{k+1}} + \Delta t\sum_{i=1}^r b_i \scp{\frac{\p L}{\p q}(Q_{i,k})}{\delta Q_{i,k}} + \sum_{i=1}^r \ob{b}_i\scp{\frac{\p \mathfrak{L}}{\p q}(Q_{i,k})}{\delta Q_{i,k}}\cdot Z_{t_k, t_{k+1}}\,.
    \end{align}
\end{proposition}
\begin{proof}
    The direct proofs are cumbersome but straightforward and follows the same logic which we will not write down explicitly. For each conservation law, one expand the terms $\rmd p_{k+1}$, $\rmd q_{k+1}$, $\delta q_{k+1}$ and $p_{k+1}$ in terms of $\rmd p_k$, $\rmd q_k$, $\delta q_k$ and $\delta p_k$ using the RSPRK method definitions \eqref{eq:RSPRK} as its variational equation \eqref{eq: var RSPRK}. Then, substitute in for $\rmd p_k$, $\rmd q_k$, $\delta q_k$ and $\delta p_k$ in the expansion of $\rmd p_{k+1}$, $\rmd q_{k+1}$, $\delta q_{k+1}$ and $p_{k+1}$ using the equations for the internal stages $P_{i,k}$, $Q_{i,k}$, $\delta P_{i,k}$ and $\delta Q_{i,k}$. Simplifying and using the symplecticity conditions \eqref{eq:sym cond 2} one can obtain the required results.
    
    For the first conservation law, See \cite[Thm. 4.1]{HHW2018} for the same proof with Gaussian rough path and the proof for arbitrary geometric rough path is essentially the same. 
\end{proof}

We present several concrete examples of the RSPRK methods by restricting to the collocation points for integration against $t$ and $\mathbf{Z}$ to be the same, that is, $b_i = \ob{b}_i$ for all $i = 1,\ldots, s$. Then, we have that $\ob{a}_{ij} = a_{ij}$ and $\beta_{ij} = \ob{\beta}_{ij}$. A classical choice of letting $s = 1$,
\begin{align}
\begin{array}{c|c}
\mathbf{c}^T  & \mathbf{a} \\
\hline
  & \mathbf{b}    
\end{array} 
\quad  =  \quad 
\begin{array}{c|c}
1/2  & 1/2     \\
\hline
  & 1    
\end{array}\quad  \text{and} \quad 
\begin{array}{c|c}
\mathbf{c}^T  & \bs{\beta} \\
\hline
  & \mathbf{b}    
\end{array} 
\quad  =  \quad 
\begin{array}{c|c}
1/2  & 1/2     \\
\hline
  & 1  
\end{array}\,,\label{eq:RPSPRK IM}
\end{align}
gives the implicit midpoint (IM) method for both $q$ and $p$ which is trivially checked to satisfy the conditions \eqref{eq:RSPRK coeff conds}. Restricting the consideration to adjoint systems, we consider the case where $H$ and $\mcal{H}$ are affine in $p$. We note that when $a_{ij}$ is strictly lower triangular, the relations $\beta_{ij} = b_ja_{ji}/b_i$ in \eqref{eq:RSPRK coeff conds} implies that $\beta_{ij}$ is strictly upper triangular. Thus, an explicit scheme of the $q$ dynamics gives an explicit scheme for the $p$ dynamics in reverse time. We give three examples of RSPRK methods that are useful in practice for adjoint systems. Let the $q$ dynamics be solved via Heun's method (the explicit trapezoidal rule, also called the improved Euler method), which is the method referred to as RK$2$ in Section \ref{sec:examples} and in the legends of the figures there. Then, we have the following relation between the Butcher tableau of the rough adjoint system,
\begin{align}
\begin{array}{c|c}
\mathbf{c}^T  & \mathbf{a} \\
\hline
  & \mathbf{b}    
\end{array} 
\quad  =  \quad 
\begin{array}{c|cc}
0  & 0  & 0  \\
1  & 1  & 0   \\
\hline
  & 1/2 &1/2   
\end{array}\quad  \Longrightarrow \quad 
\begin{array}{c|c}
\mathbf{c}^T  & \bs{\beta} \\
\hline
  & \mathbf{b}    
\end{array} 
\quad  =  \quad 
\begin{array}{c|cc}
0  & 0  & 1  \\
1  & 0  & 0   \\
\hline
  & 1/2 &1/2   
\end{array}\,.\label{eq:RPSPRK rk2}
\end{align}
In the same spirit, let the $q$ dynamics be solved via the standard RK$4$ method whose coefficients are represented by the Butcher tableau, then
\begin{align}
\begin{array}{c|c}
\mathbf{c}^T  & \mathbf{a} \\
\hline
  & \mathbf{b}    
\end{array} 
\quad  =  \quad 
\begin{array}{c|cccc}
0  & 0  & 0  & 0 & 0\\
1/2  & 1/2  & 0  & 0 & 0 \\
1/2  & 0  & 1/2 & 0 & 0  \\
1  & 0  & 0 & 1 & 0 \\
\hline
  & 1/6 &1/3 &1/3   & 1/6    
\end{array}\quad  \Longrightarrow \quad 
\begin{array}{c|c}
\mathbf{c}^T  & \bs{\beta} \\
\hline
  & \mathbf{b}    
\end{array} 
\quad  =  \quad 
\begin{array}{c|cccc}
0  & 0  & 1  & 0 & 0\\
1/2  & 0  & 0  & 1/2 & 0 \\
1/2  & 0  & 0 & 0 & 1/2  \\
1  & 0  & 0 & 0 & 0 \\
\hline
  & 1/6 &1/3 &1/3   & 1/6   
\end{array}\,.\label{eq:RPSPRK rk4}
\end{align}
Thus, if one uses Heun's method or the RK$4$ scheme forward in time for the $q$ dynamics and use the same scheme in reverse time for the $p$ dynamics with suitable reversal of the time dependent driving rough path increments as well as evaluation of the $q$ states to ensure symplecticity. However, the adjoint butcher tableau coinciding with the forward butcher tableau is not true in general. Consider the case where the $q$ dynamics is solved via the $3^{rd}$-order Strong Stability Preserving Runge--Kutta scheme (SSPRK(3,3)) \cite{Shu1988}. Then, we have
\begin{align}\begin{array}{c|c}
\mathbf{c}^T  & \mathbf{a} \\
\hline
  & \mathbf{b}    
\end{array} 
\quad  =  \quad 
\begin{array}{c|ccc}
0  & 0  & 0  & 0  \\
1  & 1  & 0 & 0  \\
1/2  & 1/4  & 1/4 & 0  \\
\hline
  & 1/6 &1/6 &2/3    
\end{array}\quad  \Longrightarrow \quad 
\begin{array}{c|c}
\mathbf{c}^T  & \bs{\beta} \\
\hline
  & \mathbf{b}    
\end{array} 
\quad  =  \quad 
\begin{array}{c|ccc}
0 & 0  & 1  & 1 \\
1  & 0  & 0 & 1  \\
1/2  & 0  & 0 & 0  \\
\hline
   & 1/6 &1/6 &2/3  
\end{array}\,. \label{eq:RSPRK SSP}
\end{align}

\paragraph{Rate of convergence.}\label{subsec:convergence}
In the discrete action \eqref{eq:discrete action}, only the increment $ Z_{t_k, t_{k+1}}$ of the driving rough path $\mathbf{Z}$ is used. This construction presents an order barrier on the rate of convergence of the resulting Galerkin scheme \eqref{eq:Galerkin stat cond 1}, and subsequently, the rate of convergence of equivalent RSPRK method \eqref{eq:RSPRK}. To demonstrate this, we follow the constructions presented in \cite{RR2020} for rough Runge--Kutta methods and extend it to RSPRK method. For arbitrary fixed $a,b \in \mathbb{R}^N$, consider the RDE
\begin{align}
    \begin{split}
        dx_t = f(x_t, y_t)d\mathbf{Z}_t\,,\quad dy_t = g(x_t, y_t)d\mathbf{Z}_t\,, \quad x_0 = a\,,\quad y_0 = b\,.
    \end{split}\label{eq:2 component RDE general}
\end{align}
where $\mathbf{Z}\in \mcal{C}^\alpha_g([t_0, t_1],\mathbb{R}^K)$, $x, y \in \mcal{D}_Z^{2\alpha}(\mathbb{R}^N)$ and $f, g \in C^\infty(\mathbb{R}^N\times \mathbb{R}^N, \mcal{L}(\mathbb{R}^K,\mathbb{R}^N))$. Associated with \eqref{eq:2 component RDE general}, we have the differential equation driven by a smoothed signal $Z^{\Delta t}$ defined in \eqref{eq:piecewise smooth path},
\begin{align}
    \begin{split}
        dx^{\Delta t}_t = f(x^{\Delta t}_t, y^{\Delta t}_t)dZ^{\Delta t}_t\,,\quad dy^{\Delta t}_t = g(x^{\Delta t}_t, y^{\Delta t}_t)d Z^{\Delta t}_t\,, \quad x^{\Delta t}_0 = a\,,\quad y^{\Delta t}_0 = b.
    \end{split}\label{eq:2 component smooth ODE general}
\end{align}
Consider a $s$-stage partitioned RK method defined through a pair of RK coefficients, $(a_{ij}, b_i)$ and $(\alpha_{ij}, \ob{b}_i)$,
\begin{align}
    \begin{split}
    X^{\Delta t}_{i,k} &= x^{\Delta t}_k + Z^{\Delta t}_{t_{k+1}, {t_k}}\sum_{j=1}^s a_{ij} f(X^{\Delta t}_{j,k}, Y^{\Delta t}_{j,k})\,,\quad \forall i = 1,\ldots,s\,,\\
    Y^{\Delta t}_{i,k} &= y^{\Delta t}_k + Z^{\Delta t}_{t_{k+1}, {t_k}}\sum_{j=1}^s \alpha_{ij} g(X^{\Delta t}_{j,k}, Y^{\Delta t}_{j,k})\,,\quad \forall i = 1,\ldots,s\,,\\
    x^{\Delta t}_{k+1} &= x^{\Delta t}_k + Z^{\Delta t}_{t_{k+1}, {t_k}}\sum_{i=1}^s b_i f(X^{\Delta t}_{i,k}, Y^{\Delta t}_{i,k})\,,\\
    y^{\Delta t}_{k+1} &= y^{\Delta t}_k + Z^{\Delta t}_{t_{k+1}, {t_k}}\sum_{i=1}^s \ob{b}_i g(X^{\Delta t}_{i,k}, Y^{\Delta t}_{i,k})\,.
    \end{split} \label{eq:PRK general}
\end{align}
We remark that \eqref{eq:2 component RDE general} is driven by $\mathbf{Z}$ alone without drift. To include explicit drift terms, one may simply extend the rough path $\mathbf{Z}$ to the canonical lift of the path $\wh{Z} = (t, Z)$ and replace the occurrences of $\mathbf{Z}$ and $Z^{\Delta t}_t$ with $\mathbf{\wh{Z}}$ and $\wh{Z}^{\Delta t}_t := (t, Z^{\Delta t}_t)^T$, respectively. Since the time component is represented exactly by its piecewise linear interpolant, we have the same convergence rates from the canonical lift of $\mathbf{\wh{Z}}^{\Delta t}$ to $\mathbf{\wh{Z}}$ under the rough path metric. 

Define $(x^{\Delta t}_k(a, b), y^{\Delta t}_k(a, b))$ as the result of iterating the above partitioned RK method to $(x^{\Delta t}_0, y^{\Delta t}_0) = (a, b)$ for $k$ steps. Let $(x^{\Delta t}_{t_k}(a,b), y^{\Delta t}_{t_k}(a,b))$ to be the solution to the differential equation \eqref{eq:2 component smooth ODE general} at $t = t_k = t_0 + k \Delta t$ where $(x^{\Delta t}_{t_0}, y^{\Delta t}_{t_0}) = (a, b)$. Let $e_k(a, b, h)$ be the error of the partitioned RK method at step $k$,
\begin{align*}
    e_k(a, b, h) := (x^{\Delta t}_k(a, b) - x^{\Delta t}_{t_k}(a, b), y^{\Delta t}_k(a, b) - y^{\Delta t}_{t_k}(a, b))^T\,.
\end{align*}
Through the same arguments as \cite[Thm 3.3]{RR2020}, we obtain local rate of convergence $|e_1(a, b, h)| = \mcal{O}(h^{(p+1)\alpha})$ between \eqref{eq:PRK general} and \eqref{eq:2 component smooth ODE general} when the order conditions of order $p$ for the partitioned RK method defined by $(a_{ij}, b_i)$ and $(\alpha_{ij}, \ob{b}_i)$ is applied to an ODE problem of the form 
\begin{align*}
    \dot{x}_t = f(x_t, y_t)\,,\quad \dot{y}_t = g(x_t, y_t)\,,
\end{align*}
are satisfied. See \cite[Chapter 3]{Hairer2006}, \cite[Chapter 7]{sanz2018numerical} for an in depth discussion of deriving order conditions of partitioned RK methods using bi-coloured trees. We give concrete conditions on the RK coefficients for $p = 1,2,3$. For $p=1$, we require
\begin{align}
    \begin{split}
        \sum_{i=1}^s b_i = 1 \,,\quad \sum_{i=1}^s \ob{b}_i = 1\,.
    \end{split}
\end{align}
Let $c_i := \sum_{j=1}^s a_{ij}$ and $\ob{c}_i := \sum_{j=1}^s \alpha_{ij}$. For $p=2$, we require $p=1$ conditions to be satisfied as well as
\begin{align}
    \begin{split}
        \sum_{i=1}^s b_ic_i = \frac{1}{2} \,,\quad \sum_{i=1}^s b_i\ob{c}_i = \frac{1}{2} \,,\quad \sum_{i=1}^s \ob{b}_i c_i = \frac{1}{2} \,,\quad \sum_{i=1}^s \ob{b}_i \ob{c}_i = \frac{1}{2} \,,
    \end{split}
\end{align}
For $p=3$, we require $p=1,2$ conditions to be satisfied as well as
\begin{align}
    \begin{split}
        &\sum_{i=1}^s b_ic_ic_i = \sum_{i=1}^s b_ic_i\ob{c}_i = \sum_{i=1}^s b_i\ob{c}_i\ob{c}_i =  \frac{1}{3}\,,\quad b \leftrightarrow \ob{b}\,\\
        &\sum_{i,j=1}^s b_i a_{ij}c_j = \frac{1}{6}\,,\quad \quad b \leftrightarrow \ob{b}\,, (a_{ij})\leftrightarrow (\alpha_{ij})\,, c \leftrightarrow \ob{c}\,.
    \end{split}
\end{align}
Recall from the definition of the method that $\mathbf{Z}^{\Delta t}$, the canonical lift of the smooth path $Z^{\Delta t}$, is assumed to converge to $\mathbf{Z}$ as $h \rightarrow 0$ with rate $r_0$ in the $\alpha$-H\"older rough path metric, $\rho^\alpha_g(\mathbf{Z}, \mathbf{Z}^{\Delta t}) = \mcal{O}(\Delta t^{r_0})$. Further assume that $\sup_t||x_t - x^{\Delta t}_t| + |y_t - y^{\Delta t}_t|| = \mcal{O}(\Delta t^{r_0})$. Then, we have global rate of convergence
\begin{align*}
    \max_{k} ||x_k^{\Delta t}(a,b) - x_{t_k}(a,b)| + |y_k^{\Delta t}(a,b) - y_{t_k}(a,b)|| = \mcal{O}(h^r)\,,\quad r = \min\{r_0, (p+1)\alpha - 1\}\,,
\end{align*}
where $x_{t_k}(a,b), y_{t_k}(a,b)$ are solutions to the RDE \eqref{eq:2 component RDE general} using \cite[Thm. 4.2]{RR2020}. We remark that in \cite[Thm. 4.2]{RR2020}, the vector fields are assumed to be $\operatorname{Lip}^\gamma_{loc}$ for the global existence of RDE solution. For the purpose here, we are quoting the result localised to the an assumed local pathwise solution up to some $t = t_0 + k\Delta t$. 

We note that imposing the condition $b_i = \ob{b}_i$ and $\alpha_{ij} = b_j - b_j a_{ji}/b_i$, the method \eqref{eq:PRK general} becomes the Symplectic Partitioned Runge--Kutta (SPRK) method. Additionally, we can relate the RSPRK method \eqref{eq:RSPRK} to the current setting, by setting $a_{ij} = \ob{a}_{ij}$ and $\beta_{ij} = \ob{\beta}_{ij}$ in \eqref{eq:RSPRK} to arrive at \eqref{eq:PRK general} under the preceding specialisations. However, starting from one RK method satisfying order conditions up to order $p$, and build the corresponding SPRK method using the symplectic conditions does not imply the SPRK method is of order $p$ in general. This property only holds for order $p=2$. 
Via direct computation, one can show the classical RK$4$ tableau and the conjugate pair defined in \eqref{eq:RPSPRK rk4} does satisfy every order condition up to $p=3$. However, this is not the case SSPRK$(3,3)$ tableau and its conjugate defined in \eqref{eq:RSPRK SSP}, since 
\begin{align*}
    b = (1/6,1/6,2/3)\,,\quad \ob{c} = (-1,0,1)\,, \quad \Longrightarrow \sum_i b_i \ob c_i\ob c_i = 5/6 \neq 1/3\,.
\end{align*}
Thus, the method defined by \eqref{eq:RSPRK SSP} only satisfy $p=2$ order conditions on a general coupled system of the form \eqref{eq:2 component RDE general}, even when the SSPRK$(3,3)$ method satisfy $p=3$ order conditions as a RK method. 
The global convergence rates of the various example RSPRK methods are as follows. The rough IM method \eqref{eq:RPSPRK IM}, rough symplectic Heun \eqref{eq:RPSPRK rk2} and rough symplectic SSPRK$(3,3)$ have global convergence rates $r = \min\{r_0, 3\alpha-1\}$ which tends to $0$ as the $\alpha$-H\"older regularity of the driving path tends to $\frac{1}{3}$. 
For the rough RK4 method \eqref{eq:RPSPRK rk4}, the global convergence rate is given $r = \min\{r_0, 4\alpha-1\}$ that limits to $\min\{r_0,1/3\}$ as $\alpha \rightarrow 1/3$.
\begin{remark}
    We remark that the present estimates give no positive convergence rate for the symplectic Euler method, which satisfies order conditions only for $p=1$ such that the global convergence rate $r = \min\{r_0, 2\alpha-1\} \leq 0$ for all $\alpha \leq 1/2$. This implies the rough symplectic Euler method does not converge in general, however, in special cases such as constant vector fields, the method can be shown to converge to the governing RDE \eqref{eq:2 component RDE general}. For a discussion of the pathwise convergence properties of Euler--Maruyama like schemes for RDEs, see e.g., \cite{Allan2025}.
\end{remark}
\begin{remark}
    We remark that the convergence rates are for general RDEs of the form \eqref{eq:2 component RDE general}. Several simplifications can be made to the order conditions when specialising to adjoint systems, where say $x$ is decoupled from $y$ and the $y$ evolutions is linear in $y$. See e.g., \cite{Hager2000}.
\end{remark}

\subsection{Applications to adjoint systems}\label{subsec:rough RK adjoint}
Let us restrict considerations to rough adjoint systems where $H$ and $\mcal{H}$ are linear in $p$, wlog in the form of \eqref{eq:linear ham}. In this case, the rough Hamilton's equations of $q$ becomes \eqref{eq:rough forward eq} the RSPRK method simplifies into a rough RK method where the equations \eqref{eq:RSPRK 1}, \eqref{eq:RSPRK 3} become
\begin{align}
\begin{split}\label{eq:rough RK forward eq}
    Q_{i,k} &= q_k + \Delta t \sum_{j=1}^s a_{ij} f(Q_{j,k}) + \sum_{j=1}^s \ob{a}_{ij}\sigma(Q_{j,k})\cdot Z_{t_k, t_{k+1}}\,,\quad \forall i = 1,\ldots,s\,,\\
    q_{k+1} &= q_k + \Delta t \sum_{i=1}^s b_i f(Q_{i,k}) + \sum_{i=1}^s \ob{b}_i\sigma(Q_{i,k})\cdot Z_{t_k, t_{k+1}}\,,\\
\end{split}
\end{align}
Furthermore, the rough variational systems becomes \eqref{eq:rough variational eq} and the RSPRK method for the variational system also simplifies to have
\begin{align}\label{eq:rough RK var eq}
    \begin{split}
    \delta Q_{i,k} &= \delta q_k + \Delta t \sum_{j=1}^s a_{ij} \frac{\p f}{\p q}(Q_{j,k}) \delta Q_{j,k} + \sum_{j=1}^s \ob{a}_{ij}\frac{\p \sigma}{\p q}(Q_{j,k})\delta Q_{j,k}\cdot Z_{t_k, t_{k+1}}\,,\quad \forall i = 1,\ldots,s\,,\\
    \delta q_{k+1} &= \delta q_k + \Delta t \sum_{i=1}^s b_i \frac{\p f}{\p q}(Q_{i,k}) \delta Q_{i,k} + \sum_{i=1}^s \ob{b}_i\frac{\p \sigma}{\p q}(Q_{i,k})\delta Q_{i,k}\cdot Z_{t_k, t_{k+1}}\,.\\
    \end{split}
\end{align}
Let us first present the general case, of which the rough RK case defined in \eqref{eq:rough RK forward eq}--\eqref{eq:rough RK var eq} is a specialisation. Let $\mcal{M}_{t_k, t_{k+1}, S(Z)} : Q \rightarrow Q$, $\mcal{M}_{t_k, t_{k+1}, S(Z)}(q_k) = q_{k+1}$ for $k=1,\ldots,n-1$, denotes a one-step numerical method for the RDE \eqref{eq:rough forward eq} that depends smoothly on the truncated signature of the rough path $Z$ and the vector fields $f$ and $\sigma$. Under these assumptions, $\mcal{M}_{t_k, t_{k+1}, S(Z)}$ is therefore differentiable in its $q$ argument. E.g., the rough RK method \eqref{eq:rough RK forward eq} which only depends on the path increments (Level 1 signature).
Let $T_{q_k}\mcal{M}_{t_k, t_{k+1}, S(Z)} : T_{q_k}Q \rightarrow T_{q_{k+1}}Q$ for $k=1,\ldots,n$ be the numerical method defined as the tangent lift of the map $\mcal{M}_{t_k, t_{k+1}, S(Z)}$ at $q_k$. Then, we have the following,
\begin{proposition}\label{prop:tangent general}
    Let $q$ satisfy a RDE driven by a geometric rough path $\mathbf{Z}$. Then, the formation of rough variational equation commutes with discretisation if the variational variable is solved using the tangent lift of the discretisation of $q$.
\begin{center}
\begin{tikzpicture}[
    node distance=1.8cm,
    >=Stealth,
    every node/.style={rounded corners, align=center, font=\large}
]
    \node (top left) {
        $d q_t = f(q_t)\,dt + \sigma(q_t)\,d\mathbf{Z}_t, \quad q_0 = a$ 
    };
    \node (top right)[ right=of top left] {
        $d q_t = f(q_t)\,dt + \sigma(q_t)\,d\mathbf{Z}_t, \quad q_0 = a$ \\
        $d \delta q_t = \frac{\p f}{\p q_t}\delta q_t\, dt + \frac{\p \sigma}{\p q_t}\delta q_t \,d\mathbf{Z}_t\,,\quad \delta q_0 = \delta a$
    };
    \node (left) [below =of top left] {
        $\{q_k\}$
    };
    \node (right) [below=of top right, yshift = 0.3cm] {
        $\{q_k, \delta q_k\}$
    };
    \draw[->] (top left.south) -- (left.north)node[midway, right] {method $\mcal{M}_{t_k, t_{k+1}, S(Z)}$};   ;
    \draw[->] (top left.east) -- (top right.west);
    \draw[->] (left.east) -- (right.west);
    \draw[->] (top right.south) -- (right.north) node[midway, right] {methods $\mcal{M}_{t_k, t_{k+1}, S(Z)}$, \\ $ T_{q_k}\mcal{M}_{t_k, t_{k+1}, S(Z)}$};  ;
\end{tikzpicture}    
\end{center}
Here, horizontal arrows means a linearisation with respect to a perturbation to the initial condition.
\end{proposition}
\begin{proof}
    Consider the perturbation $q^\varepsilon_0 = q_0 + \varepsilon \delta q_0$ and sequence $\{q^\varepsilon_k\}$ generated by iteratively applying the method $\mcal{M}_{t_k, t_{k+1}, S(Z)}$. Implicit differentiation of $q^\varepsilon_{k+1}$ in $\varepsilon$ yields 
    \begin{align*}
        \delta q_{k+1} = \frac{d}{d\varepsilon}\biggr|_{\varepsilon = 0} q^\varepsilon_{k+1} = \frac{d}{d\varepsilon}\biggr|_{\varepsilon = 0} \mcal{M}_{t_k, t_{k+1}, S(Z)}(q^\varepsilon_k) = T_{q_k}\mcal{M}_{t_k, t_{k+1}, S(Z)}\left(\frac{d}{d\varepsilon}\biggr|_{\varepsilon = 0}q^\varepsilon_k\right) = T_{q_k}\mcal{M}_{t_k, t_{k+1}, S(Z)}\left(\delta q_k\right)\,.
    \end{align*}
    This defines the time stepping method for $\delta q$ through $T_{q_k}\mcal{M}_{t_k, t_{k+1}, S(Z)}$, $k = 1,\ldots, n$.
\end{proof}
Specialising to rough Runge--Kutta methods, as the coefficients $a_{ij}, \ob{a}_{ij}$, $b_i$ and $\ob{b}_i$ are shared between the rough RK methods for $q$ and $\delta q$, the above methods can be considered as a single RK method for the concatenated vector $[q, \delta q]^T$ for the system of forward RDEs formed by combining \eqref{eq:rough forward eq}-\eqref{eq:rough variational eq}. Thus, we have the following corollary as an rough analogue of the standard results for ODEs, see e.g., \cite[Chapter VI, Lem. 4.1]{Hairer2006},
\begin{corollary}\label{prop:tangent rough RK}
    Let $q$ satisfy a RDE driven by a geometric rough path $\mathbf{Z}$. Then, rough Runge--Kutta discretisation commutes with the formation of variational equations. That is, the rough RK method \eqref{eq:rough RK var eq} of the rough variational equation \eqref{eq:rough variational eq} is the perturbations of initial conditions in the rough RK method \eqref{eq:rough RK forward eq} for the RDE \eqref{eq:rough forward eq}.
\end{corollary} 
\begin{proof}
    For rough RK methods of the form \eqref{eq:rough RK forward eq}, the associated tangent lift \eqref{eq:rough RK var eq} is the same rough RK method applied to the variational equations. The statement is therefore the specialisation of Proposition \ref{prop:tangent general} to the one-step method $\mcal{M}_{t_k, t_{k+1}, S(Z)}$ defined by \eqref{eq:rough RK forward eq}.
\end{proof}

Consider the rough adjoint system \eqref{eq:rough adjoint system} and the associated conservation law \eqref{eq:1-form inner prod conservation}. We define the cotangent lift of the method $\mcal{M}_{t_k, t_{k+1}, S(Z)}$ as the numerical method $T_{q_k}^*\mcal{M}_{t_k, t_{k+1}, S(Z)} : T_{q_{k+1}}^*Q \rightarrow T_{q_k}^*Q$ defined by the duality pairing
\begin{align*}
    \scp{p_{k+1}}{T_{q_k}\mcal{M}_{t_k, t_{k+1}, S(Z)}(\delta q_k)} = \scp{T^*_{q_k}\mcal{M}_{t_k, t_{k+1}, S(Z)}(p_{k+1})}{\delta q_k}
\end{align*}
We also define the adjoint of the numerical method $\mcal{M}_{t_k, t_{k+1}, S(Z)}$ for $q_k$ as the numerical method $\mcal{M}^*_{t_k, t_{k+1}, S(Z)}$ applied to the adjoint variable $p_k$ such that the conservation law $\scp{p_k}{\delta q_k} = \scp{p_0}{\delta q_0}$ holds for all $k=1,\ldots,n$.
Then, we have the following proposition
\begin{proposition}\label{prop:cotangent general}
    Let $q$ satisfy a RDE driven by a geometric rough path $\mathbf{Z}$. Then, the formation of rough adjoint commutes with discretisation if the adjoint variable is solved using the cotangent lift of the discretisation for $q$.
\begin{center}
\begin{tikzpicture}[
    node distance=1.8cm,
    >=Stealth,
    every node/.style={rounded corners, align=center, font=\large}
]
    \node (top left) {
        $d q_t = f(q_t)\,dt + \sigma(q_t)\,d\mathbf{Z}_t, \quad q_0 = a$ 
    };
    \node (top right)[ right=of top left] {
        $d q_t = f(q_t)\,dt + \sigma(q_t)\,d\mathbf{Z}_t, \quad q_0 = a$ \\
        $d p_t =- \left(\frac{\p f}{\p q_t}\right)^*p_t\,dt - \left(\frac{\p \sigma}{\p q_t}\right)^* p_t\,d\mathbf{Z}_t,\quad p_{t_1} = b$
    };
    \node (left) [below =of top left] {
        $\{q_k\}$
    };
    \node (right) [below=of top right, yshift = 0.4cm] {
        $\{q_k, p_k\}$
    };
    \draw[->] (top left.south) -- (left.north)node[midway, right] {method $\mcal{M}_{t_k, t_{k+1}, S(Z)}$};   ;
    \draw[->] (top left.east) -- (top right.west);
    \draw[->] (left.east) -- (right.west);
    \draw[->] (top right.south) -- (right.north) node[midway, right] {methods $\mcal{M}_{t_k, t_{k+1}, S(Z)}$, \\ $ T^*_{q_k}\mcal{M}_{t_k, t_{k+1}, S(Z)}$};  ;
\end{tikzpicture}    
\end{center}
Here, horizontal arrows means the formation of adjoints for RDEs (top arrow) and numerical methods (bottom arrow), respectively. 
\end{proposition}
\begin{proof}
    Let $k \in [1,\ldots,n]$ be arbitrary and consider the inner product $\scp{p_k}{\delta q_k}$. We have
    \begin{align*}
    	\scp{p_{k+1}}{\delta q_{k+1}} = \scp{p_{k+1}}{T_{q_k}\mcal{M}_{t_k, t_{k+1}, S(Z)}(\delta q_k)}  =  \scp{T^*_{q_k}\mcal{M}_{t_k, t_{k+1}, S(Z)} (p_{k+1})}{\delta q_k} =: \scp{p_k}{\delta q_k} \,.
    \end{align*}
    where we have $T^*_{q_k}\mcal{M}_{t_k, t_{k+1}, S(Z)} (p_{k+1}) =: p_k$. Iterating over $k$ gives $\scp{p_k}{\delta q_k} = \scp{p_0}{\delta q_0}$ for every $k$, and the discretisation is natural precisely when $\mcal{M}^*_{t_k, t_{k+1}, S(Z)} = T_{q_k}^*\mcal{M}_{t_k, t_{k+1}, S(Z)}$.
\end{proof}
Focusing on the case of rough RK methods for adjoint systems in the form of \eqref{eq:rough adjoint system}, the RSPRK method \eqref{eq:RSPRK} becomes 
\begin{align}
\begin{split}
    P_{i,k} &= p_{k+1} + \Delta t \sum_{j=1}^s \beta_{ij} \left( \frac{\p f}{\p q}\right)^*_{Q_{j,k}} P_{j,k} + \sum_{j=1}^s \ob{\beta}_{ij} \left( \frac{\p \sigma}{\p q}\right)^*_{Q_{j,k}} P_{j,k}\cdot Z_{t_k, t_{k+1}} \,,\quad \forall i = 1,\ldots,s\,,\\
    p_k &= p_{k+1} + \Delta t \sum_{i=1}^s b_i  \left( \frac{\p f}{\p q}\right)^*_{Q_{i,k}} P_{i,k} + \sum_{i=1}^s \ob{b}_i  \left( \frac{\p \sigma}{\p q}\right)^*_{Q_{i,k}} P_{i,k} \cdot Z_{t_k, t_{k+1}} \,,
\end{split}
\end{align}
where the RK coefficients are related to \eqref{eq:rough RK forward eq} via \eqref{eq:RSPRK coeff conds}. Assuming $b_i$ and $\ob{b}_i$ are nonzero, as a consequence of the conservation law in Proposition \ref{Prop:RSPRK conservation laws}, we have the following rough extension of the classical results \cite[Thm. 3.3 \& 3.4]{Sanz-Serna2016}.
\begin{corollary}\label{cor:rk adjoint commutes}
    Let the driving rough path $\mathbf{Z}$ be geometric. Then rough Partitioned Runge--Kutta discretisation commutes with the formation of adjoint equations of RDEs provided the pair of tableaux $\{(a_{ij}, b_i), (\beta_{ij}, b_i)\}$ used for the state and the adjoint satisfies the symplecticity relations \eqref{eq:RSPRK coeff conds}, and the same holds for rough Runge--Kutta discretisation when the pair of tableaux defining the partitioned method coincide and satisfies \eqref{eq:RSPRK coeff conds}.
\end{corollary}
\begin{proof}
    Under \eqref{eq:RSPRK coeff conds}, the rough Partitioned Runge--Kutta coincide with the RSPRK method defined in \eqref{eq:RSPRK} such that Proposition \ref{Prop:RSPRK conservation laws} gives $\scp{p_{k+1}}{\delta q_{k+1}} = \scp{p_k}{\delta q_k}$ at every step. That identity is exactly the defining property of the adjoint $\mcal{M}^*_{t_k, t_{k+1}, S(Z)}$ of the method, so the discrete adjoint coincides with the cotangent lift $T^*_{q_k}\mcal{M}_{t_k, t_{k+1}, S(Z)}$ and Proposition \ref{prop:cotangent general} holds. The non-partitioned case is the same argument with the pair of tableaux coinciding.
\end{proof}

\section{Numerical Examples}\label{sec:examples}
\subsection{Kubo Oscillator}
In this example, we demonstrate the convergence properties of the RSPRK methods by considering the example of two coupled Kubo Oscillators driven by Fractional Brownian Motion (fbm) at different Hurst parameters as an initial value problem. Let $\mathbf{Z} = (\mathbf{Z}^1, \mathbf{Z}^2)$ be a two component i.i.d fbm with Hurst parameter $\mathfrak{h}$. For $\mathfrak{h} \in (\frac{1}{3}, 1]$, fbm is a member of a large family of Gaussian rough paths that can be embedded into $\mcal{C}^\alpha_g$ for any $\alpha \in (\frac{1}{3}, \mathfrak{h})$ almost surely \cite{FV10}. As shown in e.g., \cite{FR2014}, the convergence rate $r_0$ of piecewise linear approximation of fbm can be taken as arbitrarily close to $2\mathfrak{h} - 1/2$ for $\mathfrak{h} \in (\frac{1}{4}, 1]$. For simplicity of demonstration, we will take the $r_0 = 2\mathfrak{h} - 1/2$ and ignore the arbitrary constant when plotting the reference rates. For the rough IM method \eqref{eq:RPSPRK IM}, it satisfies order conditions up to order $p=2$ which implies an expected pathwise convergence rate of $r = \min\{2\mathfrak{h} - 1/2,\, 3\mathfrak{h}-1\}$. Let $q, p \in \mathbb{R}^2$ and consider the Hamiltonians  
\begin{align}
    \begin{split}
        &H(q,p) = \frac{1}{2}\sum_{i=1}^2\left(p_i^2 + q_i^2\right) + \frac{\kappa}{2}\left(q_1 - q_2\right)^2 \,,\\
        &\mcal{H} = \left(H_1, H_2\right)^T\,,\quad H_1(q,p) = \frac{1}{2}\mu\sum_{i=1}^2\left(p_i^2 + q_i^2\right) \,,\quad H_2(q,p) = \nu \left(q_1p_2 + q_2p_1\right)\,,
    \end{split}
\end{align}
where $\kappa, \mu, \nu \in \mathbb{R}$ are parameters that control the mechanical coupling between the oscillators and the noise intensity.
The resulting rough canonical Hamilton's equations are
\begin{align*}
    \begin{split}
        d\begin{pmatrix}
            q_1 \\q_2
        \end{pmatrix} 
        = 
        \begin{pmatrix}
            p_1 \,dt + \mu p_1 \,d\mathbf{Z}^1_t + \nu q_2\,d\mathbf{Z}^2_t \\
            p_2 \,dt + \mu p_2 \,d\mathbf{Z}^1_t + \nu q_1\,d\mathbf{Z}^2_t
        \end{pmatrix}\,,\quad
        d\begin{pmatrix}
            p_1 \\ p_2
        \end{pmatrix}
        = 
        \begin{pmatrix}
            -q_1\,dt - \kappa(q_1 - q_2)\,dt - \mu q_1\,d\mathbf{Z}^1_t - \nu p_2\,d\mathbf{Z}^2_t \\
            -q_2\,dt + \kappa(q_1 - q_2)\,dt - \mu q_2\,d\mathbf{Z}^1_t - \nu p_1\,d\mathbf{Z}^2_t 
        \end{pmatrix}\,.
    \end{split}
\end{align*}
The parameters are taken to be $q_0 = (1,0)$, $p_0 = (0,1)$, $\kappa = 0.2$, $\mu = \nu = 0.25$ and the time window for the simulation is taken to be $[t_0, t_1] = [0, 1]$.
As fbm is random path, we consider convergence results of rough IM method \eqref{eq:RPSPRK IM} pathwise and in the strong sense using the the RDE above. Here, the pathwise convergence is measured by the pathwise error defined as
\begin{align*}
    err_p(x^{\Delta t}, x_{t_k}) = \max_k |x^{\Delta t}_k - x_{t_k}|\,,
\end{align*}
and the strong convergence is measured by the strong error
\begin{align*}
    err_s(x^{\Delta t}, x_{t_k}) = \mathbb{E}[\max_k |x^{\Delta t}_k - x_{t_k}|]\,,
\end{align*}
where $x := (q,p)^T$, $x^{\Delta t}$ is the numerical solution solved using \eqref{eq:RPSPRK IM} and $x_{t}$ is the exact solution for a fixed sample fbm path. As the RDE does not possess analytical solution, the exact solution replaced by a reference is computed using an ultra fine resolution simulation using $\Delta t = 2^{-22}$ for the same realisation of fbm. As the smaller $\mathfrak{h}$ implies rougher sample path, the pathwise convergence rates are less visible for small $\mathfrak{h}$ than larger $\mathfrak{h}$ for fixed $\Delta t$ and much higher resolution simulation is required to demonstrate the convergence asymptotics numerically. The pathwise convergence results are shown in Figure \ref{fig:pathwise conv rates}.
\begin{figure}[!ht]
    \centering
    \begin{subfigure}[b]{0.32\textwidth}
        \centering
        \includegraphics[width=\textwidth]{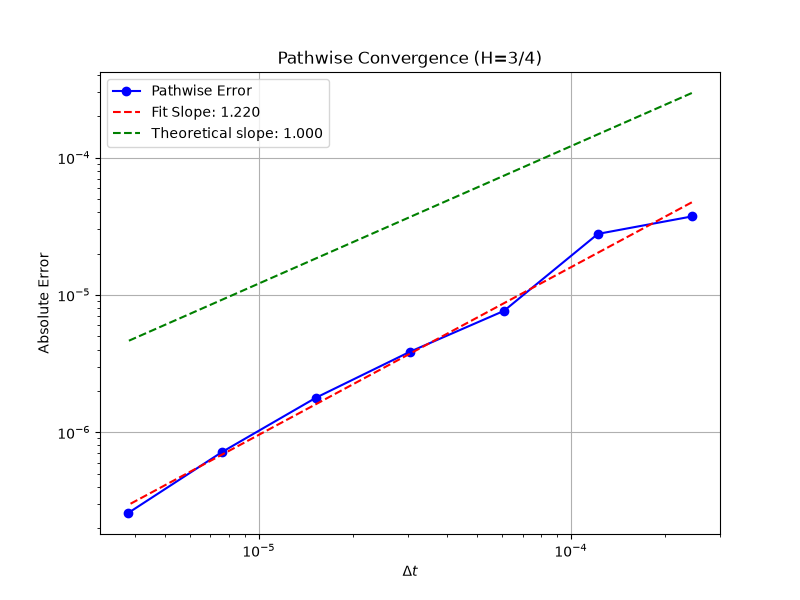}
    \end{subfigure}
    \begin{subfigure}[b]{0.32\textwidth}
        \centering
        \includegraphics[width=\textwidth]{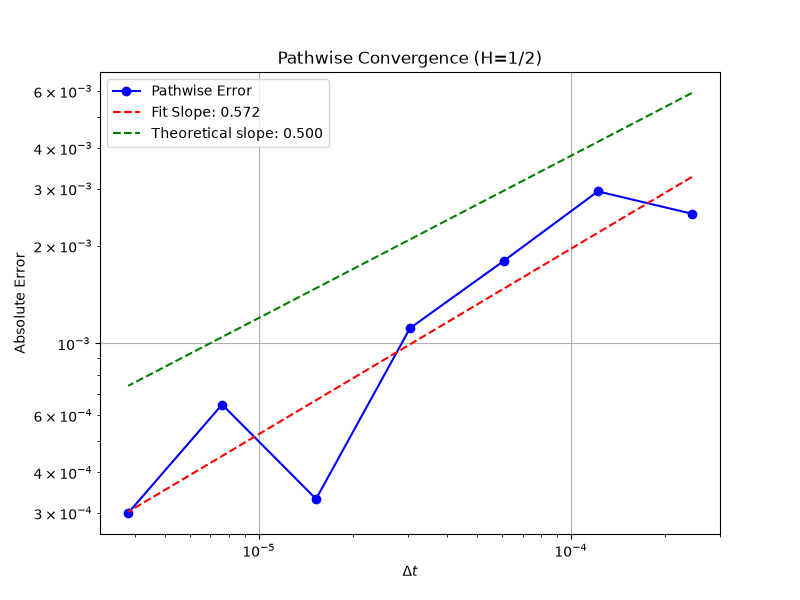}
    \end{subfigure}
    \begin{subfigure}[b]{0.32\textwidth}
        \centering
        \includegraphics[width=\textwidth]{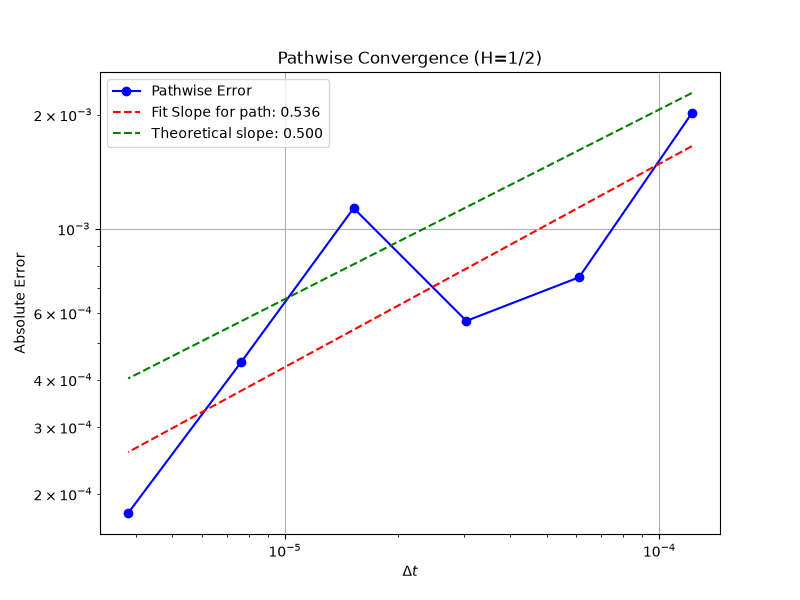}
    \end{subfigure}
    \newline
    \begin{subfigure}[b]{0.32\textwidth}
        \centering
        \includegraphics[width=\textwidth]{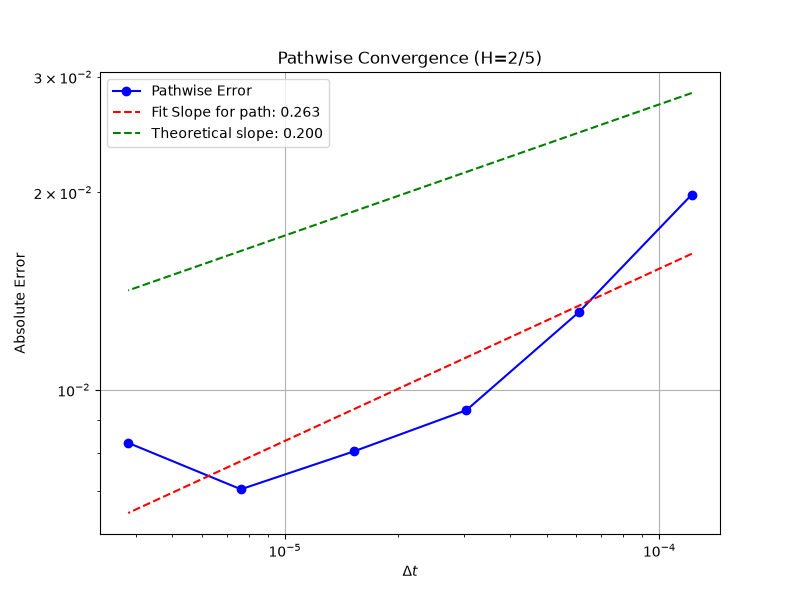}
    \end{subfigure}
    \begin{subfigure}[b]{0.32\textwidth}
        \centering
        \includegraphics[width=\textwidth]{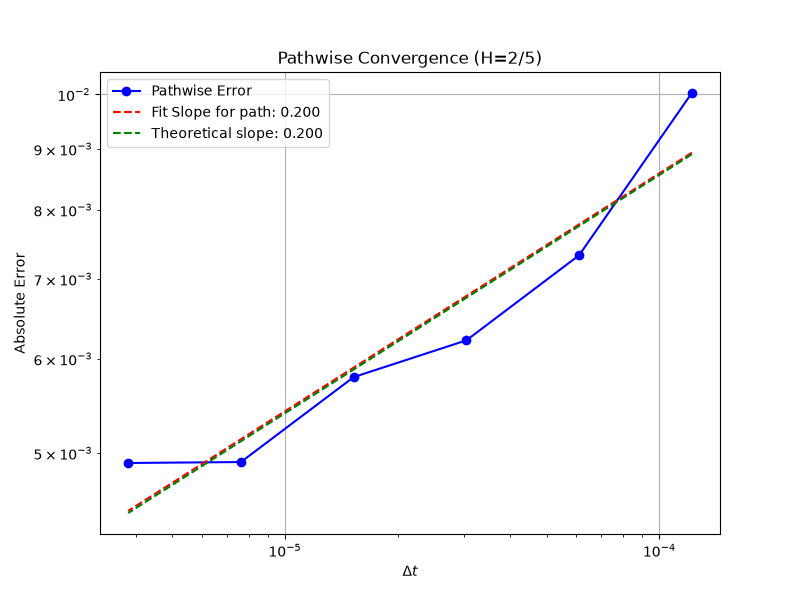}
    \end{subfigure}
    \begin{subfigure}[b]{0.32\textwidth}
        \centering
        \includegraphics[width=\textwidth]{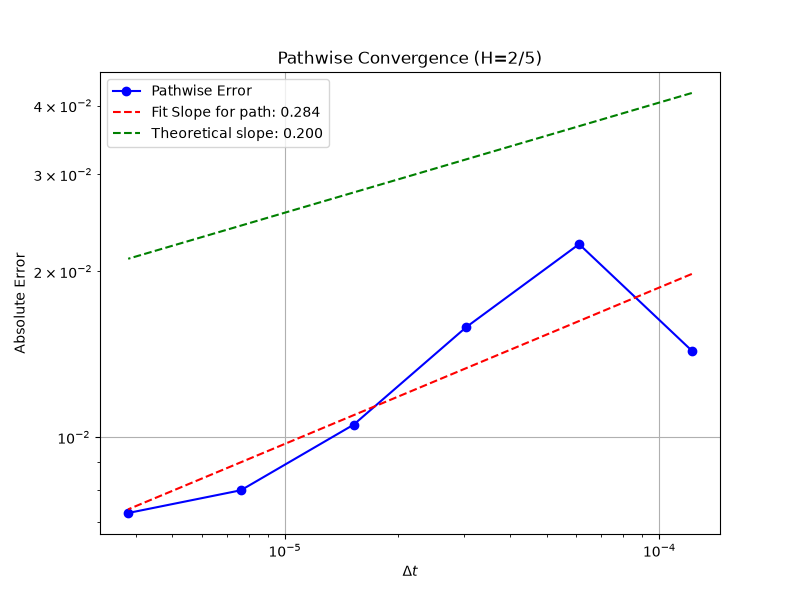}
    \end{subfigure}
    \caption{Pathwise convergence rates for the rough IM method applied RDE driven by fbm for different values of Hurst parameter $\mathfrak{h} = 3/4$ (top left), $\mathfrak{h}= 1/2$ (top middle \& right) and $\mathfrak{h} = 2/5$ (bottom row) for single sample path. For $\mathfrak{h} = 3/4$ and $\mathfrak{h} = 1/2$, the observed convergence rates are consistent with the theoretical bound. At $\mathfrak h = 1/2$, the Wong--Zakai rate $2\mathfrak h - 1/2$ and the truncation rate $3\mathfrak h - 1$ coincide implying the rough IM method is optimal as it achieves the bounds of the Wong--Zakai approximation without exceeding it. For $\mathfrak{h} = 2/5$ however, the convergence rate is bounded by the local truncation error of the rough IM method. }
    \label{fig:pathwise conv rates}
\end{figure}

The strong convergence rates for different values of $\mathfrak{h}$ are computed using $32$ independent realisations of the driving fbm. For computational efficiencies, the fine resolution reference solution is computed using $\Delta t = 2^{-16}$. Numerical result for strong convergence rates is shown in Figure \ref{fig:strong conv rates}.
\begin{figure}[!ht]
    \centering
    \begin{subfigure}[b]{0.32\textwidth}
        \centering
        \includegraphics[width=\textwidth]{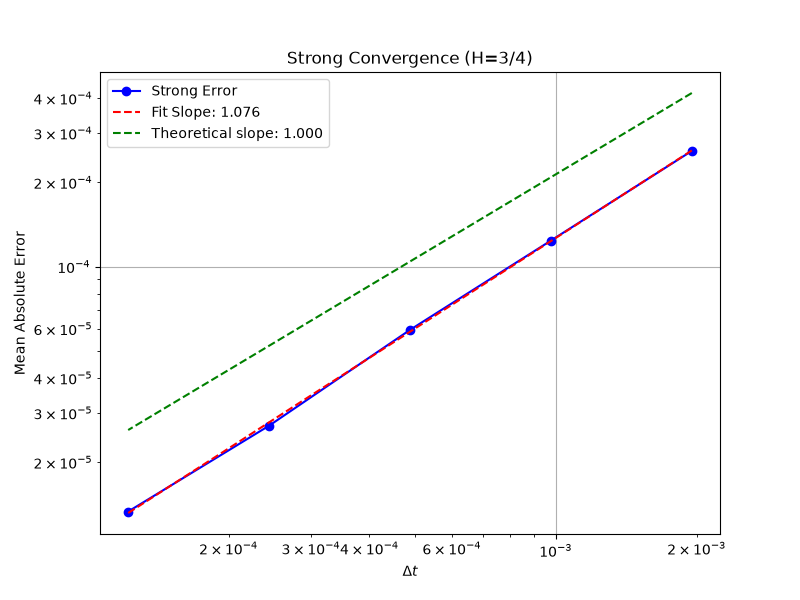}
    \end{subfigure}
    \begin{subfigure}[b]{0.32\textwidth}
        \centering
        \includegraphics[width=\textwidth]{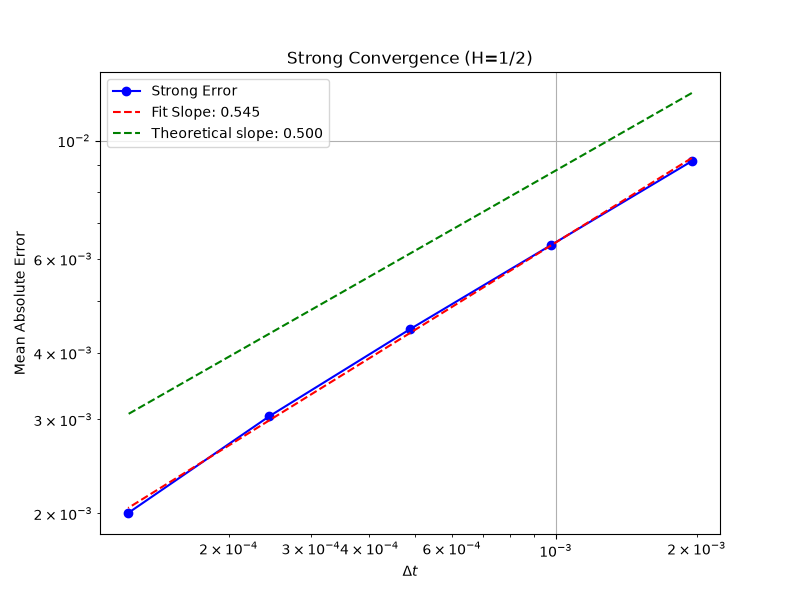}
    \end{subfigure}
    \begin{subfigure}[b]{0.32\textwidth}
        \centering
        \includegraphics[width=\textwidth]{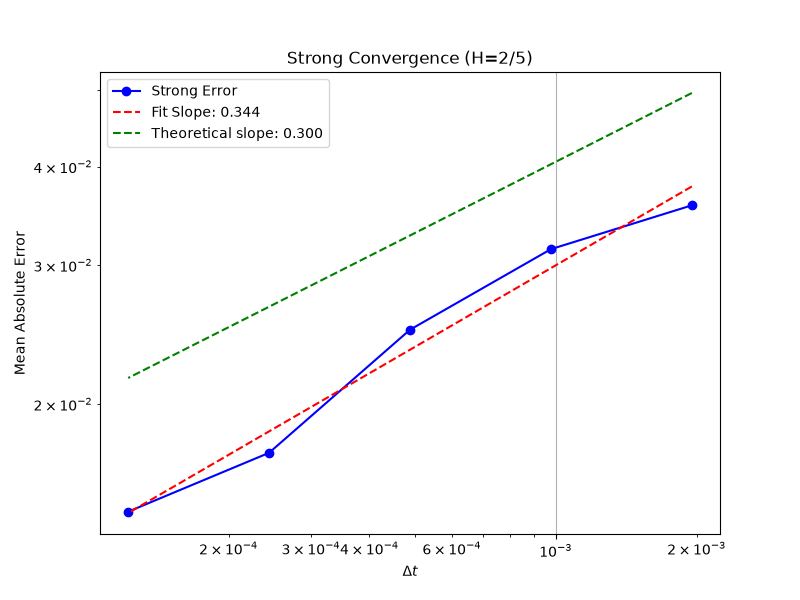}
    \end{subfigure}
    \caption{Strong convergence rates for the rough IM method applied to the RDE driven by fbm for different values of Hurst parameter $\mathfrak{h} = 3/4$ (left), $\mathfrak{h}= 1/2$ (middle) and $\mathfrak{h} = 2/5$ (right). Reference strong convergence rates are found in \cite{HHW2018}.}
    \label{fig:strong conv rates}
\end{figure}

\subsection{Sensitivity analysis}\label{subsec:gradient comp}
We demonstrate the pathwise adjoint sensitivity with respect to initial conditions and parameters using the examples of rough path driven Lorenz 96 model. The deterministic Lorenz 96 model \cite{LE1998} is a $N$-dimensional ODE system ($N$ is typically $40$) that models some atmospheric physical quantities at $N$-equidistant sites situated at the a fixed latitude circle. The dynamics consist of quadratic terms corresponds to advection, as well as dissipation and forcing. In the deterministic setting, standard numerical integration has shown that the local discrepancies tends to double every $2$ days which propagate eastwards, eventually encapsulating the circle. This property makes the Lorenz 96 model a useful toy-model for evaluating data assimilation techniques. The Lorenz 96 model is expressed as the ODE
\begin{align}
    \frac{d}{dt}q_n = (q_{n+1} - q_{n-2})q_{n-1} - q_{n} + F\,, \label{eq:det Lorenz 96}
\end{align}
where $q \in \mathbb{R}^N$ is the state space, $F \in \mathbb{R}$ is a given constant and the indexing is periodic at the boundaries. That is, $q_{N+1} = q_{1}$, $q_{0} = q_{N}$, etc. The quadratic part of the Lorenz 96 dynamics can be written as a skew-gradient (almost-Poisson) flow, so that
\begin{align*}
    \frac{d}{dt}q_n = \{q_n, H\}- q_{n} + F \,, 
\end{align*}
where the Hamiltonian $H$ and the skew-symmetric bracket $\{\cdot,\cdot\}: C^\infty(\mathbb{R}^N)^2 \rightarrow C^\infty(\mathbb{R}^N)$ are given by
\begin{align*}
    H(q) = \frac{1}{2}\sum_{n=1}^N q_n^2\,, \quad \{f,g\} = \nabla_q f \cdot \mathbb{J} \nabla_q g\,,\quad \text{where} \quad \mathbb{J}_{ij} := q_{i-1}\delta_{j, i+1} - q_{j-1}\delta_{i, j+1}\,,
\end{align*}
for all $f, g \in C^\infty(\mathbb{R}^N)$ and $\delta$ is the Kronecker delta tensor. Note that $\{\cdot,\cdot\}$ is skew-symmetric but is \emph{not} a Poisson bracket as it fails to satisfy the Jacobi identity. This is consistent with the observation in \cite{FCHZ2025} that the standard inviscid Lorenz models conserve energy but are not Hamiltonian. In this sense, Lorenz 96 can be treated as a poor discretisation of a continuum Poisson bracket that does not preserve the Jacobi identity. 

We consider a rough path perturbation to the Lorenz 96 model by perturbing the Poisson bracket part of the motion following the Stochastic Advection by Lie Transport (SALT) framework \cite{Holm2015, DHL2024} that has found usefulness in parameterising small-scale features occurring in ocean dynamics into coarse resolution simulations \cite{HP2023}. Once again let $\mathbf{Z} = (\mathbf{Z}^1,\ldots,\mathbf{Z}^K)$ be a $K$-component fbm with Hurst parameter $\mathfrak{h}$. Let $\xi \in (\mathbb{R}^{N})^{K}$ be given parameters and form the Hamiltonian $\mcal{H} := (H_1,\ldots,H_K)$ by $H_k = \sum_{i=1}^N \xi_{i,k}q_i$. Consider the forward RDE
\begin{align}
\begin{split}
    d q_n &= \{q_n, H\}\,dt + \{q_n, \mcal{H}\}\,d\mathbf{Z}_t- q_{n}\,dt + F\,dt \\
    &= (q_{n+1} - q_{n-2})q_{n-1}\,dt + \sum_{k=1}^K\left(q_{n-1}\xi_{n+1, k} - q_{n-2}\xi_{n-1,k}\right)\,d\mathbf{Z}_t^k- q_{n}\,dt + F\,dt \,, \label{eq:SALT lorenz 96 forward}
\end{split}
\end{align}
and its adjoint RDE given by
\begin{align}
    d p_n = \left(p_{n+2}q_{n+1} -p_{n-1}q_{n-2} - (q_{n+2} - q_{n-1})p_{n+1} + p_n\right)\,dt - \sum_{k=1}^K\left(p_{n+1}\xi_{n+2, k} - p_{n+2}\xi_{n+1,k}\right)\,d\mathbf{Z}_t^k \,, \label{eq:SALT lorenz 96 adjoint}
\end{align}
In the subsequent numerical results, we fix $N = 40$ and $K = 2$. Typical evolutions of the states, tangents and adjoints for a single realisation of the driving fbm where $\mathfrak{h} = 2/5$ are shown in Figure \ref{fig:sample evo}.
\begin{figure}[!ht]
    \centering
    \begin{subfigure}[b]{0.32\textwidth}
        \centering
        \includegraphics[width=\textwidth]{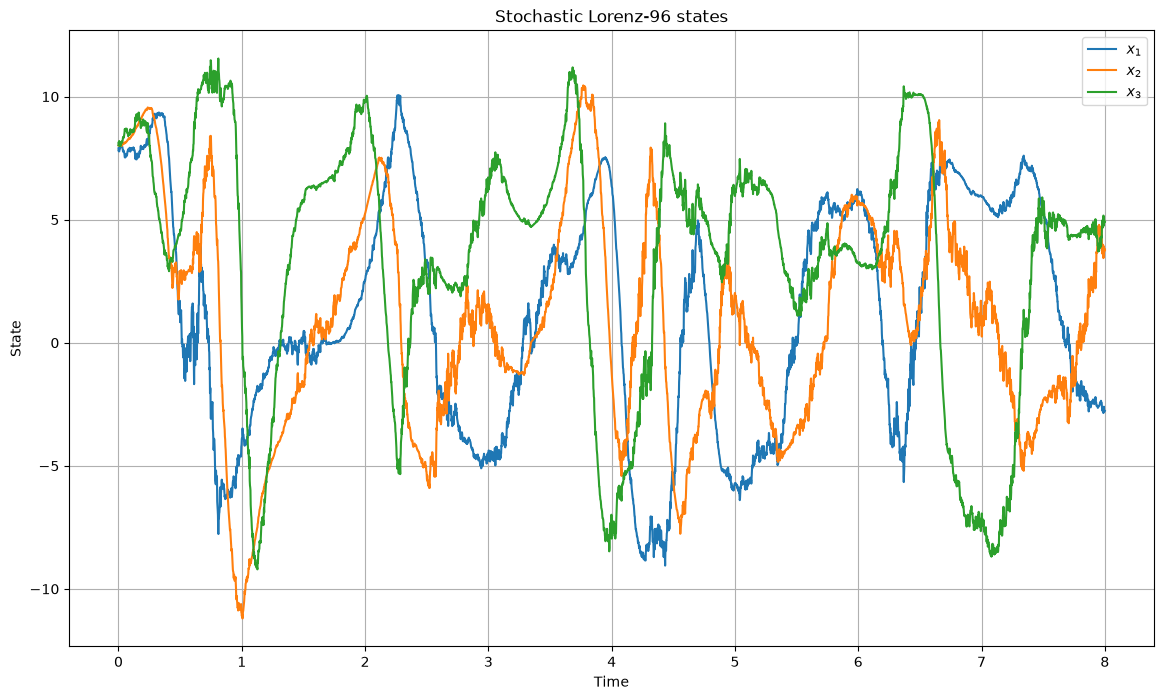}
    \end{subfigure}
    \begin{subfigure}[b]{0.32\textwidth}
        \centering
        \includegraphics[width=\textwidth]{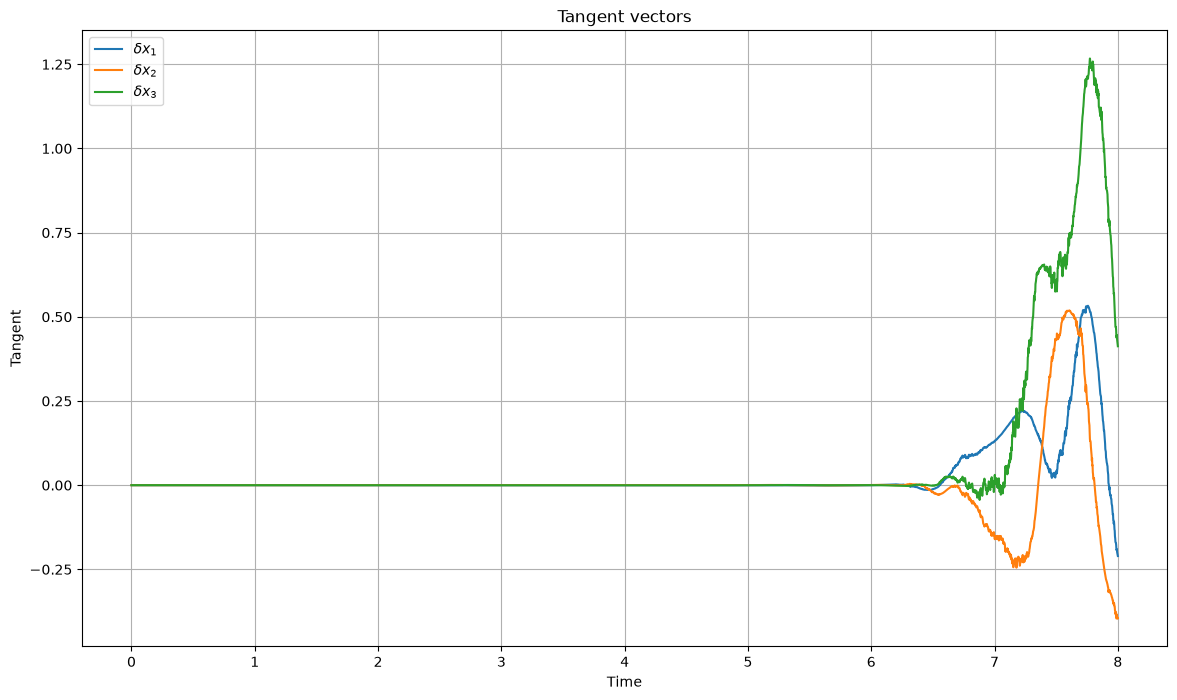}
    \end{subfigure}
    \begin{subfigure}[b]{0.32\textwidth}
        \centering
        \includegraphics[width=\textwidth]{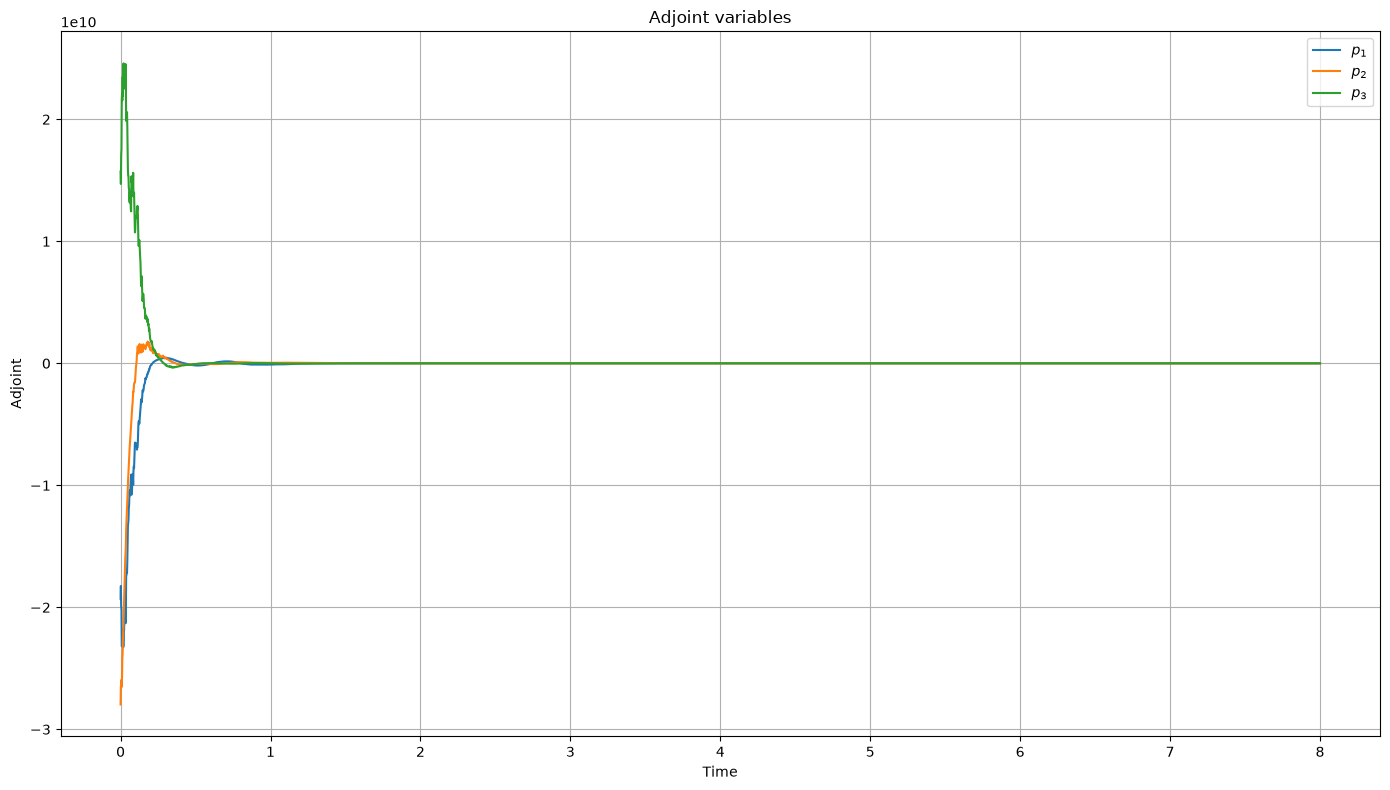}
    \end{subfigure}
    \caption{Typical evolutions of the stochastic Lorenz 96 adjoint system: state $q$ (Left), tangent vectors $\delta q$ (Middle) and adjoints $p$ (Right) for the first $3$ states. Similar to the deterministic case, the tangent vectors grow rapidly due to the sensitivity of trajectories with respect to its initial conditions. The same growth is visible in the adjoint variables near $t = t_0$ at the current scale they are the measure of gradients of a function of final states against initial states.}
    \label{fig:sample evo}
\end{figure}

Using the RSPRK method \eqref{eq:RSPRK} with RK coefficients defined in \eqref{eq:RPSPRK rk2}, \eqref{eq:RPSPRK rk4} and\eqref{eq:RSPRK SSP}, we numerically demonstrate proposition \ref{Prop:RSPRK conservation laws} to machine precision in the absence of running costs $L$ and $\mathfrak{L}$ for single realisation of the driving fbm where $\mathfrak{h} = 2/5$. Hurst parameter $\mathfrak{h} = 2/5$ is selected in this case due to its roughness whilst remaining in the $\mcal{C}^\alpha_g$, $\alpha \in (\frac{1}{3}, \frac{1}{2}]$ setting such that the effectiveness of the RSPRK methods can be demonstrated more clearly.
For contrast, we also present the deviation of the conservation of $1$-form when the backwards adjoint solve is simulated using the same RK scheme as the forward state solve without conforming to the symplecticity conditions. The results are shown in Figure \ref{fig:1form conservation}.
\begin{figure}[!ht]
    \centering
    \begin{subfigure}[b]{0.48\textwidth}
        \centering
        \includegraphics[width=\textwidth]{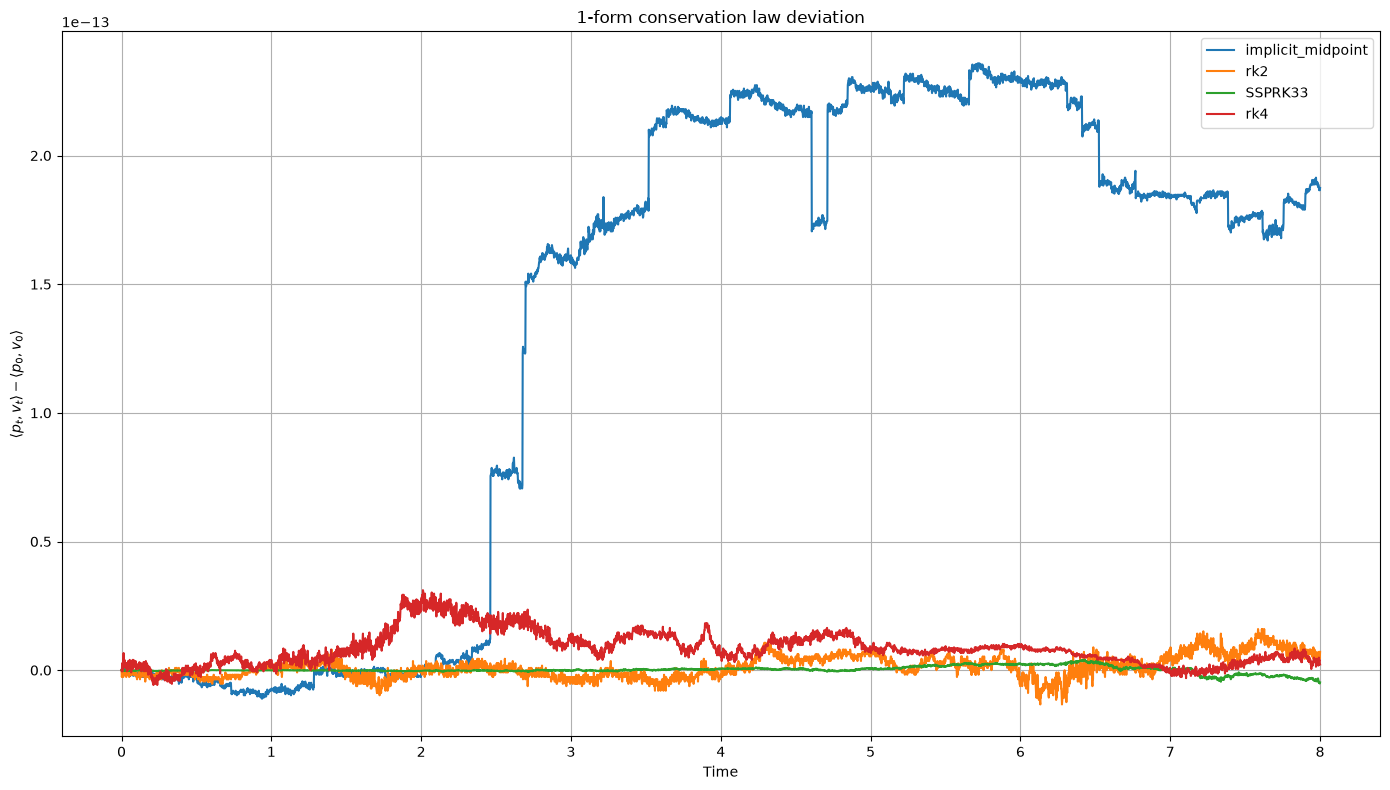}
    \end{subfigure}
    \begin{subfigure}[b]{0.48\textwidth}
        \centering
        \includegraphics[width=\textwidth]{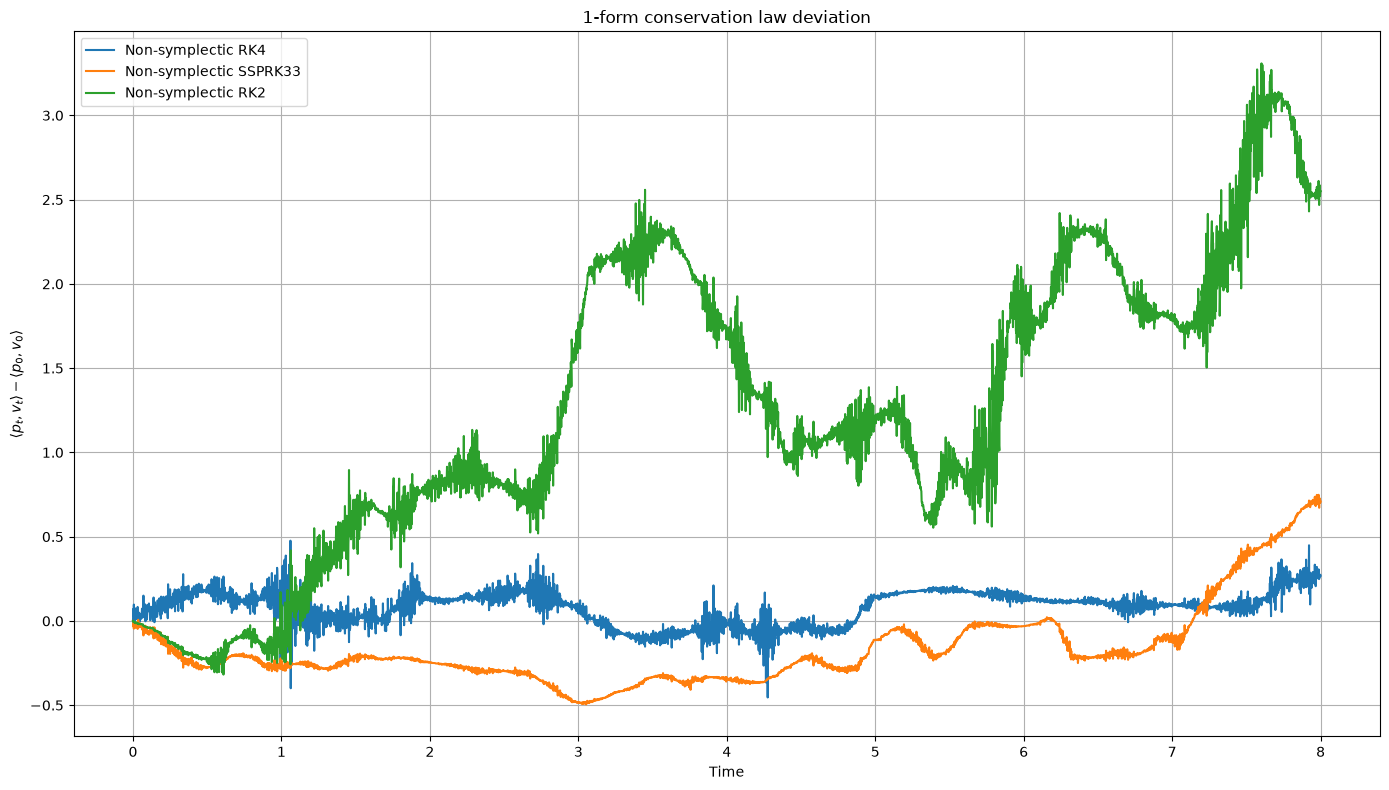}
    \end{subfigure}
    \caption{Errors in the numerical conservation of $\scp{p}{\delta q}$ as a function of time using the RSPRK family of methods \eqref{eq:RSPRK} (left) and non-symplectic methods (right). The RSPRK methods achieve near machine precision conservation of $\scp{p}{\delta q}$ for each realisation of the driving fbm whilst the non-symplectic methods develop errors of $\mcal{O}(1)$.}
    \label{fig:1form conservation}
\end{figure}

To demonstrate the effectiveness of accurate gradient computations using symplectic discretisation of the rough adjoint system, we first consider an optimisation problem where the driving rough path $\mathbf{Z}$ is fixed. Let $q(t_0)$ be a fixed initial conditions and let $q_{ref}(t; q(t_0), \mathbf{Z}, \xi_{ref})$ be a reference solution of the RDE \eqref{eq:SALT lorenz 96 forward} starting from the initial condition $q(t_0)$ for a fixed realisation of fbm with Hurst parameter $\mathfrak{h}$ and an unknown $\xi = \xi_{ref}$. Consider the following optimisation problem with $L_2$ loss function,
\begin{align}
    \min_{\xi} J = \min_{\xi}\left[\frac{1}{2}|| q(t_1; \mathbf{Z},\xi) - q_{ref}(t_1; \mathbf{Z}, \xi_{ref}) ||^2 \right]\,,\label{eq:pathwise loss}
\end{align}
where $q(t_1; \mathbf{Z},\xi)$ is the solution of the RDE \eqref{eq:SALT lorenz 96 forward} for same realisation of fbm as the reference solution and $\xi$ to be optimised. We use a Quasi--Newton gradient method to solve the minimisation problem \eqref{eq:pathwise loss} where the pathwise gradients with respect to $\xi$ are computed using various rough symplectic methods in the class of \eqref{eq:RSPRK}. For comparison, we also use a rough RK$4$ method and rough RK$2$ to solve the adjoint equation \eqref{eq:SALT lorenz 96 adjoint} where $q$ evaluation during the internal steps of the adjoint variable $p$ dynamics do not conform to symplectic conditions. The results are shown in Figure \ref{fig:grad descent pathwise}. We note that in the optimisation problem presented in Figure \ref{fig:grad descent pathwise}, the Quasi--Newton (BFGS) method can be forced to terminate early when symplectic adjoints are not used due. This is due to the gradient norm criterion within the BFGS cannot be satisfied when the gradient computations are not exact (to machine precision), as it is the case with non-symplectic adjoint methods.
\begin{figure}[!ht]
    \centering
    \begin{subfigure}[b]{0.32\textwidth}
        \centering
        \includegraphics[width=\textwidth]{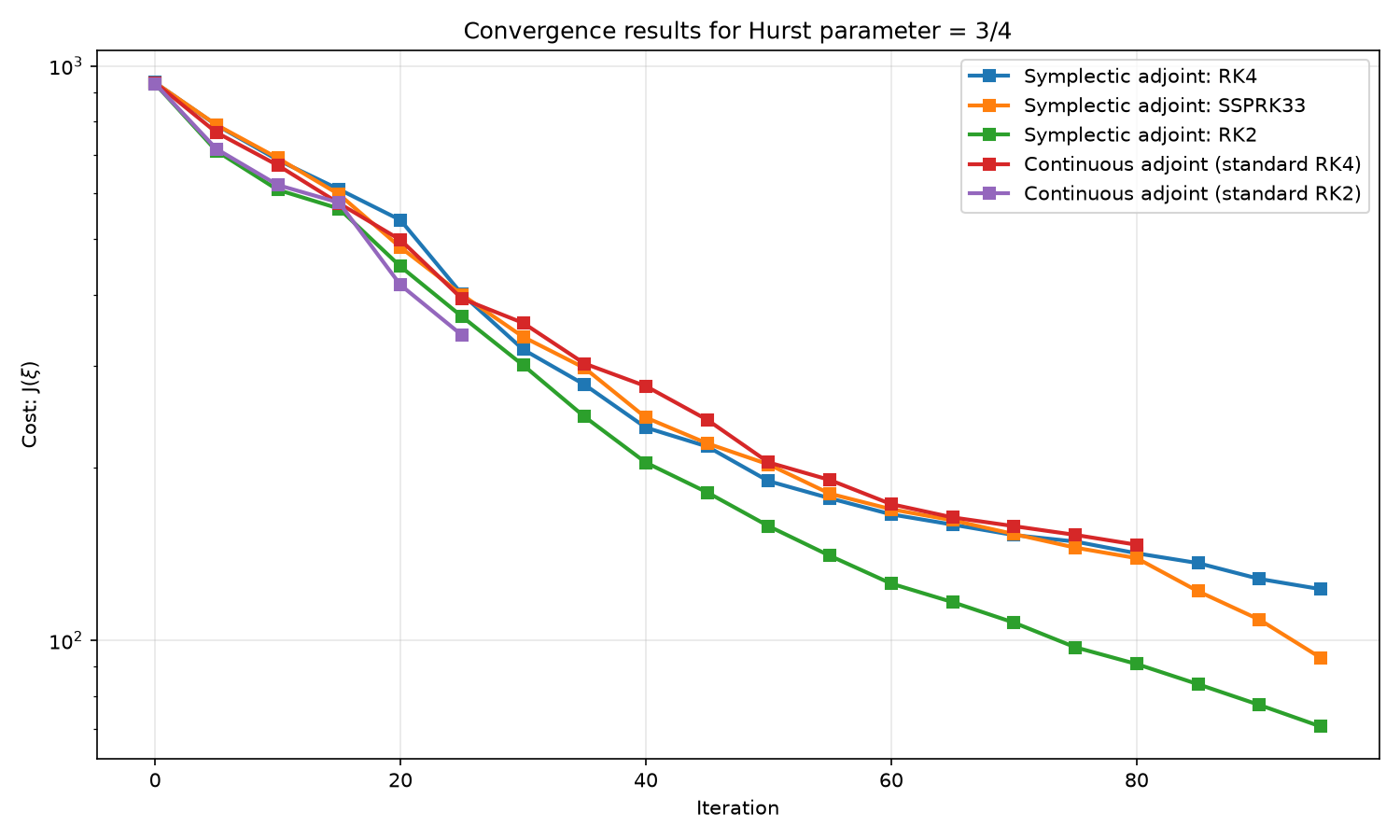}
    \end{subfigure}
    \begin{subfigure}[b]{0.32\textwidth}
        \centering
        \includegraphics[width=\textwidth]{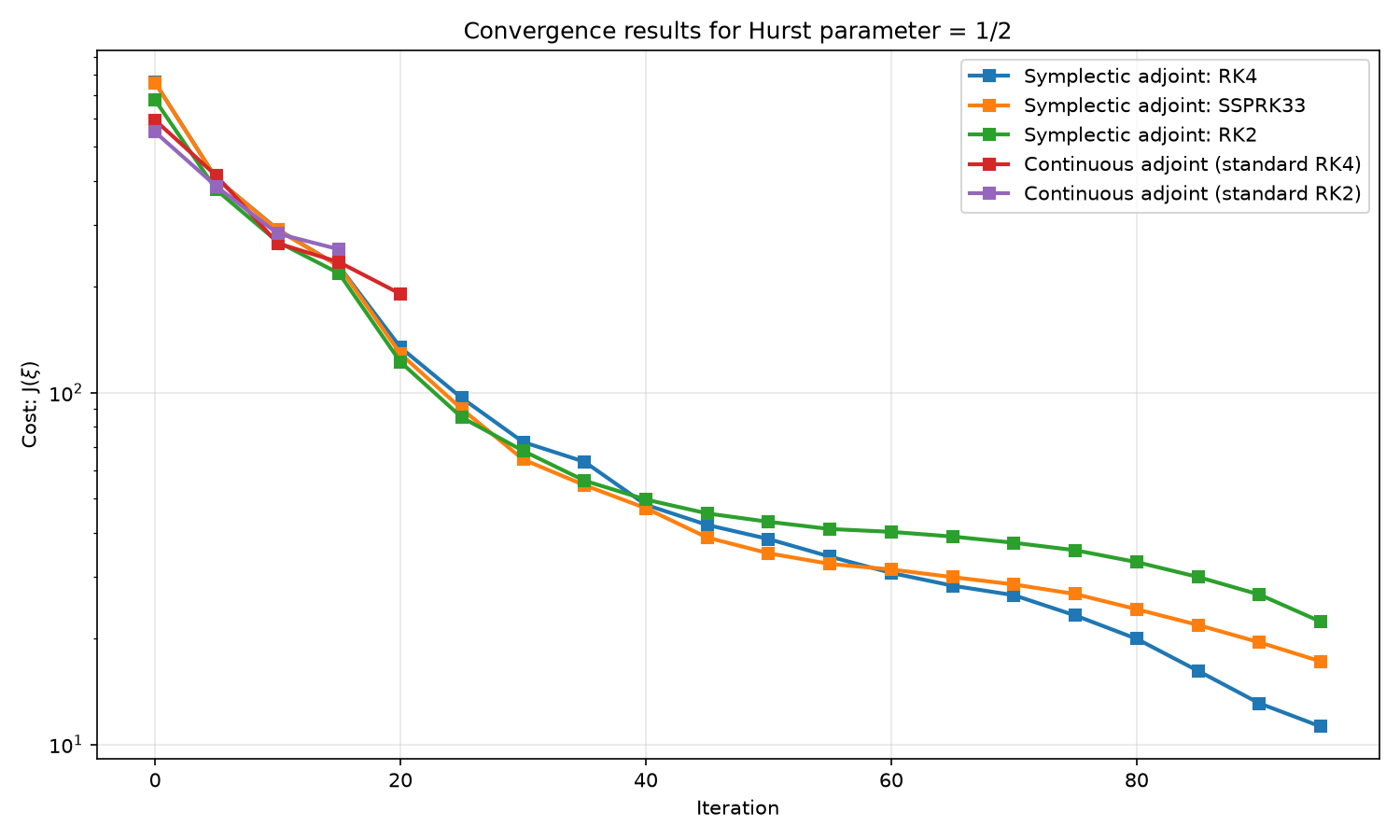}
    \end{subfigure}
    \begin{subfigure}[b]{0.32\textwidth}
        \centering
        \includegraphics[width=\textwidth]{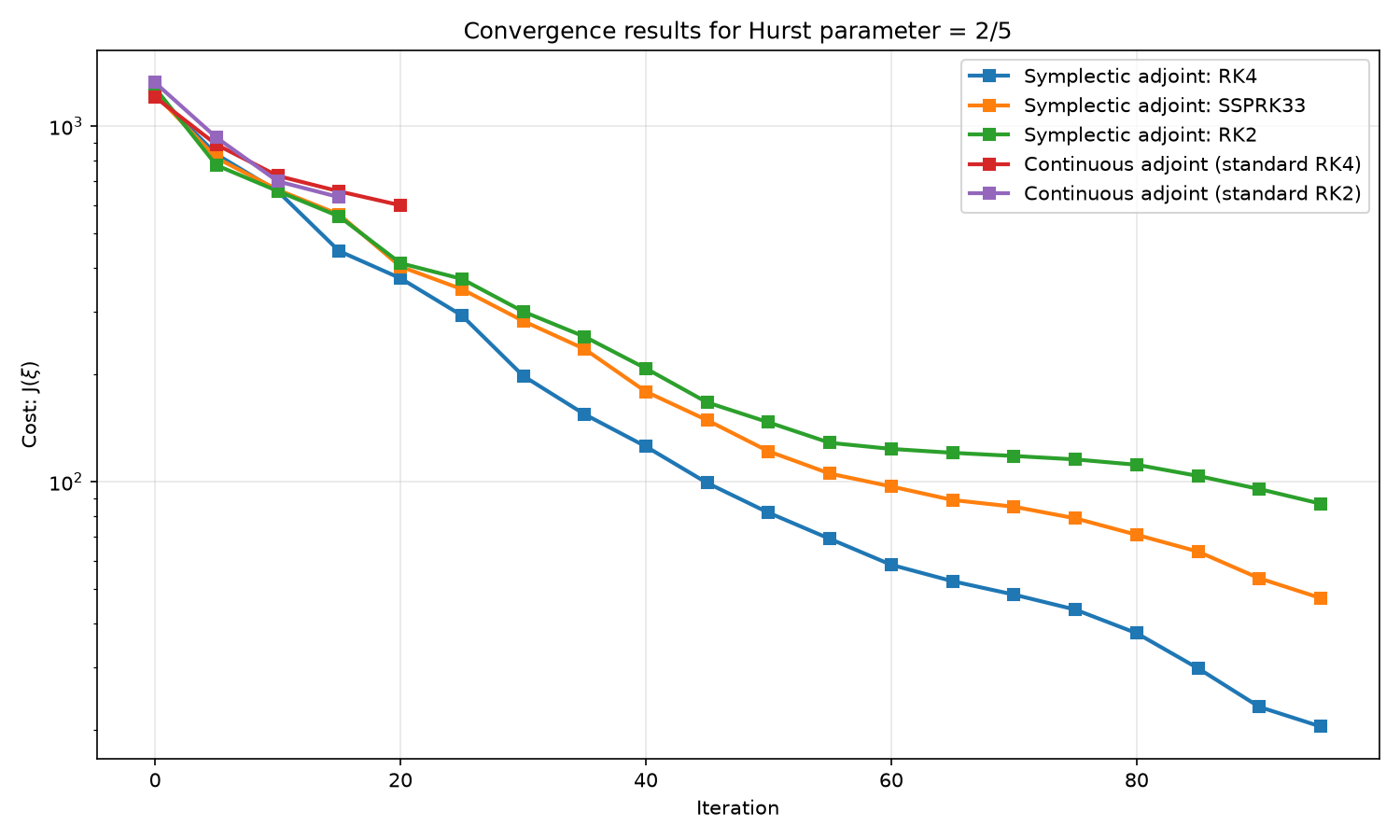}
    \end{subfigure}
    \caption{Convergence results of the Quasi--Newton gradient descent for optimisation problem \eqref{eq:pathwise loss} with different Hurst parameters, $\mathfrak{h} = 3/4$ (Left), $\mathfrak{h} = 1/2$ (Middle) and $\mathfrak{h} = 2/5$ (Right). As the roughness of the adjoint system is determined by $\mathfrak{h}$, for smaller $\mathfrak{h}$, the requirement of symplectic adjoint methods to compute gradients is essential. As seen in the optimisation example, the Quasi--Newton algorithm can be forced to terminated early (at iteration $50$ for $\mathfrak{h}=3/4$, iteration $25$ for $\mathfrak{h} = 1/2$ and iteration $10$ for $\mathfrak{h}=2/5$) due to the loss of precision in gradient computation when symplectic adjoints are not used. }
    \label{fig:grad descent pathwise}
\end{figure}

We also consider the optimisation problem where the pathwise loss is the cross-entropy between two states,
\begin{align}
    \min_{\xi}\, \mathbb{E}[J] = \min_{\xi}\, \mathbb{E}\left[-\sum_{i=1}^N \big[\operatorname{softmax}(q_{ref}(t_1; q(t_0)))\big]_i\log\Big(\big[\operatorname{softmax}(q(t_1; q(t_0), \mathbf{Z}(\omega), \xi))\big]_i + \eps\Big)\right]\,. \label{eq:expected loss}
\end{align}
Here, $\eps \approx 10^{-12}$ is a floating-point safeguard for the logarithm, rather than a regularity assumption as the components of $\operatorname{softmax}$ are already strictly positive. Note that $\eps$ must be added outside the softmax, since adding a constant to every logit leaves the softmax unchanged. $q_{ref}(t; q(t_0))$ is a solution to the deterministic Lorenz 96 system \eqref{eq:det Lorenz 96} with initial condition $q(t_0)$, solved using high resolution simulation. We once again use a Quasi--Newton gradient method to solve the minimisation problem \eqref{eq:expected loss} where the pathwise gradients with respect to parameters are computed using rough symplectic methods in the class of \eqref{eq:RSPRK}. The empirical mean is computed using Monte Carlo simulation with variable sample size. The convergence results of the optimisation problem are shown in Figure \ref{fig:grad descent expected}.
\begin{figure}[!ht]
    \centering
    \begin{subfigure}[b]{0.32\textwidth}
        \centering
        \includegraphics[width=\textwidth]{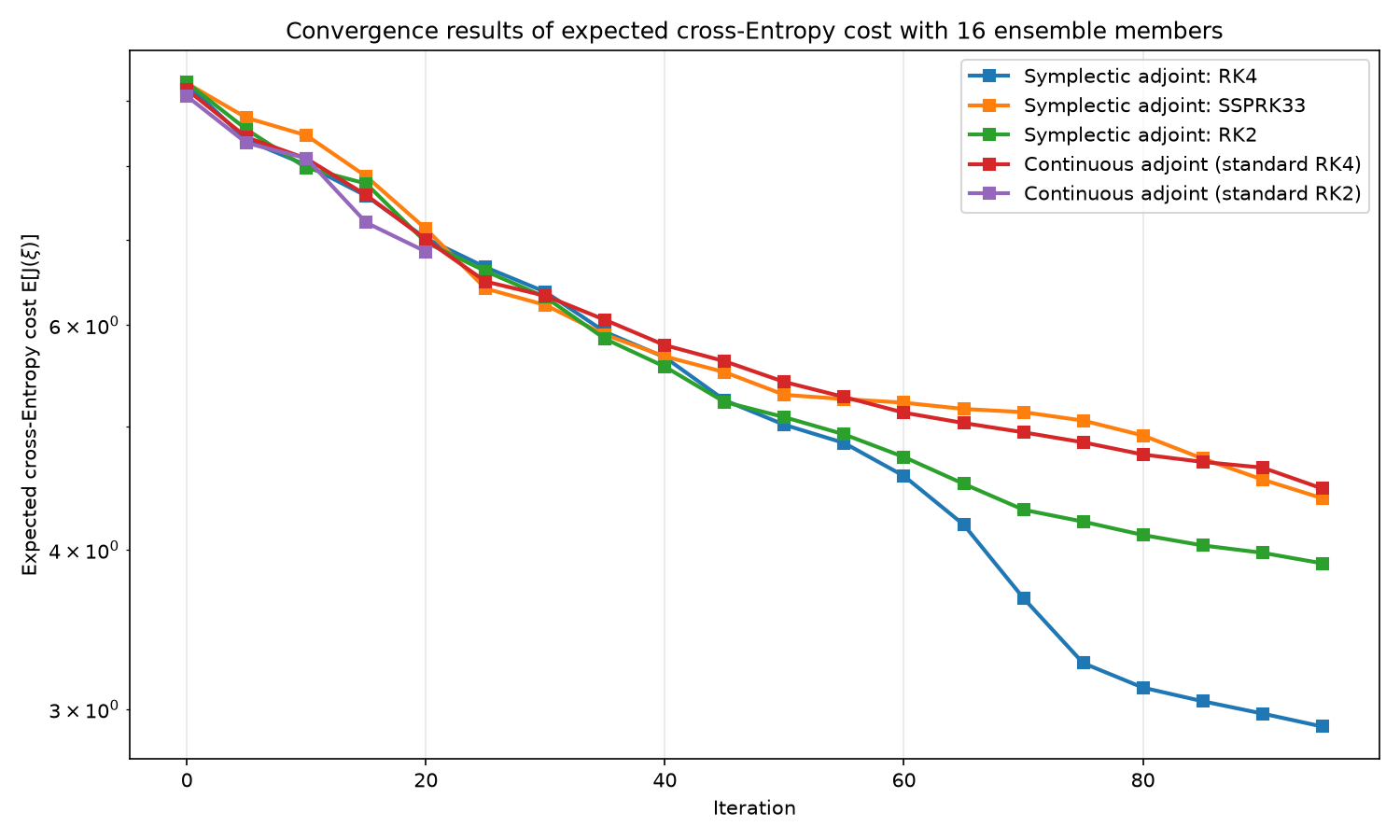}
    \end{subfigure}
    \begin{subfigure}[b]{0.32\textwidth}
        \centering
        \includegraphics[width=\textwidth]{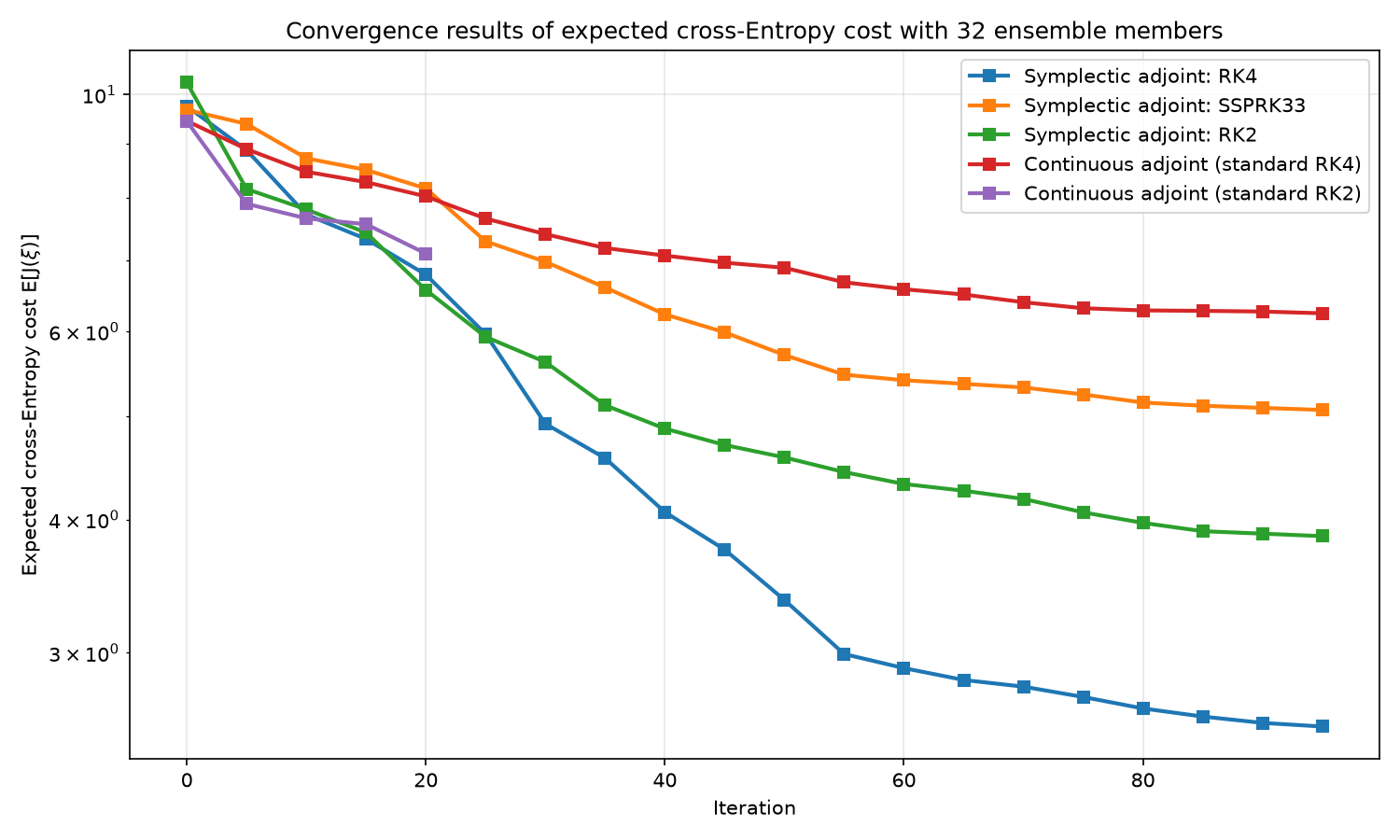}
    \end{subfigure}
    \begin{subfigure}[b]{0.32\textwidth}
        \centering
        \includegraphics[width=\textwidth]{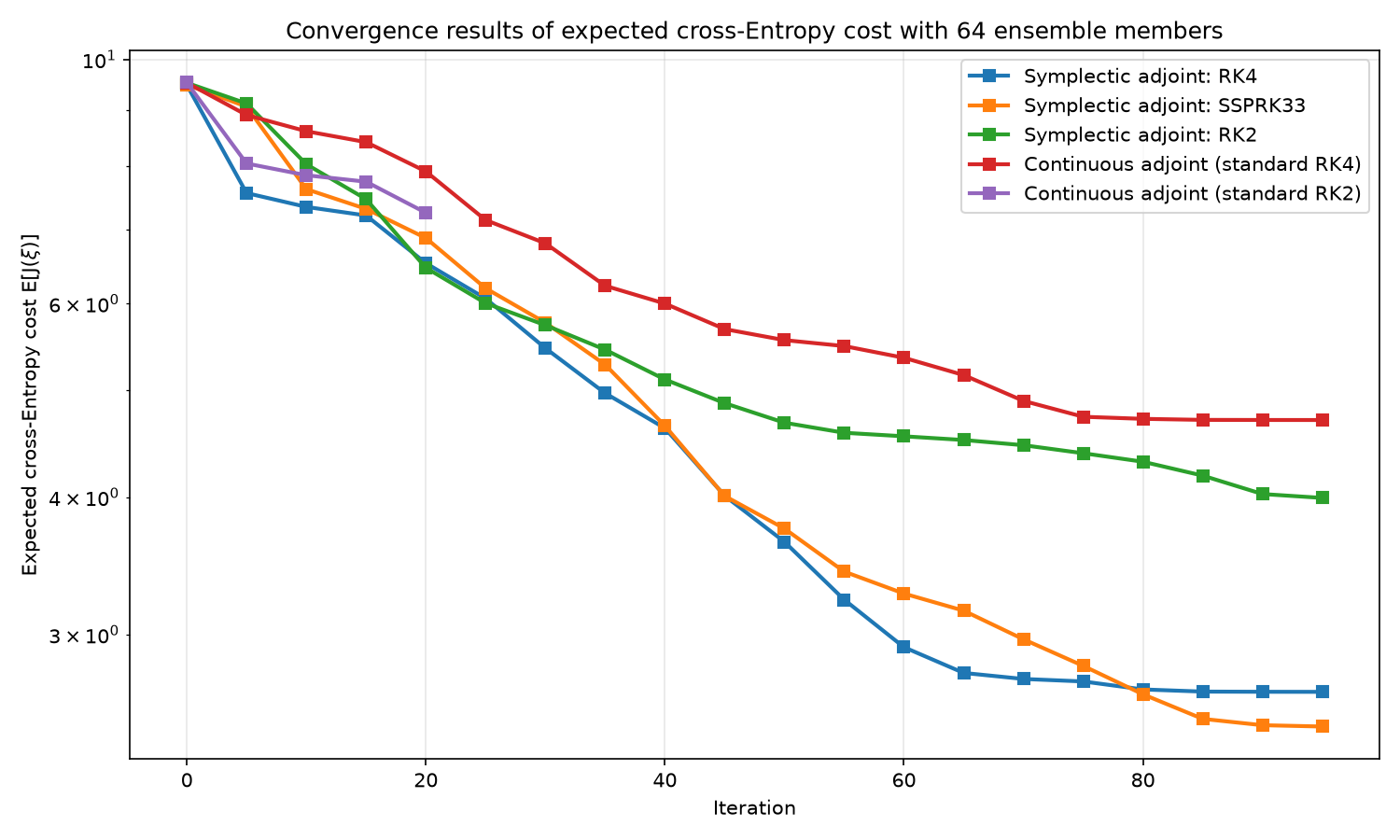}
    \end{subfigure}
    \caption{Convergence results of the Quasi--Newton gradient descent for optimisation problem \eqref{eq:expected loss} with different numbers of realisations of the driving fbm with Hurst parameter $\mathfrak{h} = 2/5$. Due to the empirical mean, the gradients are smoother than the pathwise counterparts such that the the stability of Quasi--Newton algorithm is not affected to the same extent, c.f. Figure \ref{fig:grad descent pathwise}.
    }
    \label{fig:grad descent expected}
\end{figure}

\section{Conclusion and future work}\label{sec:conclusion}
In this work, we developed a rough variational framework and used it to study rough Hamiltonian systems, rough adjoint systems, and their structure-preserving numerical discretisation. In Section \ref{subsec:rough VP}, we formulated a Type-II variational principle driven by a geometric rough path with split boundary conditions where initial conditions for the state variable and terminal conditions for the adjoint variable are prescribed. This class of boundary conditions is especially natural for sensitivity analysis and optimal control. Using the rough variational principle, we derived rough Hamilton's equations and demonstrated their conservation laws and the associated rough Hamilton--Jacobi equation. In Section \ref{sec:cts rough adjoints}, we considered rough adjoint systems arising as a particular class of rough Hamiltonian systems, and we established pathwise conservation and quasi-conservation laws for these systems. Subsequently, in Section \ref{sec:cts rough adjoints sensitivities}, we showed how adjoint gradients with respect to initial conditions and parameters can be obtained from rough variational principles. We also demonstrated the connection between adjoint sensitivities and the conservation laws of rough adjoint systems. These results extend standard deterministic conservation laws for Hamiltonian and adjoint systems, and they motivated us to consider structure preserving discretisations of rough Hamiltonian systems.

To this end, in Section \ref{subsec:rough galerkin}, we introduced a rough Galerkin discretisation of the Type-II variational principle, derived a rough Galerkin method for the rough Hamilton equations, and showed that it generates a symplectic flow and possesses discrete analogues of the conservation laws of the rough Hamilton equations. In Section \ref{subsec:rough RK}, we established the equivalence between the derived rough Galerkin methods and Rough Symplectic Partitioned Runge--Kutta (RSPRK) methods before considering convergence rates. Focusing on rough adjoint systems in Section \ref{subsec:rough RK adjoint}, we showed that the rough Galerkin methods inherit discrete analogues of the continuous conservation and quasi-conservation laws, which are crucial for accurate gradient computation. We further showed that the formation of variational equations and adjoint equations commutes with discretisation when the numerical scheme is taken as the tangent and cotangent lift of the numerical scheme for the state variable, respectively. This extends classical results found in e.g., \cite{Sanz-Serna2016} for deterministic adjoint systems to the setting of adjoint systems driven by geometric rough paths.

The numerical examples in Section \ref{sec:examples} illustrated the convergence properties of the developed RSPRK methods as well as the impact their discrete conservation laws in gradient computation and optimisation. In particular, in Section \ref{subsec:gradient comp}, we demonstrated that the developed rough symplectic methods preserve the adjoint conservation laws of rough adjoint systems to machine precision, which leads to more accurate adjoint gradients than non-symplectic methods. We showed that this difference becomes especially visible in optimisation problems for both pathwise and expected loss functions, where the usage of rough symplectic methods for adjoint gradients simultaneously improves gradient accuracy and convergence speed.

There are several promising directions for future work.
\begin{itemize}
    \item \textbf{Extension to rough-path-driven PDEs.} A natural next step is to extend the present finite dimensional framework to infinite dimensional settings by considering adjoint systems for rough PDEs in either the Hamiltonian or multisymplectic setting. One would expect that extending the framework for stochastic multisymplectic PDEs and their structure-preserving discretisations developed, for example, in \cite{HP2025} to the rough setting would naturally yield adjoint systems that preserve multisymplectic structures.
    \item \textbf{Higher-order and adaptive methods.} As remarked in Section \ref{subsec:rough RK}, the RSPRK methods suffer from an order barrier due to the exclusion of higher order signatures in \eqref{eq:discrete action}. It would be valuable to construct higher order rough variational integrators by developing higher order discretisations of the variational principle systematically and analysing their order conditions. In the same spirit, we aim to construct adaptive time stepping methods for RDEs that preserve the symplectic structure, following, for example, \cite{Duruisseaux2021}, as part of future work.
    \item \textbf{Non-geometric rough paths.} In this work, we focused on geometric rough paths. Extending the variational framework to systems driven by non-geometric rough paths could broaden the range of applications to adjoint systems driven by more general rough signals, as well as to data assimilation and stochastic optimal control problems where the rough drivers are typically non-geometric.
\end{itemize}

Overall, the results of this paper suggest that rough variational principles provide a natural language for pathwise adjoint analysis and geometric numerical integration for systems with irregular drivers. We believe that this perspective will motivate the development of pathwise adjoint gradient methods for stochastic and rough dynamical systems, as well as structure preserving algorithms for applications across a wide range of fields.

\subsection*{Data Availability Statement}
Python and open access Python packages were used to generate the numerical results and figures. Code used can be made available under reasonable request. 

\subsection*{Conflicts of Interest}
The authors report no conflicts of interest.

\subsection*{Acknowledgements.}
We wish to thank Darryl Holm for several thoughtful suggestions during the course of this work, which have improved or clarified the interpretation of its results.

RH is grateful for the support by the Office of Naval Research (ONR) grant award N00014-22-1-2082, Stochastic Parameterization of Ocean Turbulence for Observational Networks, where part of this work was done. The research of ML was supported in part by NSF under grants CCF-2112665, DMS-2307801, and by AFOSR under grant FA9550-23-1-0279.

\appendix
\section{Calculus of variations of rough path}
\begin{lemma} \label{lemma:foundamental}
    Let $\alpha \in (\frac{1}{3}, \frac{1}{2}]$, $\mathbf{Z} \in \mcal{C}^\alpha_g([a,b];\mathbb{R}^K)$, $\mathbf{Y} = (Y, Y') \in \mcal{D}^{2\alpha}_Z([a,b];\mcal{L}(\mathbb{R}^K, \mathbb{R}))$ and $\lambda \in \mcal{D}^{2\alpha}_Z([a,b];\mathbb{R})$. Assume that 
    \begin{align}
        \int_a^b \phi(t)\left( d \lambda(t) + Y(t)\,d \mathbf{Z}_t\right) = 0\,,
    \end{align}
    for all $\phi \in C^1([a,b];\mathbb{R})$ satisfying $\phi(a) = \phi_a$ and $\phi(b) = \phi_b$ for some constants $\phi_a, \phi_b \in \mathbb{R}$, then
    \begin{align}
        \lambda(b) - \lambda(a) = -\int_a^b Y(t)\,d\mathbf{Z}_t\,.
    \end{align}
\end{lemma}
\begin{proof}
    The proof is based on that given in \cite[Lemma B.4]{CHLN2022}. Fix $a < s < t < b$ with $n^{-1} < \min\{s-a,\, b-t\}$ and consider the sequence of Lipschitz functions $\{\phi^n\}_{n \in \mathbb{N}}$ defined by
    \begin{align*}
        \phi^n(r) := \begin{cases}
            \phi_a & r \in [a,s-n^{-1}]\\
            n(r-s)(1-\phi_a) + 1 & r \in [s-n^{-1}, s] \\
            1 & r \in [s, t] \\
            n(r-t)(\phi_b-1) + 1 &r \in [t, t + n^{-1}] \\
            \phi_b & r \in [t+n^{-1}, b]
        \end{cases}
    \end{align*}
    such that $|\phi^n|_{\infty}$ is finite; each $\phi^n$ is understood as a $C^1$ (indeed $C^\infty$) mollification of this piecewise-linear profile on an $\mcal{O}(n^{-1})$ neighbourhood of the two corners, so that $\phi^n$ genuinely lies in the class quantified over in the hypothesis. The weak derivative of $\phi^n(r)$ is given by
    \begin{align*}
        \dot{\phi}^n(r) =  \begin{cases}
            n(1-\phi_a) & r \in [s-n^{-1}, s] \\
            n(\phi_b-1) &r \in [t, t + n^{-1}] \\
            0 & \text{otherwise}
        \end{cases}
    \end{align*}
    such that $|\dot{\phi}^n|_\infty = \max(|n(1-\phi_a)|, |n(\phi_b-1)|)$. One can show that 
    \begin{align*}
        \lim_{n \rightarrow \infty}\int_a^b \lambda \dot{\phi}^n_r \,dr = \lambda(s)(1-\phi_a) + \lambda(t)(\phi_b - 1)
    \end{align*}
    such that through integration by parts,
    \begin{align}
        \lim_{n \rightarrow \infty}\int_a^b \phi^n(t)d \lambda(t) = \lambda(t) - \lambda(s) + \phi_b(\lambda(b) - \lambda(t)) - \phi_a(\lambda(a)- \lambda(s))\,.
    \end{align}
    We decompose the rough integral against $\mathbf{Z}$ as the following
    \begin{align*}
        \int_a^b \phi^n(r) Y(r)\,d\mathbf{Z}_r &= \int_a^{s-n^{-1}}\phi_aY(r)\,d\mathbf{Z}_r + \int_{s - n^{-1}}^s \phi^n(r) Y(r)\,d\mathbf{Z}_r + \int_s^t Y(r)\,d\mathbf{Z}_r \\
        & \qquad + \int_t^{t+n^{-1}} \phi^n(r) Y(r)\,d\mathbf{Z}_r + \int_{t + n^{-1}}^b\phi_b Y(r)\,d\mathbf{Z}_r\,.
    \end{align*}
    The second and fourth term in the right hand side of the above equality converge to zero in the limit as $n \rightarrow \infty$, at rate $\mcal O(n^{-\alpha})$: on an interval of length $n^{-1}$ the controlled-path remainder norm of $\phi^nY$ grows like $n^{2\alpha}$ while the interval contributes $n^{-3\alpha}$ to the sewing estimate. Combining the two limits, we have 
    \begin{align}
        \lambda(t) - \lambda(s) + \phi_b(\lambda(b) - \lambda(t)) - \phi_a(\lambda(a)- \lambda(s)) = -\left[\int_a^s\phi_aY(r)\,d\mathbf{Z}_r + \int_s^t Y(r)\,d\mathbf{Z}_r + \int_t^b\phi_b Y(r)\,d\mathbf{Z}_r\right]\,.
    \end{align}
    Extending the equality from $(s,t)\in [a,b]^2$ to $(s,t) = (a,b)$ via continuity of rough integrals yields the result.
\end{proof}
\begin{corollary}\label{foundamental coro}
    Assume that the conditions of Lemma \ref{lemma:foundamental} are satisfied. Specialising to the case where $\phi(a) = \phi_a = 0$ with $\phi_b$ unconstrained, then
    \begin{align*}
        \lambda(b) - \lambda(s) = -\int_s^b Y(r)\,d\mathbf{Z}_r\,,
    \end{align*}
    for all $s \in [a,b]$. Similarly, when $\phi(b) = \phi_b = 0$ with $\phi_a$ unconstrained, 
    \begin{align*}
        \lambda(t) - \lambda(a) = -\int_a^t Y(r)\,d\mathbf{Z}_r\,,
    \end{align*}
    for all $t \in [a,b]$.
\end{corollary}
\begin{proof}
Set $\phi_a = 0$ in the identity established in the proof of Lemma \ref{lemma:foundamental}. Both sides
are then affine in the remaining free constant $\phi_b$, and equating the coefficients of $\phi_b^0$ and
$\phi_b^1$ gives $\lambda(t)-\lambda(s) = -\int_s^t Y\,d\mathbf Z$ and
$\lambda(b)-\lambda(t) = -\int_t^b Y\,d\mathbf Z$ separately, for every $a \leq s \leq t \leq b$. The
first conclusion is the second of these with $t$ renamed. The case $\phi_b = 0$ is identical with the
roles of the endpoints exchanged.
\end{proof}

\bibliographystyle{alpha}
\bibliography{main.bib}

\end{document}